\documentclass[11pt]{amsart}
\usepackage{a4wide,pdfsync,hyperref}
\usepackage{amsmath,amssymb,esint}
\usepackage{enumerate}
\usepackage{color}

\numberwithin{equation}{section}

\allowdisplaybreaks

\newtheorem{theorem}{Theorem}[section]
\newtheorem{lemma}[theorem]{Lemma}

\newtheorem{proposition}[theorem]{Proposition}
\newtheorem{remark}[theorem]{Remark}

\newcommand{\beqq}{\begin{eqnarray}}
\newcommand{\enqq}{\end{eqnarray}}

\newcommand{\enn}{\end{equation}}
\newcommand{\bef}{\begin{proof}}
\newcommand{\enf}{\end{proof}}
\let\al=\alpha

\let\d=\delta
\let\e=\varepsilon

\let\la=\lambda

\let\f=\frac

\let\Om=\Omega

\let\tri=\triangle
\let\na=\nabla
\let\th=\theta
\let\pa=\partial

\def\Re{\mathcal {R}e}
\def\Im{\mathcal {I}m}

\def\dv{\mbox{div}}

\def\eqdef{\buildrel\hbox{\footnotesize def}\over =}

\newcommand{\beq}{\begin{equation}}
\newcommand{\eeq}{\end{equation}}
\newcommand{\ben}{\begin{eqnarray}}
\newcommand{\een}{\end{eqnarray}}
\newcommand{\beno}{\begin{eqnarray*}}
\newcommand{\eeno}{\end{eqnarray*}}

\begin{document}
\title[Zero-viscosity limit of CNS]{Zero-viscosity limit of the compressible Naiver-Stokes equations in the analytic setting}

\author[C. Wang]{Chao Wang}
\address{School of Mathematical Sciences\\ Peking University\\ Beijing 100871,China}
\email{wangchao@math.pku.edu.cn}

\author[Y. Wang]{Yuxi Wang}
\address{School of Mathematics, Sichuan University, Chengdu 610064, China}
\email{wangyuxi@scu.edu.cn}

\author[Z. Zhang]{Zhifei Zhang}
\address{School of Mathematical Sciences\\ Peking University\\ Beijing 100871,China}
\email{zfzhang@math.pku.edu.cn}

\begin{abstract}
In this paper, we study the zero-viscosity limit of the compressible Navier-Stokes equations in a half-space with non-slip boundary condition. We justify the Prandtl boundary layer expansion for the analytic data:  the compressible Navier-Stokes equations can be approximated by the compressible Euler equations away from the boundary, and by the compressible Prandtl equation near the boundary.
\end{abstract}
\maketitle
%\tableofcontents

\section{Introduction}
\subsection{Presentation the problem and related results.}
In this paper, we consider the zero-viscosity limit of the 2-D compressible Navier-Stokes equations in a half-space  with non-slip boundary condition
\begin{align}\label{eq: CNS}
\left\{
\begin{aligned}
&\pa_t \rho^\e+u^\e\pa_x \rho^\e+v^\e\pa_y \rho^\e+\rho^\e(\pa_x u^\e+\pa_yv^\e )=0,\\
&\rho^\e(\pa_t u^\e+u^\e\pa_x u^\e+v^\e\pa_y u^\e )-\e^2\nu\tri u^\e-\e^2(\nu+\sigma)\pa_x(\pa_x u^\e+\pa_y v^\e)+\pa_x P(\rho^\e) =0,\\
&\rho^\e(\pa_t v^\e+u^\e\pa_x v^\e+v^\e\pa_y v^\e) -\e^2\nu\tri v^\e-\e^2(\nu+\sigma)\pa_y (\pa_x u^\e+\pa_y v^\e)+\pa_y P(\rho^\e) =0,\\
&(u^\e,v^\e)|_{y=0}=0,\\
&(\rho^\e, u^\e, v^\e)|_{t=0}=(\rho_0,u_0,v_0),
\end{aligned}
\right.
\end{align}
where $ t\geq 0, (x,y)\in\mathbb{R}_{+}^2,$ $\nu>0,~\nu+\sigma\geq0$. Here $\e>0$ is viscosity and the pressures $P(\rho)=a\rho^\gamma$ with $\gamma>1$. 

Let $\e=0$, the system \eqref{eq: CNS} reduces to the compressible Euler equations
\begin{align}\label{equ:Euler}
\left\{
\begin{aligned}
&\pa_t \rho^e+u^e\pa_x \rho^e+v^e\pa_y \rho^e+\rho^e(\pa_x u^e+\pa_y v^e )=0,\\
&\rho^e(\pa_t u^e+u^e\pa_x u^e+v^e\pa_y u^e ) +\pa_x P(\rho^e) =0,\\
&\rho^e(\pa_t v^e+u^e\pa_x v^e+v^e\pa_y v^e)  +\pa_y P(\rho^e) =0,\\
&v^e|_{y=0}=0,\\
&(\rho^e, u^e, v^e)|_{t=0}=(\rho_0,u_0,v_0).
\end{aligned}
\right.
\end{align}
Due to the mismatch of the boundary condition between the system \eqref{eq: CNS} and \eqref{equ:Euler}, the inviscid limit will give rise to the boundary layer. 
In this case, we need to introduce a Prandtl boundary layer ansatz for the solution $(u^\e,v^\e)$ of the 2-D compressible Navier-Stokes equations: 
\ben\label{formal expan}
 \left\{
 \begin{array}{l}
 \rho^\e(t,x,y)=\rho^e(t, x, y)+O(\e),\\
 u^{\e}(t,x,y) =u^{e}(t,x,y)+ u^{p}(t,x,\f{y}{\e})+O(\e),\\
 v^{\e}(t,x,y)= v^{e}(t,x,y)+\e v^{p}(t,x,\f{y}{\e})+O(\e),
 \end{array}\right.
 \een
where $(\rho^e, u^e, v^e)$ solves the compressible Euler equations \eqref{equ:Euler},  while $(u^p,v^p)$  satisfies the Prandtl type equation (see \eqref{eq: (u_p^0, v_p^1)-0}).
 
For the incompressible Naviert-Stokes equations, the validity of Prandtl's ansatz is due to Sammartino and Caflisch \cite{SC2}, who justified the inviscid limit for the data in the analytic space via the abstract Cauchy-Kowaleskaya theorem. Recently, we developed an energy method to prove Prandtl's ansatz in the analytic setting \cite{WWZ}.  More results can be referred to \cite{Mae, NN, KVW, FTZ, WW} and references therein, where all of these results require the analyticity of the data at least near the boundary. Let us mention recent progress on the stability of the boundary layer theory \cite{GGT, GVMM1, GVMM2, CWZ}.

For the compressible Navier-Stokes equations with slip boundary conditions, there will appear a weak boundary layer. Thus, one can justify the inviscid limit in Sobolev spaces, see \cite{WangY, WXY,Paddick, IP, IS, WWi, FN} and the references therein for more related results.  Xin and Yanagisawa \cite{XY} studied the vanishing viscosity limit for the linearized compressible Navier-Stokes system with non-slip boundary conditions. Rousset \cite{Rousset} studied the 1-D case and the paper \cite{GMWZ}  treated the noncharacteristic case. Liu and Wang \cite{LW} obtained the stability of boundary layers for the nonisentropic compressible circularly symmetric 2D flow. Sueur extended Kato's result \cite{Kato} to the compressible case \cite{Sueur}, see \cite{WZ1, WZ2} for more general results on Kato's criterion. We also refer to some recent results on the linear instability analysis on the compressible Navier-Stokes equations  by Yang and Zhang \cite{YZ}, and Masmoudi, Wang, Wu and Zhang \cite{MWWZ}.

To our best knowledge, even for the analytic data, the validity of Prandtl boundary layer expansion \eqref{formal expan}  remains open. The aim of this paper is to solve this open problem and justify the boundary layer ansatz in the analytic space.  We consider the initial data with the form 
\begin{align}\label{initial: 1}
\rho^\e(0,x,y)=\rho_0(x,y),\quad u^\e(0,x,y)=u_0(x,y),\quad v^\e(0,x,y)=v_0(x,y),
\end{align}
which satisfies
\begin{align}\label{initial: 2}
&u_0(x,0)=v_0(x,0)=0,\\
\label{initial: 3}
&0<4c_0\leq \rho_0(x,y)\leq c_0^{-1}<+\infty.
\end{align}
Moreover,  $\pa_t^k(\rho^\e, u^\e, v^\e)(0,x,y)$ for $|k|\le 20$ are determined by using the compressible Navier-Stokes equations \eqref{eq: CNS} and satisfy the boundary compatibility conditions.  Furthermore,  we assume that initial data is analytic with the bound
\begin{align}\label{initial: 5}
\sum_{|\beta|=0}^{20}\|\pa^\beta(\rho^\e, u^\e, v^\e)(0,\cdot,\cdot) \|_{L^2_{8\mu_0}}^2:=M_0^2 <+\infty,
\end{align}
where $\pa^\beta=\pa_x^{\beta_1}\pa_y^{\beta_2}\pa_t^{\beta_0}$  with $\beta_1+\beta_2+\beta_0=|\beta|$, and the norm $\|\cdot\|_{L^2_\mu}$ is defined by
\begin{align}\label{def: f_L^2}
\|f\|^2_{ L^2_{\mu}}=  \sup_{0\le\th, \th'< \mu}\int_{\pa D_{\th'}} \int_{\pa\Om_\th} |f(t,x,y)|^2 dydx,
\end{align}
where
\begin{align*}
 D_\mu=\big\{x\in\mathbb{C}: |\Im x|<\mu \big\},\quad\Om_\mu=\Big\{y\in\mathbb{C}:  \Re y>0,~|\Im y|< \min\{\mu \Re y, \mu \}\Big\}
\end{align*}
for $0<\mu< 8\mu_0$ with $\mu_0\in(0,\f1{100}]$  a small fixed constant. We introduce the compressible Prandtl-type equation
\begin{align}\label{eq: (u_p^0, v_p^1)-0}
\left\{
 \begin{aligned}
&\pa_t u^p+(\overline{u^e}+u^p)\pa_x u^p+\Big(-\f{\int_0^{+\infty}\pa_x(\overline{\rho^e} u^p)dy}{\overline{\rho^e}}+v^p+z\overline{\pa_y v^e}\Big)\pa_z u^p-\f{1}{\overline{\rho^e}}\pa_z^2 u^p=0,\\
&\pa_x(\overline{\rho^e}u^p )+\pa_z(\overline{\rho^e} v^p)=0,\\
&u^p|_{z=0}=-\overline{u^e},\quad v^p|_{z=\infty}=0,\\
&u^p|_{t=0}=0,
\end{aligned}
 \right.
  \end{align}
where we denote by $\overline{f}=f(t,x,0)$ for convenience. \smallskip

Now we state the main result of this paper.

\begin{theorem}\label{thm:main}
Assume that the initial data satisfies \eqref{initial: 1}-\eqref{initial: 5}. 
Then there exist $T>0$ and $C>0$ independent of $\e$ such that there exists a unique analytic solution $\big(\rho^\e,u^{\e},v^{\e}\big)$ of the Navier-Stokes equations (\ref{eq: CNS})
in $[0,T]$, such that for any $t\in [0,T]$, there hold
\beno
&& \big \| \rho^\e-\rho^e \big\|_{L^2_{x,y}\cap L^\infty_{x,y}}\le C\e,\\
&&\big\|u^{\e}(t,x,y)-\big(u^{e}(t,x,y)+u^{p} (t,x, \f y\e)\big)\big\|_{L^2_{x,y}\cap L^\infty_{x,y}}\le C\e,\\
&&\big\|v^{\e}(t,x,y)-\big(v^{e}(t,x,y)+\e v^{p}(t,x, \f y\e)\big)\big\|_{L^2_{x,y}\cap L^\infty_{x,y}}\le C\e,
\eeno
where $(\rho^e, u^e, v^e)$ solves \eqref{equ:Euler} and $(u^p, v^p)$ solves \eqref{eq: (u_p^0, v_p^1)-0} .
\end{theorem}

 \begin{remark}
 To simplify the presentation, we take $\gamma=2$ and $a=\f12$ in the proof of Theorem \ref{thm:main}. 
 Since the density is bounded by a positive constant from below, the case of general pressure could be proved similarly.
  \end{remark}

  \subsection{ Sketch of the proof}

The first step is to construct approximate solution $(\rho^a, u^a, v^a)$ (see \eqref{eq: CNS-a}) by using the asymptotic matching expansion technique, and then derive the error system \eqref{eq: Error-(u,v,rho)} of the error functions:
\begin{align*}
u^R=u^\e-u^a, \quad v^R=v^\e-v^a,\quad \rho^R=\rho^\e-\rho^a.
\end{align*}
The main job is to obtain the uniform estimates (in $\e$) for the error:
\begin{align}\label{est: decay rate}
\|(u^R, v^R, \rho^R)\|_{L^2_{x,y}\cap L^\infty_{x,y}}\leq C\e.
\end{align}
To make our proof clearer, we just look at the following toy model, which has already captured essential difficulties:
\begin{align}\label{eq: Error-Toy}
\left\{
\begin{aligned}
&\pa_t \rho^R+\pa_x u^R+ \pa_y v^R = \cdots,\\
&\pa_t u^R+v^R\pa_{y}u^p-\e^2\tri u^R=\cdots,\\
&\pa_t v^R+ \pa_y \rho^R-\e^2\tri v^R=\cdots.
%&(u^R,v^R)|_{y=0}=0,\\
%&(u^R,v^R,\rho^R)|_{t=0}=0,
\end{aligned}
\right.
\end{align}
Here $u^p=u^p(\f y\e)= u^p(z)$ is the Prandtl part.

 To obtain the uniform estimates, the worst term is  $v^R\pa_{y}u^p$ in the second equation of \eqref{eq: Error-Toy}, which will give rise to a bad factor $\f1\e$. More precisely, we have
\beno
v^R\pa_{y}(u^p( \f{y}{\e})) =\f{v^R}{y}\cdot  (z\pa_z)u^p(\f{y}{\e})\sim\pa_y v^R,
\eeno
by using $v^R|_{y=0}=0$ and the Hardy inequality, which leads to one derivative loss in $y$. As a result, it is natural to work in the analytical setting which will help us recover one derivative. However, if we directly take $\pa_y$ to the error equation \eqref{eq: Error-Toy}, the Prandtl part will give rise to more bad factors $\f{1}\e$. For the incompressible case, we can use the divergence free condition to change one derivative loss in $y$ into one derivative loss in $x$. For the compressible case,  we  use the density equation (the first equation in \eqref{eq: Error-Toy}) to change the derivative loss
 \beno
- \pa_y v^R= \pa_t \rho^R+\pa_x u^R.
 \eeno
In this way, the term $\pa_t \rho^R$ will lose one derivative in $t$. If we use the analyticity in $t$, we should add infinitely many compatibility conditions on initial data. This kind of initial data will not be easy to find. Thus, we have to find new ways.\smallskip
 
\underline{$\bullet$ Deal with term $v^R\pa_y u^p$. } 
 In order to overcome the loss of time derivative,  the key idea in this paper is as follows. Define $\mathcal{L}(f)=\f1y\int_0^y f dy',$ and the trouble term $v^R\pa_y u^p$ behaves like $\mathcal{L}(\pa_y v^R).$  Using the density equation in $\eqref{eq: Error-Toy}$, we have
 \begin{align*}
\int_0^t \int_{\mathbb{R}_{+}^2} v^R\pa_y u^p~ u^R dxdyds\sim& \int_0^t \int_{\mathbb{R}_{+}^2} \mathcal{L}(\pa_y v^R)~ u^R dxdyds\\
\sim&-\int_0^t \int_{\mathbb{R}_{+}^2} \mathcal{L}(\pa_s\rho^R)~ u^R dxdyds-\int_0^t \int_{\mathbb{R}_{+}^2} \mathcal{L}(\pa_x u^R)~ u^R dxdyds.
 \end{align*} 
Integrating by parts in time and using the equation of $u^R$ in $\eqref{eq: Error-Toy}$,  we obtain
\begin{align*}
-\int_0^t \int_{\mathbb{R}_{+}^2} \mathcal{L}(\pa_s\rho^R)~ u^R dxdyds\sim&
\int_0^t \int_{\mathbb{R}_{+}^2} \mathcal{L}(\rho^R)~ \pa_su^R dxdyds+\cdots\\
\sim&-\int_0^t \int_{\mathbb{R}_{+}^2}\mathcal{L}(\rho^R)~ v^R\pa_y u^pdxdyds+\cdots\\
\sim&-\int_0^t \int_{\mathbb{R}_{+}^2}\mathcal{L}(\rho^R)~ \mathcal{L}(\pa_y v^R)dxdyds+\cdots.
\end{align*}
Using  the equation of density in $\eqref{eq: Error-Toy}$ again, we get
{\small
\begin{align*}
-\int_0^t \int_{\mathbb{R}_{+}^2}\mathcal{L}(\rho^R)~ \mathcal{L}(\pa_y v^R)dxdyds
\sim& \int_0^t \int_{\mathbb{R}_{+}^2}\mathcal{L}(\rho^R)~ \mathcal{L}(\pa_s\rho^R)dxdyds+\int_0^t \int_{\mathbb{R}_{+}^2}\mathcal{L}(\rho^R)~ \mathcal{L}(\pa_x u^R)dxdyds\\
\sim&\|\mathcal{L}(\rho^R) \|_{L^2}^2\Big|_{s=0}^{s=t}+\int_0^t \int_{\mathbb{R}_{+}^2}\mathcal{L}(\rho^R)~ \mathcal{L}(\pa_x u^R)dxdyds.
\end{align*}
}
By the above process and $\|\mathcal{L}(\rho^R) \|_{L^2}\sim \|\rho^R \|_{L^2}$(by the Hardy inequality), we find 
\begin{align*}
\int_0^t \int_{\mathbb{R}_{+}^2} v^R\pa_y u^p~ u^R dxdyds\sim&
\|\rho^R \|_{L^2}^2\Big|_{s=0}^{s=t}+\mbox{terms losing $\pa_x$}.
\end{align*}
In this way, we change the derivative loss in $y$ to the derivative loss in $x$. See Lemma \ref{lem: est: A} for more details. \smallskip

\underline{$\bullet$ Energy estimate of $u^R.$} Due to one derivative loss, we work in the analytic space $X_{\mu}^k$(see section 2), where $\mu$ is the analytic radius depending on $t$. We use the following toy model to describe the main idea of the proof:
\begin{align*}
\pa_t u-\e^2\tri u=F(\pa_x u, Z_2 u),\quad u|_{y=0}=0,
\end{align*}
where $Z_2=\varphi\pa_y$ is the conormal derivative in $y$ and $F$ is a bounded operator from $X^k_\mu$ to $X^k_\mu$. By the energy estimate and enlarging the analytic radius $\mu$, we have
 \begin{align}\label{est: Toy 1}
 \|u\|_{X^k_\mu}^2+\|\e\na u^R\|_{\widetilde{L}^2(0, t;X^k_{\mu})}^2
 \leq& \|u_0\|_{X^k_\mu}^2+\int_0^t\|F(\pa_x u, Z_2 u)\|_{X^k_\mu} \|u\|_{X^k_\mu} ds\\
  \nonumber
 \lesssim&\|u_0\|_{X^k_\mu}^2+\int_0^t\|(\pa_x u, Z_2 u)\|_{X^k_\mu} \|u\|_{X^k_\mu} ds\\
  \nonumber
 \lesssim& \|u_0\|_{X^k_\mu}^2+\int_0^t\f{1}{\mu'-\mu}\|u\|_{X^k_{\mu'}} \|u\|_{X^k_\mu} ds.
 \end{align}
 By taking $\mu'=\mu+\f12 h(s,\mu)>\mu,\, h(s,\mu)=\mu_0-\mu-\la s> 0,$ which implies $h(s,\mu')=\f12 h(s,\mu),$ we have
 \begin{align}\label{est: Toy 2}
 \int_0^t\f{1}{\mu'-\mu}\|u\|_{X^k_{\mu'}} \|u\|_{X^k_\mu} ds\leq&C\int_0^t\f{1}{h(s,\mu')}\|u\|_{X^k_{\mu'}} \|u\|_{X^k_\mu} ds\\
 \nonumber
 \leq&C\sup_{s\in[0,t]}\sup_{\mu< \mu_0-\la s}\Big(h^{\f{\eta}{2}}(s,\mu')\|u\|_{X^{k}_{\mu'}}\Big)\cdot\Big(h^{\f{\eta}{2}}(s,\mu)\| u\|_{X^{k}_{\mu}}\Big) \\
  \nonumber
  &\qquad\times\int_0^t h^{-1-\f{\eta}2}(s,\mu) \cdot h^{-\f{\eta}{2}} (s,\mu')ds\\
   \nonumber
 \leq& C \la^{-1}h^{-\eta}(t,\mu) \sup_{s\in[0,t]}\sup_{\mu< \mu_0-\la s}\Big(h^{\eta}(s,\mu)\|u\|_{X^{k}_{\mu}}^2\Big),
 \end{align}
by $\eta\in(0,1)$ small and
 \begin{align}
\mu< \mu_0-\la s\quad \mbox{if and only if}\quad \mu'<\mu_0-\la s.
\end{align}
 Inserting \eqref{est: Toy 2} into \eqref{est: Toy 1} and multiplying both sides by $h^\eta(t,\mu)$, we arrive at
 \begin{align*}
 h^\eta(t,\mu)\|u\|_{X^k_\mu}^2+h^\eta(t,\mu)\|\e\na u^R\|_{\widetilde{L}^2(0, t;X^k_{\mu})}^2\leq  \|u_0\|_{X^k_\mu}^2+C \la^{-1} \sup_{s\in[0,t]}\sup_{\mu< \mu_0-\la s}\Big(h^{\eta}(s,\mu)\|u\|_{X^{k}_{\mu}}^2\Big).
 \end{align*}
Taking $\sup_{t\in[0,T]}\sup_{\mu< \mu_0-\la t}$ and choosing $\la$ large enough, the second term on the right hand side can be absorbed by the left one.\smallskip
 
 \underline{$\bullet$ Energy estimate of $(\rho^R, v^R).$} The main difference between an incompressible case and a compressible case is the pressure. For the compressible case, we need to use the cancellation structure between $\pa_y \rho^R$ in $\eqref{eq: Error-Toy}_1$ and $\pa_y v^R$ in $\eqref{eq: Error-Toy}_3$. 
 More precisely, we need to deal with the following term
\begin{align*}
 \big\langle \pa_y \rho^R, v^R \big\rangle_{X^{k}_{\mu,t}}+ \big\langle \pa_y v^R, \rho^R \big\rangle_{X^{k}_{\mu,t}}
 \sim&\sum_{|\al|=0}^{k} \big\langle [Z^{\al},\pa_y] v^R, Z^\al \rho^R\big\rangle_{X^{0}_{\mu,t}}\\
 \leq& C\int_0^t \|\pa_y v^R\|_{X^{k-1}_\mu} \|\rho^R\|_{X^{k}_\mu}ds.
\end{align*}
The index $k-1$ of $\pa_y v^R$ in the last line above comes from the commutator $ [Z^{\al},\pa_y] v^R\sim \varphi' |\al|Z^{\al-e_2}\pa_y v^R.$ To control $\pa_y v^R$, we need to use the equation of $\rho^R$ in $\eqref{eq: Error-Toy}$, which gives
\beno
\|\pa_y v^R\|_{X^{k-1}_\mu}  \leq \|\pa_t \rho^R\|_{X^{k-1}_\mu} +\|\pa_x u^R\|_{X^{k-1}_\mu}.
\eeno

 Let us point out that the analytic norm introduced in \cite{WWZ} involves an infinite order $Z_2$ derivative. 
 Thus, for any $m\in\mathbb{N}_+$,  we need to handle the following commutator 
 \begin{align*}
 \int_{\mathbb{R}_{+}^2} [Z_2^m,\pa_y]\rho^R Z_2^m v^Rdxdy\sim &\varphi' \int_{\mathbb{R}_{+}^2} mZ_2^{m-1}\pa_y\rho^R Z_2^m v^Rdxdy\\
  \sim&\varphi'\int_{\mathbb{R}_{+}^2} mZ_2^{m}\rho^R Z_2^{m-1} \pa_yv^Rdxdy.
 \end{align*}
In the incompressible case, we may write $\pa_y v^R=-\pa_x u^R$. While, in the compressible case, the control of $\pa_y v^R$ will involve the infinite order derivatives in $t$ by using the density equation. To avoid this,  
{\bf a key point of this paper} is to extend the system \eqref{eq: Error-Toy} into a complex domain in $y$ variable so that it is enough to estimate the finite order derivatives in $t$ and $y$.\medskip

 \medskip

 \subsection{Organization and notations}
 This paper is organized as follows. In section 2, we introduce the analytic functional space and product estimates used in this paper.  In section 3, we construct approximate solutions and derive the error system. In section 4, we give the analytic-type estimates of the velocity and density. In section 5, we present the normal derivative estimates of the velocity and density.  In section 6, we present the proof of our main theorem. 
 
 \smallskip

Throughout this paper, $C_0$ is a constant independent of $\e, ~\d$ and $A$, which may depend on the initial data.  $C$ is a  constant independent of $\e$ and $A$, and may depend on  $\d.$
  
 \section{Analytic functional spaces and some technical lemmas.}

\subsection{Analytic functional spaces}
 
 We first introduce three kinds of  complex domains, which are used in the compressible Navier-Stokes equations, the compressible Euler equations and the Prandtl type equation respectively:
 the pencil-like complex domain
\begin{align*}
&\Om_\mu=\Big\{y\in\mathbb{C}:  \Re y>0,~|\Im y|< \min\{\mu \Re y, \mu \}\Big\}, 
\end{align*}
the strip complex domain
\begin{align*}
 D_\mu=\Big\{x\in\mathbb{C}: |\Im x|<\mu \Big\},
\end{align*}
and the cone-shaped complex domain
\begin{align*}
\widetilde{\Om}_\mu=\Big\{z\in\mathbb{C}: \Re z>0,~|\Im z|< \mu \Re z \Big\}
\end{align*}
for some small constant $\mu>0$. It is easy to see that  $\widetilde{\Om}_\mu\supset \Om_\mu$  .

%In the first place, we extend the equations of $(\rho^\e, u^\e, v^\e)$ in \eqref{eq: CNS} and the equation of $(\rho^a, u^a, v^a)$ in \eqref{eq: CNS-a} into complex domain. The process of extension will be given in the last section. More precisely, $(\rho^R, u^R, v^R)$ is holomorphic and defined on $y\in \Om_\mu$, $(\rho^a_e, u^a_e, v^a_e)$ is holomorphic and defined on $y\in \widetilde{\Om}_\mu$ and $(\rho^a_p, u^a_p, v^a_p)$ is holomorphic and defined on $z\in \widetilde{\widetilde{\Om}}_\mu$ with $z=\f{y}{\e}$. 
 
\medskip

%First, we give the definition of the analytical norm on $x$.

For any $(x,y)\in D_\mu\times\Om_\mu$, there exist $\th,~\th'\in[0,\mu)$ such that $x\in \pa D_{\th'}$ and $y\in\pa\Om_\th$. Thus, for the function $f(t,x,y)$ with $(x,y)\in D_\mu\times\Om_\mu$, we define
\begin{align}\label{def: f_L^2}
&\|f\|^2_{ L^2_{\mu}}=  \sup_{0\le\th, \th'< \mu}\int_{\pa D_{\th'}} \int_{\pa\Om_\th} |f(t,x,y)|^2 dy d x,\\
\label{def: f_L^infty}
&\|f\|_{ L^\infty_{\mu}}=  \sup_{0\le\th, \th'< \mu}\sup_{\substack{x\in\pa D_{\th'}\\ y\in \pa\Om_\th}} |f(t,x,y)| .
%&\label{def: f_L^infty}
%\|f\|_{ L^\infty_{\mu,y}(L^2_{\mu,x})}^2=  \sup_{0\le\th,\th'< \mu}~\sup_{y\in \pa\Om_\th}\int_{\pa D_{\th'}}|f(t,x,y)|^2dx.
\end{align}
We always take $\mu< \mu_0-\la t$ in this paper.

\smallskip

Now we introduce the analytic norm:
\begin{align}\label{def: f_X^k}
\|f\|_{X^k_\mu}^2=\sum_{|\al|=0}^k\|Z^\al f\|_{ L^2_{\mu}}^2=\sum_{|\al|=0}^k\sup_{0\le\th,\th'< \mu}\int_{\pa D_{\th'}} \int_{\pa\Om_\th}|Z^\al f(t,x,y)|^2 dy d x,
\end{align}
where $Z^\al=Z_0^{\al_0} Z_1^{\al_1} Z_2^{\al_2} $ for $Z_1=\d\pa_x,~Z_2=\d\varphi(y)\pa_y,~Z_0=\d\pa_t$ with $|\al|=\al_1+\al_2+\al_0,~\varphi(y)=\f{ y}{1+y}$ for some small constant $\d$ determined later.

We also define the inner product $\langle \cdot,\cdot \rangle_{X^k_{\mu}}$ and $\langle  \cdot,\cdot \rangle_{X^{k}_{\mu,t}} $ as follows
\begin{align*}
 \big\langle f,g \big \rangle_{X^k_{\mu}}=&\sum_{|\al|=0}^k \big\langle Z^\al f,Z^\al g  \big\rangle_{L^2_\mu}\\
=&\sum_{|\al|=0}^k\sup_{0\leq \th,\th'< \mu} \int_{\pa D_{\th'}} \int_{\pa\Om_\th}Z^\al f(t,x,y) \overline{Z^\al g(t,x,y)}  dy  d x,
\end{align*}
and
\beno
 \big\langle f,  g  \big\rangle_{X^{k}_{\mu,t}}&=& \sum_{|\al|=0}^k \big\langle Z^\al f,Z^\al g  \big\rangle_{X^{0}_{\mu,t}} \\
&=&\sum_{|\al|=0}^k  \sup_{0\le\th,\th'<\mu} \int_0^t\int_{\pa D_{\th'}} \int_{\pa\Om_\th}Z^{\al}f(s, x, y) \overline{Z^{\al} g(s, x, y)}~ dy dx ds.
\eeno
We also introduce
\beno
\|f\|_{\widetilde{L}^p(0,t; X^k_\mu)}&=&\sum_{|\al|=0}^k\sup_{0\leq \th,\th'<\mu}\Big(\int_0^t\|Z^\al f\|_{L^2_\mu(\pa D_{\th'}\times\pa\Om_\th)}^p ds\Big)^{1/p}\\
&=& \sum_{|\al|=0}^k\sup_{0\leq \th,\th' <\mu}\Big(\int_0^t \Big(\int_{\pa\Om_\th}\int_{\pa D_{\th'}}|Z^{\al}f|^2 dx dy\Big)^{\f{p}{2}}ds\Big)^{\f{1}{p}}.
\eeno

% \begin{remark}\label{rem: norm}
% We emphasize here that $\sup_{0\leq\th,\th'<\mu}$ and $\int_0^t$ can't change its order. More precisely, $\mu$ is restricted by $\mu< \mu_0-\la t$ when $\sup_{0\leq\th,\th'<\mu}\int_0^tf(s)ds$ is appeared and $\mu< \mu_0-\la s$ when $\int_0^t\sup_{0\leq\th,\th'<\mu}f(s)ds$ is appeared.
% \end{remark}
 
  It is easy to check by using the fact $\mu<\mu_0-\la t\leq \mu_0-\la s$ for $s\leq t$ that 
\begin{align}\label{inequ: norm 2}
 | \big\langle f, g  \big\rangle_{X^{k}_{\mu,t}}|\leq& \int_0^t\|f(s)\|_{X^k_\mu}\|g(s)\|_{X^k_\mu}ds.
\end{align}

To deal with the Euler part and Prandtl part, we introduce the norms $Y_\mu^k$ and $W^k_\mu$. For the Euler part, we define 
 \begin{align}\label{def: Y^k}
\|f\|_{Y^k_\mu}^2=&\sum_{|\beta|=0}^{k}\|\widetilde{\pa}^{\beta}f\|_{L^2_\mu(D_\mu\times\Om_\mu)}^2,  
\end{align}
where $\widetilde{\pa}^\beta=\pa_x^{\beta_1}(\varphi\pa_y)^{\beta_2}\pa_y^{\beta_3}\pa_t^{\beta_0}$  with $\beta_1+\beta_2+\beta_3+\beta_0=|\beta|$. For convenience, we denote $\|f\|_{\dot{Y}^k_\mu}$ for $\beta_3=0.$
For any function $f(t,x,y)$ with $(x,y)\in D_\mu\times\Om_\mu,$ it holds that $\|f\|_{X^k_\mu}, \|f\|_{\dot{Y}^k_\mu}\leq C\|f\|_{Y^k_\mu}.$

For the Prandtl part, we define 
 \begin{align}\label{def: W^k}
 \|f\|_{W^k_\mu}^2=\sum_{|\al|=0}^k\|e^{\phi}\widetilde{Z}^\al f\|_{L^2_\mu(D_\mu\times\widetilde{\Om}_\mu)}^2, 
\end{align}
where $\widetilde{Z}^\al=(\kappa\pa_x)^{\al_1}(\kappa z\pa_z)^{\al_2}(\kappa\pa_t)^{\al_0}$ with $\al_1+\al_2+\al_0=|\al|$ and $\phi(t,z)=e^{(2-\la_P t)|z|^2}.$ Here $\kappa$ is a small given constant satisfying $\d^\f12\ll \kappa\in(0, \f1{100}]$.

 \subsection{Some useful lemmas.}
We first give some useful results on analytic functions, then we present the product estimates in the analytic space.\smallskip

The first result is the integration by parts in the complex domain.

\begin{lemma}\label{lem: integration by parts}
For any analytic functions $f(y)$ and $g(y)$ defined on $\Om_\mu$, it holds that
\begin{align*}
\int_{\mathcal{C}}f(y) g'(y)dy=-\int_{\mathcal{C}}f'(y) g(y)dy+f(y_0) g(y_0)-f(0)g(0),
\end{align*}
where $\mathcal{C}$ is a simple curve in $\Om_\mu$ which connects $0$ and $y_0\in \pa\Om_\th$ for some $\th\in[0,\mu)$.
\end{lemma}

  The following Hardy inequality will be used frequently in this paper. 

\begin{lemma}\label{lem: Hardy inequality}
Let $y\in \pa\Om_\th$ and $\mathcal{C}$ be a simple curve in $\pa\Om_\th$ which connects $0$ and $y$. For any analytic function $f$ in $\Om_\mu,$ it holds that
\begin{align*}
\Big\|\f{1}{y}\int_{\mathcal{C}} f(z)dz\Big\|_{L^2(\pa\Om_\th)}\leq C_0\|f\|_{L^2(\pa\Om_\th)}.
\end{align*}
\end{lemma}
 
 \begin{proof}
 This result is obtained by the classical Hardy inequality and by splitting the complex variable into real and imaginary parts.
 \end{proof}

In order to  gain derivatives from the analyticity,  we need the following key lemma for analytic functions.

\begin{lemma}\label{lem: analyticity recovery}
Let $k\geq0.$ For any $\widetilde{\mu}>\mu> 0,$ we have
\begin{align*}
\|\pa_x f\|_{X^k_\mu}\leq \f{C_0}{\widetilde{\mu}-\mu}\|f\|_{X^k_{\widetilde{\mu}}},\quad\|\varphi(y)\pa_y f\|_{X^k_\mu}\leq \f{C_0}{\widetilde{\mu}-\mu}\|f\|_{X^k_{\widetilde{\mu}}},
\end{align*}
where $\varphi(y)=\f{y}{1+y}.$
\end{lemma}
\begin{proof}
We can follow the proof of Lemma 2.2 in \cite{NN} by using the Cauchy integral formula. 
\end{proof}

In order to estimate the product of analytical functions, we need the following lemma.
\begin{lemma}\label{lem: product 1}
Let $k\geq8$ and $\mu>0.$ It holds that
\begin{align}
\|fg\|_{X^k_\mu}\leq &C_0\d^{-1}(\|f\|_{X^{k-3}_{\mu}}+\|\pa_yf\|_{X^{k-3}_{\mu}})\|g\|_{X^k_\mu}\label{est: product 1.1}+C_0\d^{-1}\|f\|_{X^k_\mu}(\|g\|_{X^{k-3}_\mu}+\|\pa_yg\|_{X^{k-3}_\mu}),\\
\|fg\|_{X^k_\mu}\leq &C_0\d^{-1}(\|f\|_{X^{k}_\mu}+\|\pa_yf\|_{X^{k}_\mu})\|g\|_{X^k_\mu},\label{est: product 1.2}\\
\|fg\|_{X^k_\mu}\leq &C_0(\|f\|_{Y^{k}_\mu}+\|\pa_yf\|_{Y^{k}_\mu})\|g\|_{X^k_\mu} \label{est: product 1.3}.
\end{align}
For $\widetilde{f}(t,x,z)=f(t,x,\f{y}{\e})$,  we have
\begin{align}
\|fg\|_{X^k_\mu}\leq &C_0\big(\|\widetilde{f}\|_{W^{k}_\mu}+\|\pa_z\widetilde{f}\|_{W^{k}_\mu}\big)\|g\|_{X^k_\mu}.\label{est:product1.4}
\end{align}
 
\end{lemma}

\begin{proof}
By the definition of $\|\cdot\|_{X^k_\mu}$ in \eqref{def: f_X^k},
 we have
\begin{align}\label{est; fg-product}
\|fg\|_{X^k_\mu}\leq&\sum_{|\al|=0}^k\|Z^\al (fg)\|_{L^2_\mu} 
\leq \sum_{|\al|=0}^k\sum_{|\beta|+|\gamma|\leq |\al|}\| Z^\beta f Z^\gamma g\|_{L^2_\mu}\\
\nonumber
\leq& \sum_{\substack{ |\beta|+|\gamma|\leq k,\\ |\beta|\leq |\gamma|}}\|Z^\beta f\|_{L^\infty_\mu}\|Z^\gamma g\|_{L^2_\mu}
+\sum_{\substack{ |\beta|+|\gamma|\leq k,\\ |\beta|\geq |\gamma|}}\|Z^\beta f\|_{L^2_\mu}\|Z^\gamma g\|_{L^\infty_\mu}.
\end{align}
For the function $h(t,x,y)$ with $y\in \pa\Om_\th,$ we have
\begin{align*}
|h|^2\leq 2\int_{\pa\Om_\th}|\pa_y h \cdot h| dy\leq 2\|h\|_{L^2(\pa\Om_\th)} \|\pa_yh\|_{L^2(\pa\Om_\th)}\leq \|h\|_{L^2(\pa\Om_\th)}^2+\|\pa_yh\|_{L^2(\pa\Om_\th)}^2,
\end{align*}
which along with the Sobolev embedding  implies that
\begin{align*}
\|f\cdot g\|_{X^k_\mu}\leq & C_0\d^{-1}\sum_{\widetilde{\beta}=0}^1\sum_{\substack{ |\beta|+|\gamma|\leq k,\\ |\beta|\leq \f k2}}\big(\|Z^{\beta+\widetilde{\beta}} f\|_{L^2_\mu}+\|\pa_y Z^{\beta+\widetilde{\beta}} f\|_{L^2_\mu}\big)\|Z^\gamma g\|_{L^2_\mu}\\
&+C_0\d^{-1}\sum_{\widetilde{\gamma}=0}^1\sum_{\substack{ |\beta|+|\gamma|\leq k,\\  |\gamma|\leq \f k2}}\|Z^\beta f\|_{L^2_\mu}\big(\|Z^{\gamma+\widetilde{\gamma}} g\|_{L^2_\mu}+\|\pa_yZ^{\gamma+\widetilde{\gamma}} g\|_{L^2_\mu}\big).
\end{align*}
Here factor $\d^{-1}$ is appeared according to $\pa_x=\d^{-1}Z_1.$

Due to $\f k2+1\leq k-3$ for $k\geq 8$ and the fact $\pa_y Z^\beta g=Z^\beta\pa_yg+[\pa_y, Z^\beta]g=Z^\beta\pa_yg+\d \varphi'|\beta|Z^{\beta-e_2}\pa_y g, e_2=(0,1,0)$, it holds that
\begin{align*}
\|fg\|_{X^k_\mu}\leq  & C_0\d^{-1}\big(\|f\|_{X^{k-3}_\mu}+\|\pa_yf\|_{X^{k-3}_\mu}\big) \|g\|_{X^k_\mu}+C_0\d^{-1}\big(\|g\|_{X^{k-3}_\mu}+\|\pa_yg\|_{X^{k-3}_\mu}\big) \|f\|_{X^k_\mu},
\end{align*}
which gives \eqref{est: product 1.1}.

For the estimates \eqref{est: product 1.2}, \eqref{est: product 1.3} and \eqref{est:product1.4}, we use the similar argument above. The only difference is the definition of norms ( i.e. $Y^k_\mu$ and $W^k_\mu$).  

\end{proof}

The following  lemma can be easily got from the proof of Lemma \ref{lem: product 1}.

\begin{lemma}\label{lem: product 2}
Let $k\geq8$ and $\mu>0.$ It holds that
\begin{align*}
\sup_{\mu<  \mu_0-\la t}\|f g\|_{\widetilde{L}^2(0, t;X^k_{\mu})} 
\leq& C_0\d^{-1}\sup_{ \mu<  \mu_0-\la t}\|f\|_{\widetilde{L}^2(0, t;X^k_{\mu})} \sup_{s\in[0,t]}\sup_{ \mu<  \mu_0-\la s}\big(\|g(s)\|_{X^{k-3}_{\mu}}+\|\pa_y g(s)\|_{X^{k-3}_{\mu}}\big)\\
&+C_0\d^{-1} \sup_{ \mu<  \mu_0-\la t} \|g\|_{\widetilde{L}^2(0, t;X^k_{\mu})}  \sup_{s\in[0,t]}\sup_{  \mu<  \mu_0-\la s}\big(\|f(s)\|_{X^{k-3}_{\mu}}+\|\pa_y f(s)\|_{X^{k-3}_{\mu}}\big),
\end{align*}
\begin{align*}
\sup_{\mu<  \mu_0-\la t}\|f g\|_{\widetilde{L}^2(0, t;X^k_{\mu})} \leq&C_0\d^{-1}\sup_{ \mu<  \mu_0-\la t}\|f\|_{\widetilde{L}^2(0, t;X^k_{\mu})} \sup_{s\in[0,t]}\sup_{ \mu<  \mu_0-\la s}\big(\|g(s)\|_{X^{k}_{\mu}}+\|\pa_y g(s)\|_{X^{k}_{\mu}}\big),
\end{align*}
and
\begin{align*}
\sup_{\mu<  \mu_0-\la t}\|f g\|_{\widetilde{L}^2(0, t;X^k_{\mu})} \leq&C_0\sup_{ \mu<  \mu_0-\la t}\|f\|_{\widetilde{L}^2(0, t;X^k_{\mu})} \sup_{s\in[0,t]}\sup_{ \mu<  \mu_0-\la s}\big(\|g(s)\|_{Y^{k}_{\mu}}+\|\pa_y g(s)\|_{Y^{k}_{\mu}}\big).
\end{align*}
\end{lemma}

\medskip

In this paper, we need some estimates in analytical space for composite functions which are given by the following lemma.
\begin{lemma}\label{lem:comp}
Let $k\geq8$ and $\mu>0$. Assume that $F\in C^\infty(\mathbb{R})$ with $F(0)=0$ and $|F^{(n)}(0)|\le M^n n!$ for $n\ge 1$. 
Then there exists a constant $c>0$ so that if $\|f\|_{X^k_\mu}+\|\pa_yf\|_{X^k_\mu}\le c$, then it holds that
\begin{align*}
\|F(f)\|_{X^k_\mu}\leq &C_0\|f\|_{X^k_\mu}.
\end{align*}
\end{lemma}

\begin{proof}
By the Taylor expansion, we have
\beno
F(f)=\sum_{n\ge 1}\f {F^{(n)}(0)} {n!} f^n,
\eeno
which along with Lemma \ref{lem: product 1} gives 
\begin{align*}
\|F(f)\|_{X^k_\mu}\leq \sum_{n\ge 1} (CM)^n(\|f\|_{X^k_\mu}+\|\pa_yf\|_{X^k_\mu})^{n-1}\|f\|_{X^k_\mu}\le CM\sum_{n\ge 0} (CMc)^n\|f\|_{X^k_\mu},
\end{align*}
which gives the result if $CMc\le 1$.

\end{proof}

\section{Approximate solutions and error system}

\subsection{Approximate solutions}

The construction of the approximate solution is based on the asymptotic matching expansion technique.
Since the process is quite standard, we just give the final result and omit the details of the derivation.  \smallskip

The approximate solution $(\rho^a, u^a, v^a)$ is defined by
\begin{align}\label{eq: app s}
 \left\{
 \begin{aligned}
&\rho^a(t,x,y)= (\rho^0_e+\e\rho^1_e)(t,x,y)+\e^2\rho_p^2(t,x,\f y{\e}), \\
&u^a(t,x,y)= (u_e^0+\e u_e^1)(t,x,y)+(u_p^0+\e u_p^1)(t,x,\f y{\e}), \\
&v^a(t,x,y)= (v^0_e+\e v_e^1)(t,x,y)+(\e v_p^1+\e^2 v_p^2)(t,x,\f y{\e})-\e^2 v_p^2(t,x,0), 
 \end{aligned}
 \right.
  \end{align}
where $(\rho_e^0, u_e^0, v^0_e)= (\rho^e, u^e, v^e)$ is given by \eqref{equ:Euler} with the initial data $(\rho_e^0, u_e^0, v^0_e)|_{t=0}=(\rho_0, u_0, v_0)$. And $(u_p^0, v_p^1)=(u^p, v^p)$ is given by \eqref{eq: (u_p^0, v_p^1)-0}.
 
$(\rho_e^1, u_e^1, v_e^1)$ is defined by 
\begin{align}\label{eq: (rho_e^1, u_e^1, v_e^1)}
\left\{
 \begin{aligned}
&\pa_t \rho_e^1+u_e^0\pa_x \rho_e^1+v_e^0\pa_y \rho_e^1+u_e^1\pa_x \rho_e^0+v_e^1\pa_y \rho_e^0+\rho_e^0(\pa_x u_e^1+\pa_y v_e^1)\\
&\qquad\qquad+\rho_e^1(\pa_x u_e^0+\pa_y v_e^0)=0,\\
&\pa_t u_e^1+u_e^0\pa_x u_e^1+v_e^0\pa_y u_e^1+u_e^1\pa_x u_e^0+v_e^1\pa_y u_e^0+\pa_x\rho_e^1=0,\\
&\pa_t v_e^1+u_e^0\pa_x v_e^1+v_e^0\pa_y v_e^1+u_e^1\pa_x v_e^0+v_e^1\pa_y v_e^0+\pa_y\rho_e^1=0,\\
&v_e^1(t,x,0)=-v_p^1(t,x,0),\\
&(\rho_e^1, u_e^1, v_e^1)|_{t=0}=0.
\end{aligned}
 \right.
  \end{align}
 
$(u_p^1, v_p^2)$ solves
\begin{align}\label{eq: (u_p^1, v_p^2)}
\left\{
\begin{aligned}
&\pa_t u_p^1+(\overline{u_e^0}+u_p^0)\pa_z u_p^1+(\overline{v_e^1}+v_p^1+z\overline{\pa_y v_e^0})\pa_z u_p^1+(u_p^1+\overline{u_e^1}+z\overline{\pa_y u_e^0})\pa_x u_p^0\\
&\qquad+(v_p^2- \overline{v_p^2}+z\overline{\pa_y v_e^0})\pa_z u_p^0+ \f12 z^2\overline{\pa_y^2 v_e^0}\pa_z u_p^0-\f{\nu}{\overline{\rho_e^0}}\pa_z^2 u_p^1-\f{\nu\overline{\rho_e^1}}{\overline{\rho_e^0}^2}\pa_z^2 u_p^0\\
&\qquad\qquad\qquad\qquad=-\Big(z\overline{\pa_x\pa_y v_e^0} u_p^0+u_p^0\overline{\pa_xu_e^1}+v_p^1\overline{\pa_y u_e^0}\Big),\\
&\pa_x(\overline{\rho_e^1}u_p^0+\overline{\rho_e^0}u_p^1 )+\pa_z(\overline{\rho_e^1} v_p^1+\overline{\rho_e^0} v_p^2)=0,\\
&u_p^1|_{z=0}=-\overline{u_e^1},\quad v_p^2|_{z=+\infty}=0,\\
&u_p^1|_{t=0}=0,
\end{aligned}
\right.
\end{align}
 and $\rho_p^2$ is given by 
 \begin{align*}
 \rho_p^2=\int_z^{+\infty} \mathcal{P}_p^2(t,x,z') dz',
 \end{align*}
where
\begin{align*}
\mathcal{P}_p^2(t,x,z)=&\pa_t v_p^1+\overline{u_e^0}\pa_x v_p^1+u_p^0(\overline{\pa_x v_e^1}+\pa_x v_p^1)+v_p^1(\overline{\pa_y v_e^0}+\pa_z v_p^1)+\overline{v_e^1}\pa_z v_p^1\\
&\qquad-\f{\nu}{\overline{\rho_e^0}}\pa_z^2 v_p^1-\f{\nu+\sigma}{\overline{\rho_e^0}}(\pa_x\pa_z u_p^0+\pa_z^2 v_p^1).
\end{align*}
Here
\beno
z=\f{y}{\e},\quad \overline{v_p^1}=\f{\int_0^{+\infty}\pa_x(\overline{\rho_e^0} u_p^0)dy}{\overline{\rho_e^0}}
\eeno
and 
\beno
\overline{v_p^2}=\f{\int_0^\infty \pa_x(\overline{\rho_e^1}u_p^0+\overline{\rho_e^0}u_p^1)dy-\overline{\rho_e^1}~\overline {v_p^1}}{\overline{\rho_e^0}}.
\eeno

\medskip

By a direct calculation, the approximate solution $(\rho^a, u^a, v^a)$ satisfies 
\begin{align}\label{eq: CNS-a}
\left\{
\begin{aligned}
&\pa_t \rho^a+u^a\pa_x \rho^a+v^a\pa_y \rho^a+\rho^a(\pa_x u^a+\pa_yv^a )=-R_\rho,\\
& \pa_t u^a+u^a\pa_x u^a+v^a\pa_y u^a  -\f{\e^2}{\rho^a}\nu\tri u^a-\f{\e^2}{\rho^a}(\sigma+\nu)\pa_x(\pa_x u^a+\pa_y v^a)+\pa_x  \rho^a  =-R_u,\\
& \pa_t v^a+u^a\pa_x v^a+v^a \pa_y v^a -\f{\e^2}{\rho^a}\nu\tri v^a-\f{\e^2}{\rho^a}(\sigma+\nu)\pa_y(\pa_x u^a+\pa_y v^a)+\pa_y \rho^a=-R_v,\\
&(u^a, v^a)|_{y=0}=0,\\
&(\rho^a, u^a, v^a)|_{t=0}=(\rho_0, u_0, v_0),
\end{aligned}
\right.
\end{align}
where $(R_\rho, R_u, R_v) \sim \e^2$.

\medskip

\subsection{Error system}
We define the reminder terms as follows
\beno
(\rho^R,u^R, v^R)=(\rho^\e,u^\e, v^\e)-(\rho^a,u^a, v^a),
\eeno
which satisfies the following error system 
\begin{align} \nonumber
\left\{
\begin{aligned}
&\pa_t \rho^R+u^a\pa_x \rho^R+v^a\pa_y \rho^R+u^R\pa_x \rho^a+v^R\pa_y \rho^a+\rho^a(\pa_x u^R+\pa_y v^R)\\
&\qquad+\rho^R(\pa_x u^a+\pa_y v^a)=R_\rho -\mathcal{N}_\rho ,\\
&\pa_t u^R+u^a\pa_x u^R+v^a\pa_y u^R+u^R\pa_x u^a+v^R\pa_y u^a+\pa_x \rho^R-\f{\e^2}{\rho^a+\rho^R}\nu\tri u^R-\e^2\nu\big(\f{1}{\rho^\e}-\f{1}{\rho^a}\big)\tri u^a\\
&\qquad-\f{\e^2}{\rho^a+\rho^R}(\sigma+\nu)\pa_x(\pa_x u^R+\pa_y v^R)-\e^2(\nu+\sigma)\big(\f{1}{\rho^\e}-\f{1}{\rho^a}\big)\pa_x(\pa_xu^a+ \pa_yv^a) =R_u -\mathcal{N}_u,\\
&\pa_t v^R+u^a\pa_x v^R+v^a\pa_y v^R+u^R\pa_x v^a+v^R\pa_y v^a+\pa_y \rho^R-\f{\e^2}{\rho^a+\rho^R}\nu\tri v^R-\e^2\nu\big(\f{1}{\rho^\e}-\f{1}{\rho^a}\big)\tri v^a\\
&\qquad-\f{\e^2}{\rho^a+\rho^R}(\sigma+\nu)\pa_y(\pa_x u^R+\pa_y v^R)-\e^2(\nu+\sigma)\big(\f{1}{\rho^\e}-\f{1}{\rho^a}\big)\pa_y(\pa_x u^a+\pa_y v^a)  =R_v-\mathcal{N}_v,\\
&(u^R,v^R)|_{y=0}=0,\\
&(u^R,v^R,\rho^R)|_{t=0}=0,
\end{aligned}
\right.
\end{align}
where  $\mathcal{N}_\rho, \mathcal{N}_u$ and $\mathcal{N}_v$ are nonlinear terms defined by
\begin{align}
&\mathcal{N}_\rho=u^R\pa_x \rho^R+v^R\pa_y \rho^R+\rho^R(\pa_x u^R+\pa_y v^R),\label{def: N_rho}\\
& \mathcal{N}_u= u^R\pa_x u^R+v^R\pa_y u^R,\label{def: N_u}\\
&\mathcal{N}_v= u^R\pa_x v^R+v^R\pa_y v^R.\label{def: N_v}
\end{align}

To simplify the notations, we denote
\begin{align}
\label{eq: F_rho}
F_\rho\eqdef&u^a\pa_x \rho^R+v^a\pa_y \rho^R+u^R\pa_x \rho^a+v^R\pa_y \rho^a+\rho^a\pa_x u^R+\rho^R(\pa_x u^a+\pa_y v^a)-R_\rho ,\\
\label{eq: F_u}
F_u\eqdef&u^a\pa_x u^R+v^a\pa_y u^R+u^R\pa_x u^a+v^R\pa_y (u^a-u_p^0)+\pa_x \rho^R\\
\nonumber
&-\e^2\nu\big(\f{1}{\rho^\e}-\f{1}{\rho^a}\big)\tri u^a-\e^2(\nu+\sigma)\big(\f{1}{\rho^\e}-\f{1}{\rho^a}\big)\pa_x(\pa_x u^a+\pa_y v^a)  -R_u,\\
\label{eq: F_v}
F_v\eqdef&u^a\pa_x v^R+v^a\pa_y v^R+u^R\pa_x v^a+v^R\pa_y v^a\\
\nonumber
&-\e^2\nu\big(\f{1}{\rho^\e}-\f{1}{\rho^a}\big)\tri v^a-\e^2(\nu+\sigma)\big(\f{1}{\rho^\e}-\f{1}{\rho^a}\big)\pa_y(\pa_x u^a+\pa_y v^a) -R_v.
\end{align}
Here, we point out that $F_\rho, F_u$ and $F_v$ lose one derivative in $\pa_x$ and $\phi\pa_y$.

\medskip

Based on the above notations, the error system is rewritten as 
\begin{align}\label{eq: Error-(u,v,rho)}
\left\{
\begin{aligned}
&\pa_t \rho^R+\rho^a \pa_y v^R = -\mathcal{N}_\rho-F_{\rho},\\
&\pa_t u^R+ v^R\pa_y u_p^0-\e^2\nu\mathfrak{a}\tri u^R-\e^2(\nu+\sigma)\mathfrak{a}\pa_x(\pa_x u^R+\pa_y v^R)=- \mathcal{N}_u-F_{u} ,\\
&\pa_t v^R+ \pa_y \rho^R-\e^2\nu\mathfrak{a}\tri v^R-\e^2(\nu+\sigma)\mathfrak{a}\pa_y(\pa_x u^R+\pa_y v^R) =- \mathcal{N}_v-F_v,\\
&(u^R,v^R)|_{y=0}=0,\\
&(u^R,v^R,\rho^R)|_{t=0}=0,
\end{aligned}
\right.
\end{align}
where $\mathfrak{a}\eqdef\f{1}{\rho^a+\rho^R}$. For convenience, we denote $\mathfrak{a}_0\eqdef\f{1}{\rho^a}.$ By using the equation \eqref{eq: Error-(u,v,rho)}, we have
\begin{align}\label{initial: 6}
\pa_t^k(\rho^R, u^R, v^R)|_{t=0}\sim O(\e^2)\quad \textrm{for}\,\,k\ge 1.
\end{align}

 \bigskip

\section{Tangential energy estimates}

The goal of this section is to give the tangential energy estimates of $(u^R, v^R,\rho^R)$. We recall that $(u^R, v^R,\rho^R)$ satisfies \eqref{eq: Error-(u,v,rho)} and we extend the system  \eqref{eq: Error-(u,v,rho)} into a complex plane $D_\mu\times\Om_\mu$. More precisely, we view the system \eqref{eq: Error-(u,v,rho)} with  the variable $(t,x,y)\in[0, T]\times D_\mu\times \Om_\mu$ and $\mu< \mu_0-\la T$.\smallskip

For convenience, we take $\nu=1, ~\sigma=0$ in \eqref{eq: Error-(u,v,rho)} in the sequel. The case $\nu>0$,  $\nu+\sigma\geq0$ also can be handled by the same method in this paper. Thus, the error system we consider in \eqref{eq: Error-(u,v,rho)}  is reduced to 
\begin{align}\label{eq: Error-(u,v,rho)-1}
\left\{
\begin{aligned}
&\pa_t \rho^R+ \rho^a\pa_y v^R = -\mathcal{N}_\rho-F_{\rho},\\
&\pa_t u^R+ v^R\pa_y u_p^0-\e^2\mathfrak{a}\tri u^R-\e^2\mathfrak{a}\pa_xd^R=- \mathcal{N}_u-F_{u} ,\\
&\pa_t v^R+ \pa_y \rho^R-\e^2\mathfrak{a}\tri v^R-\e^2\mathfrak{a}\pa_yd^R=- \mathcal{N}_v-F_v,\\
&(u^R,v^R)|_{y=0}=0,\\
&(u^R,v^R,\rho^R)|_{t=0}=0.
\end{aligned}
\right.
\end{align}
Here we denote $d^R$ the divergence of velocity, i.e., $d^R=\pa_x u^R+\pa_y v^R.$
\medskip

To proceed, let's assume that the following uniform estimates hold for the approximate solution. The proof is presented in the appendix.

\begin{lemma}\label{lem: app}
There exists $T_a>0$ such that for any $t\in[0, T_a]$, there hold for $i=0,1$,
\begin{align*}
&\sup_{\mu< 8\mu_0-\la_E t}\big(\|(u_e^i, v_e^i)\|_{Y^{16}_\mu}+\|\rho_e^1\|_{Y^{16}_\mu}\big)\leq C_0,\\
&\sup_{\mu< 3\mu_0-\la_P t}\sum_{k\leq 2}\|e^{z^2}\pa_z^k(u_p^i, v_p^{i+1},\rho_p^{i+1})\|_{W^{13}_\mu}\leq C_0.
\end{align*}
\end{lemma}

 With Lemma \ref{lem: app},  by using product estimates (Lemma \ref{lem: product 1}),  we deduce that 
\begin{lemma}\label{cor: est-a_0}
There exists a constant $C_0>0$ such that for any $t\in [0,T_a]$,
 \begin{align*}
 &\e^{-1}\sup_{\mu< \mu_0-\la t}\|(R_\rho, R_u, R_v)\|_{X^{10}_\mu}+ \sup_{\mu< \mu_0-\la t}\|(\pa_yR_\rho, \pa_yR_u, \pa_yR_v)\|_{X^{9}_\mu}\leq C_0\e,\\
%&\sup_{\mu< \mu_0-\la t} \|(\rho^a-1,\na\rho^a, \mathfrak{a}_0-1,\na\mathfrak{a}_0)\|_{X^{12}_\mu}\leq C\e,\\
&\sup_{\mu< \mu_0-\la t}\Big( \|(u^a,v^a,\rho^a, \mathfrak{a}_0)\|_{X^{12}_\mu}+\|\pa_y(v^a,\rho^a, \mathfrak{a}_0)\|_{X^{12}_\mu}+\|(\pa_y^2 \rho^a, \pa_y^2 \mathfrak{a}_0)\|_{X^{12}_\mu}\Big)\leq C_0.\end{align*}
\end{lemma}
\medskip

Let us first introduce the energy functional
\begin{align}
\mathcal{E}(t)=&\sup_{\mu<\mu_0-\la t}h^\eta(t,\mu)\Big(\e^{-2}\|(A^{-\f12} u^R,  v^R,  \rho^R)(t)\|_{X^{10}_\mu}^2+h(t,\mu)\|\pa_x\rho^R(t)\|_{X^{10}_\mu}^2\label{def: E}\\
\nonumber
&\qquad\qquad\qquad\qquad+\|\pa_y( u^R,  u^R,  \rho^R)(t)\|_{X^{9}_\mu}^2\Big),\\
\mathcal{D}(t)=&\sup_{\mu<\mu_0-\la t}h^\eta(t,\mu)\Big(\| \na ( u^R, v^R)\|^2_{\widetilde{L}^2(0, t ;X^{10}_{\mu})}+\| \e(\pa_y^2 u^R,\pa_y^2 v^R)\|^2_{\widetilde{L}^2(0, t ;X^{9}_{\mu})}\Big),\label{def: D}
\end{align}
where $A>1$ is a large constant determined later and  $\eta\in(0,1)$ is a small constant. Here 
 \ben\label{def:h}
 h(t,\mu)=\mu_0-\mu-\la t.
 \een
%  Moreover,  we  denote 
%\begin{align}\label{def:G}
%G(t)=\sup_{s\in[0,t]}\sup_{\mu<\mu_0-\la s}\big(\| \e(u^R,v^R)\|_{X^9_\mu}+\|\e\pa_y^2(u^R,v^R)\|_{X^7_\mu}\big).
%\end{align}

 In what follows, we assume that  
 \begin{align}
&\sup_{t\in[0,T]} \sup_{\mu<\mu_0-\la t}\Big(\e^{-1}\|(\rho^R, u^R, v^R)\|_{X^9_{\mu}}+\|\pa_y(\rho^R, u^R, v^R)\|_{X^8_{\mu}}+\|\pa_y^2(\e u^R, v^R)\|_{X^7_\mu}\Big)\leq  \e^\f12\label{assume: 1},
 \end{align}
for any $T\leq T_1$, where $T_1=\min\{\f{\mu_0}{\la}, T_a\}.$  It is easy to see that there exists a constant $c_0>0$ satisfying 
\beno
\mathfrak{a}(t,x,y)\geq c_0,\quad \forall (t, x,y)\in [0, T]\times D_{\mu}\times \Om_\mu.
\eeno

It follows from Lemma \ref{lem: product 1} that
\ben\label{eq:a-diff}
\|\mathfrak{a}-\mathfrak{a}_0\|_{X_\mu^{10}}\le C_0\|\rho^R\|_{X_\mu^{10}},\quad \|\pa_y(\mathfrak{a}-\mathfrak{a}_0)\|_{X_\mu^{9}}\le C_0\|\pa_y\rho^R\|_{X_\mu^{9}},
\een
by using the process in Lemma \ref{lem:comp} to ensure  that $C(\|\rho^R\|_{X^7_\mu}+\|\pa_y\rho^R\|_{X^{7}_\mu})\leq C\e^\f12\leq  \f12$ after taking $\e$ small enough.

The following facts will be used frequently:
\begin{align}\label{est: int-f,g}
 \big\langle f, g  \big\rangle_{X^k_{\mu,t}}\leq C_0\|h^{\f12}f\|_{\widetilde{L}^2(0,t; X^k_{\mu})}\|h^{-\f12}g\|_{\widetilde{L}^2(0,t; X^k_{\mu})},
\end{align}
and for $\eta\in(0,1)$ small,
\begin{align}\label{est: u^R_L^2_t}
\|f\|^2_{\widetilde{L}^2(0,t; X^{k}_\mu)}\le C_0\|h^{-\f12} f\|^2_{\widetilde{L}^2(0,t; X^{k}_\mu)} \leq C_0\la^{-1} h^{-\eta}(t, \mu) \sup_{s\in[0,t]}\sup_{\mu< \mu_0-\la s}\Big(h^{\eta}(s,\mu)\| f\|_{X^{k}_{\mu}}^2\Big),
\end{align}
where we use the estimate
\begin{align}\label{est: integral-1}
\int_0^t h(s,\mu)^{-1-\eta}ds\leq C_0\la^{-1}h^{-\eta}(t,\mu).
\end{align}

 \subsection{Estimate of $(F_u, F_v, F_\rho)$.} 
 
 \begin{lemma}\label{lem: (F_u, F_v, F_rho)}
It holds that
\begin{align}
\label{est: F_u}
&\|F_u\|_{X^{10}_\mu}\leq C_0\Big(\e^2+\|\pa_x (u^R,\rho^R)\|_{X^{10}_\mu}+\|\varphi \pa_yu^R\|_{X^{10}_\mu}+\|(u^R, v^R, \rho^R)\|_{X^{10}_\mu} \Big),
\\
\label{est: F_v}
&\|F_v\|_{X^{10}_\mu}\leq C_0 \Big(\e^2+\|\pa_x v^R\|_{X^{10}_\mu}+\|\varphi \pa_yv^R\|_{X^{10}_\mu}+\|(u^R, v^R, \rho^R)\|_{X^{10}_\mu} \Big),
\\
\label{est: F_rho}
&\|F_\rho\|_{X^{10}_\mu}\leq C_0 \Big(\e^2+\|(\pa_x u^R,\pa_x\rho^R)\|_{X^{10}_\mu}+\|\varphi \pa_y\rho^R\|_{X^{10}_\mu}+\|(u^R, v^R, \rho^R)\|_{X^{10}_\mu} \Big),\\
\label{est: pa_y v^R}
&\|\pa_y v^R\|_{X^9_\mu}\leq C_0\Big(\e^2+\|\mathcal{N}_\rho\|_{X^9_\mu}\Big)+C\|(\rho^R, u^R, v^R)\|_{X^{10}_\mu}.
\end{align}

\end{lemma}

\begin{proof}
We first give the proof of \eqref{est: F_u}. By the definition of $F_u$, we have
\begin{align*}
\|F_u\|_{X^{10}_\mu}\leq&\|v^a\pa_y u^R\|_{X^{10}_\mu}+\|u^a\pa_x u^R\|_{X^{10}_\mu}+\|u^R\pa_x u^a\|_{X^{10}_\mu}+\|v^R\pa_y (u^a-u_p^0)\|_{X^{10}_\mu}+\|\pa_x \rho^R\|_{X^{10}_\mu}\\
&+\|\e^2(\f{1}{\rho^\e}-\f{1}{\rho^a})\tri u^a\|_{X^{10}_\mu}+\|\e^2\big(\f{1}{\rho^\e}-\f{1}{\rho^a}\big)\pa_x(\pa_x u^a+\pa_y v^a)\|_{X^{10}_\mu}  +\|R_u\|_{X^{10}_\mu}\\
\eqdef&I_1+\cdots+I_8.
\end{align*}
Next, we give the estimate of $I_i (i=1,\cdots, 8)$ one by one.  By Lemma \ref{lem: product 1} and Lemma \ref{lem: app}, we obtain
\begin{align*}
I_2\leq& C_0\big(\|(u^0_e, u^1_e)\|_{Y^{11}_\mu}+\|(\pa_yu^0_e,\pa_yu^1_e)\|_{Y^{11}_\mu}+\|(u^0_p,u^1_p)\|_{W^{11}_\mu}+\|(\pa_zu^0_p,\pa_zu^1_p)\|_{W^{11}_\mu}\big)\|\pa_x u^R\|_{X^{10}_\mu}\\
\leq&C_0\|\pa_x u^R\|_{X^{10}_\mu}.
\end{align*}

Notice that 
\beno
\pa_y (u^a-u_p^0)\sim O(1),\qquad
\e^2\tri u^a=\e^2\tri u^a_e+\e^2\pa_x^2u^a_p+\pa_z^2u^a_p\sim O(1),\quad \f{1}{\rho^a}\sim O(1).
\eeno
By  Lemma \ref{lem: product 1},  Lemma \ref{lem: app} and Lemma \ref{cor: est-a_0}, 
we infer that
\begin{align*}
I_3+\cdots+I_8\leq&C_0(\|u^R\|_{X^{10}_\mu}+\|v^R\|_{X^{10}_\mu}+\|\rho^R\|_{X^{10}_\mu}+\|\pa_x\rho^R\|_{X^{10}_\mu}+\e^2).
\end{align*}

 For $I_1$, we rewrite it as
 \beno
 Z^\al(v^a\pa_y u^R)=\sum_{|\beta|+|\gamma|=|\al|}Z^{\beta}v^aZ^\gamma\pa_y u^R=\sum_{|\beta|+|\gamma|=|\al|}\f{Z^{\beta}v^a}{\varphi} Z^{\gamma}\varphi\pa_y u^R+\sum_{|\beta|+|\gamma|=|\al|}\f{Z^{\beta}v^a}{\varphi} [\varphi, Z^{\gamma}]\pa_y u^R
 \eeno
  with $\varphi=\f{y}{1+y}.$ A detailed calculation gives
  \begin{align*}
  [\varphi, Z^{\gamma}]\pa_y f=\sum_{|\widetilde{\gamma}|=0}^{|\gamma|}C_{\widetilde{\gamma}}Z^{\widetilde{\gamma}}f.
  \end{align*}
Then we deduce from  Lemma \ref{lem: product 1}, Lemma \ref{lem: Hardy inequality} ($v^a|_{y=0}=0$) and Lemma \ref{cor: est-a_0} that
 \begin{align*}
 I_1\leq C_0(\|\varphi \pa_yu^R\|_{X^{10}_\mu}+\|u^R\|_{X^{10}_\mu}).
 \end{align*}

Summing up, we conclude 
\begin{align*}
\|F_u\|_{X^{10}_\mu}\leq&C_0\Big(\e^2+\|\pa_x (u^R,\rho^R)\|_{X^{10}_\mu}+\|\varphi\pa_yu^R\|_{X^{10}_\mu}+\|(u^R, v^R, \rho^R)\|_{X^{10}_\mu}\Big).
\end{align*}

\medskip

 The estimates of $F_v$ and $F_\rho$ are similar to $F_u$. It remains to estimate  $\pa_y v^R$. For this,
using the first equation of \eqref{eq: Error-(u,v,rho)-1},
 \begin{align}
 \pa_t \rho^R+\rho^a \pa_y v^R = -\mathcal{N}_\rho-F_\rho,
 \end{align}
 we deduce from Lemma \ref{lem: product 1}, Lemma \ref{cor: est-a_0} and \eqref{est: F_rho} that 
 \begin{align*}
\|\pa_y v^R\|_{X^9_\mu}\leq&\| \mathfrak{a}_0\pa_t \rho^R\|_{X^9_\mu}+\|\mathfrak{a}_0F_{\rho}\|_{X^9_\mu}+\|\mathfrak{a}_0\mathcal{N}_\rho\|_{X^9_\mu}\\
\leq&C_0(\| \pa_t \rho^R\|_{X^9_\mu}+\|F_{\rho}\|_{X^9_\mu}+\|\mathcal{N}_\rho\|_{X^9_\mu})\\
\leq&C_0(\e^2+\|\mathcal{N}_\rho\|_{X^9_\mu})+C\|(\rho^R, u^R, v^R)\|_{X^{10}_\mu},
\end{align*}
for $Z_0=\d\pa_t,$
which gives \eqref{est: pa_y v^R}. 

  \end{proof}

 \subsection{Estimate of  the most difficult term $v^R \pa_y (u_p^0(t,x, \f y{\e}))$. }
In order to  give the energy estimates of $h^{\eta}(t,\mu)\|(u^R, v^R, \rho^R)\|_{X^{10}_\mu}^2$,  the most difficult term comes from $v^R \pa_y (u_p^0(t,x, \f y{\e}))$. We rewrite this term as 
 \beno
 v^R \pa_y u_p^0=\f{v^R}{y} z\pa_zu_p^0\sim \pa_y v^R z\pa_zu_p^0,
 \eeno
 due to $v^R|_{y=0}=0,$ which lose one derivative in $y.$  To deal with this term, we need to use the density equation to avoid the derivative loss in $t$.
\begin{lemma}\label{lem: est: A}
Let $\la>0$ and $T\leq T_1$. Then there exist $\delta_0>0, A_0>1$ such that for any $t \in[0,T]$, $\delta\in(0,\delta_0)$ and $A\geq A_0$, it holds that
\begin{align*}  
&h^\eta(t,\mu) \big\langle  v^R  \pa_y u_p^0 , ~ u^R  \big\rangle_{X^{10}_{\mu,t}} \\
&\leq C_0\e^4+(\f12-\f1 A)h^\eta(t,\mu)\|u^R(t)\|_{X^{10}_\mu}^2+C_0 \d \e^2 \mathcal{D}(t)+\Big(\f{C_0}{A} +\f{CA}{\la}  \Big)\sup_{s\in[0,t]}\e^2 \mathcal{E}(s)\\
&\quad 
+C_0\d^3 h^{\eta}(t,\mu)\|h^{\f12}(\mathcal{N}_u,\mathcal{ N}_\rho)\|_{\widetilde{L}^2(0,t; X^{10}_\mu)}^2
+C_0\d \la^{-1}\sup_{s\in[0,t]}\sup_{\mu< \mu_0-\la s}\Big( h^\eta(s,\mu)\|\mathcal{N}_\rho\|^2_{X^9_\mu}\Big).
\end{align*}

\end{lemma}

\begin{proof}
We split $\langle v^R   \pa_y u_p^0, u^R \rangle_{X^{10}_{\mu,t}}$ into two parts
\beno
  \big\langle v^R   \pa_y u_p^0, u^R  \big\rangle_{X^{10}_{\mu,t}} \eqdef A_0+A_1,
\eeno
where $A_0$ and $A_1$ are defined by
\begin{align*}
&A_0\eqdef  \sum_{k\leq 10}\d^{2k}A_0^k\eqdef \sum_{k\leq 10} \d^{2k} \langle \pa_s^{k}(v^R   \pa_y u_p^0), \pa_s^{k}u^R \rangle_{X^{0}_{\mu,t}}, \\
&A_1\eqdef \sum_{\substack{|\al|\leq 10\\\al_1+\al_2\geq1}}
\langle Z^\al(v^R   \pa_y u_p^0), Z^\al u^R \rangle_{X^{0}_{\mu,t}} .
\end{align*}
Notice that when the index $\al_1+\al_2\geq1$ in $A_1$,   $Z^\al$ contains $\d\pa_x$ or $\d\varphi \pa_y$. While,  $A_0$ only contains the time derivative. 

\medskip

\underline{Estimate of $A_1.$}   
Let $h(s,\mu)=\mu_0-\mu-\la s> 0$ and $\mu'=\mu+\f12 h(s,\mu)>\mu$ . It is easy to see that $h(s,\mu')=\f12 h(s,\mu)$ and 
\begin{align}\label{equal mu-mu'}
\mu< \mu_0-\la s\quad \mbox{if and only if}\quad \mu'<\mu_0-\la s.
\end{align}
By Lemma \ref{lem: analyticity recovery}, we have
\begin{align*}
\|(\pa_x, \varphi\pa_y)f\|_{X^{k}_{\mu}}\leq& \f{C_0}{\mu'-\mu}\| f\|_{X^{k}_{\mu'}}\leq \f{C_0}{h(s,\mu)}\| f\|_{X^{k}_{\mu'}}.
\end{align*}
 Then we infer from Lemma \ref{lem: product 1}, Lemma \ref{lem: app} and Lemma \ref{lem: Hardy inequality} that  for $k\le 9$,
\beno
\| (\pa_x, \varphi\pa_y)(v^R   \pa_y u_p^0)\|_{X^k_\mu} \leq C_0\Big( \| (\pa_x, \varphi\pa_y)(\f{v^R}{y}  ) \|_{X^k_\mu}+ \|  \f{v^R} {y}   \|_{X^k_\mu}\Big)  \leq \f{C_0}{h(s,\mu)} \|\pa_y v^R \|_{X^k_{\mu'}}.
\eeno
 Then we obtain
\begin{align}\label{est: integral-gain}
|A_1| \leq& \d\int_0^t\|(\pa_x, \varphi\pa_y)(v^R   \pa_y u_p^0)\|_{X^9_\mu}\|u^R\|_{X^{10}_\mu} ds\\
 \leq&C_0 \d\int_0^t h^{-1}(s, \mu)\|\pa_yv^R \|_{X^9_{\mu'}}\|u^R\|_{X^{10}_\mu} ds\nonumber\\
 \nonumber
 \leq&  C_0\d\sup_{s\in[0,t]}\sup_{\mu< \mu_0-\la s}\Big(\Big(h^{\f{\eta}{2}}(s,\mu')\|\pa_yv^R\|_{X^{9}_{\mu'}}\Big)\times \Big(h^{\f{\eta}{2}}(s,\mu)\| u^R\|_{X^{10}_{\mu}}\Big)\Big)\nonumber \\
 &\quad \times \int_0^t h^{-1-\f{\eta}2}(s,\mu) \cdot h^{-\f{\eta'}{2}} (s,\mu')ds\nonumber\\
\nonumber
 \leq&C_0\d \la^{-1}h^{-\eta}(t,\mu) \sup_{s\in[0,t]}\sup_{\mu< \mu_0-\la s}\Big(h^\eta(s,\mu)\| \pa_y v^R(s)\|_{X^{9}_{\mu}}\| (u^R,v^R,\rho^R)(s)\|_{X^{10}_{\mu}}\Big),
\end{align}
where we used \eqref{equal mu-mu'} and \eqref{est: integral-1}. Then we get by \eqref{est: pa_y v^R} that
\begin{align}\label{est: A_1}
|A_1|\leq& \d\la^{-1}h^{-\eta}(t,\mu) \sup_{s\in[0,t]}\sup_{\mu\leq \mu_0-\la s}\Big(h^\eta(s,\mu) \Big(C_0\e^2+C\|(\rho^R, u^R, v^R)\|_{X^{10}_\mu}  +C_0\|\mathcal{N}_\rho\|_{X^9_\mu}\Big)\\
\nonumber
&\qquad\qquad\qquad\qquad\qquad\qquad\times\| (u^R,v^R,\rho^R)(s)\|_{X^{10}_{\mu}}\Big)\\
\nonumber
\leq&C_0h^{-\eta}(t,\mu)\e^4 +C \la^{-1} h^{-\eta}(t,\mu) \sup_{s\in[0,t]}\sup_{\mu< \mu_0-\la s}\Big( h^\eta(s,\mu)\|(u^R, v^R, \rho^R)\|^2_{X^{10}_\mu}\Big)\\
\nonumber
&\qquad\qquad\qquad+C_0 \d\la^{-1}h^{-\eta}(t,\mu) \sup_{s\in[0,t]}\sup_{\mu< \mu_0-\la s}\Big( h^\eta(s,\mu)\|\mathcal{N}_\rho\|^2_{X^9_\mu}\Big)\\
\nonumber
\leq& C_0h^{-\eta}(t,\mu)\e^4 +C A\la^{-1} h^{-\eta}(t,\mu) \sup_{s\in[0,t]}\e^2 \mathcal{E}(s)\\
\nonumber
&+C_0 \d\la^{-1}h^{-\eta}(t,\mu) \sup_{s\in[0,t]}\sup_{\mu< \mu_0-\la s}\Big( h^\eta(s,\mu)\|\mathcal{N}_\rho\|^2_{X^9_\mu}\Big).
\end{align}

\underline{Estimate of $A_0$.}  This part is the most difficult one.  For any $y\in\Om_\mu,$ there exists $\th\in[0,\mu)$ such that $y\in\pa\Om_\th.$ Let $\mathcal{C}\subset\pa\Om_\th $ be a curve connecting $0$ to $y$. According to Lemma \ref{lem: integration by parts}  with $f= v^R, ~g=1$ and $v^R|_{y=0}=0$, we may write
\begin{align*}
v^R=\int_\mathcal{C}\pa_y v^R dy.
\end{align*}
We introduce a linear operator $\mathcal{L}(f)=\f1y\int_\mathcal{C}f dy,$ where $\mathcal{C}$ is a curve given in Lemma \ref{lem: Hardy inequality}. Then trouble term $v^R\pa_y u_p^0$ can be written as 
\begin{align*}
v^R\pa_y u_p^0=\mathcal{L}(\pa_y v^R) z\pa_z u_p^0.
\end{align*}
Using the density equation in \eqref{eq: Error-(u,v,rho)}, we infer 
\begin{align*}
A_0^k&=\Big\langle \pa_s^{k}\Big(\mathcal{L}(\pa_y v^R) ~z\pa_z u_p^0\Big), ~ \pa_s^{k}u^R \Big\rangle_{X^{0}_{\mu,t}}  \\
&=-\Big\langle \pa_s^{k}\Big(\mathcal{L}(\mathfrak{a}_0\pa_s \rho^R )~z\pa_z u_p^0\Big), ~ \pa_s^{k}u^R \Big\rangle_{X^{0}_{\mu,t}}-\Big\langle \pa_s^{k}\Big(\mathcal{L} (\mathfrak{a}_0\mathcal{N}_\rho+\mathfrak{a}_0F_\rho) ~z\pa_z u_p^0\Big), ~ \pa_s^{k}u^R \Big\rangle_{X^{0}_{\mu,t}}\\
&\eqdef-\Big\langle \pa_s^{k}\Big(\mathcal{L}( \mathfrak{a}_0\pa_s \rho^R )~z\pa_z u_p^0\Big), ~ \pa_s^{k}u^R \Big\rangle_{X^{0}_{\mu,t}}+D_1^k.
\end{align*}

For the first term on the right-hand side,  we get by integration by parts in $t$ that  
\begin{align*}
&-\Big\langle \pa_s^{k}\Big(\mathcal{L}( \mathfrak{a}_0\pa_s \rho^R )~z\pa_z u_p^0\Big), \pa_s^{k}u^R \Big\rangle_{X^{0}_{\mu,t}}\\
&= -\Big\langle \pa_s^{k+1}\Big(\mathcal{L}(\mathfrak{a}_0\rho^R )~z\pa_z u_p^0\Big), \pa_s^{k}u^R \Big\rangle_{X^{0}_{\mu,t}}+\Big\langle \pa_s^{k}\Big(\mathcal{L}(\mathfrak{a}_0\rho^R )~z\pa_s\pa_z u_p^0\Big), \pa_s^{k}u^R \Big\rangle_{X^{0}_{\mu,t}}\\
&\quad +\Big\langle \pa_s^{k}\Big(\mathcal{L}(\pa_s\mathfrak{a}_0\rho^R )~z\pa_z u_p^0\Big), \pa_s^{k}u^R \Big\rangle_{X^{0}_{\mu,t}}\\
&=\Big\langle \pa_s^{k}\Big(\mathcal{L}(\mathfrak{a}_0\rho^R )~z\pa_z u_p^0\Big), \pa_s^{k+1}u^R \Big\rangle_{X^{0}_{\mu,t}}+\Big\langle \pa_s^{k}\Big(\mathcal{L}(\mathfrak{a}_0\rho^R )~z\pa_s\pa_z u_p^0\Big), \pa_s^{k}u^R \Big\rangle_{X^{0}_{\mu,t}}\\
&\qquad+\Big\langle \pa_s^{k}\Big(\mathcal{L}(\pa_s\mathfrak{a}_0\rho^R )~z\pa_z u_p^0\Big), \pa_s^{k}u^R \Big\rangle_{X^{0}_{\mu,t}}-\Big\langle \pa_s^{k}\Big(\mathcal{L}(\mathfrak{a}_0\rho^R )~z\pa_z u_p^0\Big), \pa_s^{k}u^R \Big\rangle_{X^{0}_{\mu}}\Big|_{s=0}^{s=t}\\
&\eqdef \Big\langle \pa_s^{k}\Big(\mathcal{L}(\mathfrak{a}_0\rho^R )~z\pa_z u_p^0\Big), ~ \pa_s^{k+1}u^R \Big\rangle_{X^{0}_{\mu,t}}-\Big\langle \pa_s^{k}\Big(\mathcal{L}(\mathfrak{a}_0\rho^R )~z\pa_z u_p^0\Big), \pa_s^{k}u^R \Big\rangle_{X^{0}_{\mu}}\Big|_{s=0}^{s=t} +D_2^k+D_3^k\\
&\leq\Big\langle \pa_s^{k}\Big(\mathcal{L}(\mathfrak{a}_0\rho^R )~z\pa_z u_p^0\Big), ~ \pa_s^{k+1}u^R \Big\rangle_{X^{0}_{\mu,t}}-\Big\langle \pa_s^{k}\Big(\mathcal{L}(\mathfrak{a}_0\rho^R )~z\pa_z u_p^0\Big), \pa_s^{k}u^R \Big\rangle_{X^{0}_{\mu}} +C_0\e^4+D_2^k+D_3^k.
\end{align*}
Bringing the equation of $u^R$ in \eqref{eq: Error-(u,v,rho)} to the first term of the right hand side,  we get 
\beno
&&\Big\langle \pa_s^{k}(\mathcal{L}( \mathfrak{a}_0\rho^R)~z\pa_z u_p^0), ~ \pa_s^{k+1}u^R \Big\rangle_{X^{0}_{\mu,t}}
=-\Big\langle \pa_s^{k}\Big(\mathcal{L}( \mathfrak{a}_0 \rho^R )~z\pa_z u_p^0\Big), ~ \pa_s^{k}(v^R\pa_y u_p^0) \Big\rangle_{X^{0}_{\mu,t}}\\
&&\quad - \Big\langle \pa_s^{k}\Big(\mathcal{L}  (\mathfrak{a}_0\rho^R) ~z\pa_z u_p^0\Big), ~ \pa_s^{k}(\mathcal{N}_u+F_u) \Big\rangle_{X^{0}_{\mu,t}}+\e^2 \Big\langle \pa_s^{k}\Big(\mathcal{L}(\mathfrak{a}_0\rho^R)~z\pa_z u_p^0\Big), ~ \pa_s^{k}(\mathfrak{a}\tri u^R)\Big\rangle_{X^{0}_{\mu,t}}\\
&&\quad +\e^2 \Big\langle \pa_s^{k}\Big(\mathcal{L}(\mathfrak{a}_0\rho^R)~z\pa_z u_p^0\Big), ~ \pa_s^{k}\Big(\mathfrak{a}\pa_x d^R\Big)\Big\rangle_{X^{0}_{\mu,t}}\\
&&\eqdef -\Big\langle \pa_s^{k}\Big(\mathcal{L}(\mathfrak{a}_0\rho^R) ~z\pa_z u_p^0\Big), ~ \pa_s^{k}(v^R\pa_y u_p^0) \Big\rangle_{X^{0}_{\mu,t}}+D_4^k+D_5^k+D_6^k.
\eeno
 
Using the density equation again, we get 
\begin{align*}
&-\Big\langle \pa_s^{k}\Big(\mathcal{L}(\mathfrak{a}_0\rho^R) ~z\pa_z u_p^0\Big), ~ \pa_s^{k}(v^R\pa_y u_p^0) \Big\rangle_{X^{0}_{\mu,t}}\\
& =
-\Big\langle \pa_s^{k}\Big(\mathcal{L}(\mathfrak{a}_0\rho^R) ~z\pa_z u_p^0\Big), ~  \pa_s^{k}\Big(\mathcal{L}(\pa_y v^R)~z\pa_z u_p^0\Big)\Big\rangle_{X^{0}_{\mu,t}} \\
&= 
\Big\langle \pa_s^{k}\Big(\mathcal{L}(\mathfrak{a}_0\rho^R) ~z\pa_z u_p^0\Big), ~  \pa_s\pa_s^{k}(\mathcal{L}(\mathfrak{a}_0\rho^R) ~z\pa_z u_p^0)\Big\rangle_{X^{0}_{\mu,t}} \\
&\quad-\Big\langle \pa_s^{k}\Big(\mathcal{L}(\mathfrak{a}_0\rho^R) ~z\pa_z u_p^0\Big), ~ \pa_s^{k}(\mathcal{L}(\mathfrak{a}_0\rho^R)  z\pa_s\pa_z u_p^0)\Big\rangle_{X^{0}_{\mu,t}}\\
&\quad-\Big\langle \pa_s^{k}\Big(\mathcal{L}(\mathfrak{a}_0\rho^R) ~z\pa_z u_p^0\Big), ~ \pa_s^{k}(\mathcal{L}(\pa_s\mathfrak{a}_0\rho^R)  z\pa_z u_p^0)\Big\rangle_{X^{0}_{\mu,t}}\\
&\quad+\Big\langle \pa_s^{k}\Big(\mathcal{L}(\mathfrak{a}_0\rho^R) ~z\pa_z u_p^0\Big), ~  \pa_s^{k}\Big(\mathcal{L} (\mathfrak{a}_0\mathcal{N}_\rho+\mathfrak{a}_0F_\rho) ~z\pa_z u_p^0\Big)\Big\rangle_{X^{0}_{\mu,t}} \\
&= \f12 \big\|\pa_s^k\big(\mathcal{L}(\mathfrak{a}_0\rho^R) ~z\pa_z u_p^0\big)\big\|_{X^0_\mu}^2\Big|_{s=0}^{s=t}+D_7^k+D_8^k+D_9^k.
\end{align*}
 Summing up, we get by Lemma \ref{lem: product 1}, Lemma \ref{lem: app}, Lemma \ref{cor: est-a_0} and Lemma \ref{lem: Hardy inequality}  that
\begin{align}\label{eq: trouble u^R}
A_0= \sum_{k=0}^{10}\d^{2k}A_0^k
\leq& C_0\e^4+|D_0|+|D_1|+\cdots+|D_9|,
\end{align}
since $\pa_t^k \rho^R|_{t=0}=\pa_t^{k-1}R_\rho$ according to \eqref{initial: 6} and Lemma \ref{cor: est-a_0}.
Here $D_i=\sum_{k=0}^{10}\d^{2k} D_i^k$ and 
\begin{align*}
D_0^k= \Big\langle \pa_t^{k}\Big(\mathcal{L}(\mathfrak{a}_0\rho^R )~z\pa_z u_p^0\Big)(t,\cdot,\cdot), \pa_t^{k}u^R(t,\cdot, \cdot) \Big\rangle_{X^{0}_{\mu}}-\f12 \big\|\pa_t^k\big(\mathcal{L}(\mathfrak{a}_0\rho^R) ~z\pa_z u_p^0\big)(t,\cdot,\cdot)\big\|_{X^0_\mu}^2.
\end{align*}

Now we estimate $D_i, i=0,\cdots, 9$. \smallskip

\underline{Estimate of $D_0.$} Applying Young's inequality $ab\leq p a^2+\f{b^2}{4p}$ for any $p>0,$ we take $p=\f12-\f1 A $ to  find
\begin{align*}
|D_0|\leq& (\f{1}{2-\f{4}{A}}-\f12)\sum_{k=0}^{10}\d^{2k}\|\pa_t^{k}\Big(\mathcal{L}(\mathfrak{a}_0\rho^R )~z\pa_z u_p^0\Big)(t)\|_{X^0_\mu}^2+(\f12-\f1 A)\sum_{k=0}^{10}\d^{2k}\|\pa_t^{k}u^R(t)\|_{X^0_\mu}^2\\
\leq&\f{C_0}{A}\|\rho^R(t)\|_{X^{10}_\mu}^2+(\f12-\f1 A)\|u^R(t)\|_{X^{10}_\mu}^2,
\end{align*}
by using the fact $ (\f{1}{2-\f{4}{A}}-\f12)\leq \f{C_0}{A}$ for $A$ a large constant.

\underline{Estimate of $D_1.$} By Lemma \ref{lem: (F_u, F_v, F_rho)}, Lemma \ref{cor: est-a_0}, \eqref{est: int-f,g} and \eqref{est: u^R_L^2_t}, we have
 \begin{align*}
|D_1|\leq& C_0\int_0^t \|F_\rho(s)\|_{X^{10}_{\mu}}\|u^R(s)\|_{X^{10}_{\mu}}ds+C_0\|h^{\f12}\mathcal{N}_\rho\|_{\widetilde{L}^2(0,t; X^{10}_\mu)}\|h^{-\f12}u^R\|_{\widetilde{L}^2(0,t; X^{10}_\mu)}\\
\leq&C_0 \int_0^t\Big(\e^2+\|(\pa_x u^R,\pa_x\rho^R)(s)\|_{X^{10}_{\mu}}+\|\varphi \pa_y\rho^R(s)\|_{X^{10}_{\mu}}+\|(u^R, v^R, \rho^R)(s)\|_{X^{10}_{\mu}} \Big)\|u^R(s)\|_{X^{10}_{\mu}}ds\\
&\qquad+C_0\|h^{\f12}\mathcal{N}_\rho\|_{\widetilde{L}^2(0,t; X^{10}_\mu)}\|h^{-\f12}u^R\|_{\widetilde{L}^2(0,t; X^{10}_\mu)}\\
\leq& C_0\e^4+C_0\int_0^t \Big(h^{-1}(s,\mu)\|(u^R,\rho^R)(s)\|_{X^{10}_{\mu'}} +\|(u^R, v^R,\rho^R)(s)\|_{X^{10}_{\mu}} \Big)\|u^R(s)\|_{X^{10}_{\mu}} ds\\
&\qquad+ \d^3\|h^{\f12}\mathcal{N}_\rho\|_{\widetilde{L}^2(0,t; X^{10}_\mu)}^2+C\|h^{-\f12}u^R\|_{\widetilde{L}^2(0,t; X^{10}_\mu)}^2\\
 \leq&C_0\e^4+ \d^3\|h^{\f12}\mathcal{N}_\rho\|_{\widetilde{L}^2(0,t; X^{10}_\mu)}^2 + CA\la^{-1}h^{-\eta}(t,\mu)\sup_{s\in[0,t]} \e^{2} \mathcal{E}(s),
\end{align*}
 by using Young's inequality $ab\leq \d^3a^2+\f{1}{4\d^3} b^2.$\smallskip
 
 \underline{Estimates of $D_2, D_3.$}
Due to \eqref{initial: 6}, by Lemma \ref{lem: product 1}, Lemma \ref{cor: est-a_0}, Lemma \ref{lem: product 2} and Lemma \ref{lem: app},  we have
\begin{align*}
|D_2|+|D_3|\leq& C_0A\la^{-1}h^{-\eta}(t,\mu)\sup_{s\in[0,t]}\e^{2}\mathcal{E}(s).
\end{align*}
%\underline{Estimate of $D_2^k.$} It follows from Lemma \ref{lem: product 2},  Lemma \ref{lem: app}, Lemma \ref{cor: est-a_0} and \eqref{est: u^R_L^2_t} that
%\begin{align*}
%D_2^k\leq&C\big(\|(\rho^a-1)\|_{X^{10}_\mu}+\|\pa_y(\rho^a-1)\|_{X^{10}_\mu}\big)\| \pa_y v^R\|_{\widetilde{L}^2(0,t; X^{10}_\mu)}\|u^R\|_{\widetilde{L}^2(0,t; X^{10}_\mu)}\\
%\leq& C\la^{-1}h^{-\eta}(t,\mu)\sup_{s\in[0,t]}\sup_{\mu< \mu_0-\la s}\Big(h^{\eta}(s,\mu)\| u^R(s)\|_{X^{10}_{\mu}}^2\Big)+ \d\|\e \pa_y v^R\|_{\widetilde{L}^2(0,t; X^{10}_\mu)}^2.
%\end{align*}
%

\underline{Estimates of $D_4, D_7, D_8, D_9$}. Similar to $D_1$,  we can prove that
\begin{align*}
|D_4|
\leq& C_0\e^4+ \d^3\|h^{\f12}\mathcal{N}_u\|_{\widetilde{L}^2(0,t; X^{10}_\mu)}^2+ CA\la^{-1}h^{-\eta}(t,\mu)\sup_{s\in[0,t]}\e^{2}\mathcal{E}(s),
\end{align*}
and
\begin{align*}
|D_7|+|D_8|+|D_9|\leq&C_0\e^4+ \d^3\|h^{\f12}\mathcal{N}_\rho\|_{\widetilde{L}^2(0,t; X^{10}_\mu)}^2+ C\la^{-1}h^{-\eta}(t,\mu)\sup_{s\in[0,t]}\e^{2}\mathcal{E} (s).
\end{align*}

\underline{Estimate of $D_5.$} We divide $D_5^k$ into the following parts:
\begin{align*}
D_5^k=&\e^2 \sum_{\substack{ k_1+k_2=k\\
k_1\leq k_2}}\Big\langle \pa_s^{k}\Big(\mathcal{L} ( \mathfrak{a}_0\rho^R )~z\pa_z u_p^0\Big), ~ (\pa_s^{k_1}\mathfrak{a})(\pa_s^{k_2}\tri u^R)\Big\rangle_{X^{0}_{\mu,t}}\\
&\qquad+\e^2 \sum_{\substack{ k_1+k_2=k\\
k_1\geq k_2}}\Big\langle \pa_s^{k}\Big(\mathcal{L}   (\mathfrak{a}_0\rho^R )~z\pa_z u_p^0\Big), ~ (\pa_s^{k_1}(\mathfrak{a}-\mathfrak{a}_0))(\pa_s^{k_2}\tri u^R)\Big\rangle_{X^{0}_{\mu,t}}\\
&\qquad+\e^2 \sum_{\substack{ k_1+k_2=k\\
k_1\geq k_2}}\Big\langle \pa_s^{k}\Big(\mathcal{L}(\mathfrak{a}_0\rho^R) ~z\pa_z u_p^0\Big), ~ (\pa_s^{k_1}\mathfrak{a}_0)(\pa_s^{k_2}\tri u^R)\Big\rangle_{X^{0}_{\mu,t}}\\
=&D_{51}^k+D_{52}^k +D_{53}^k.
\end{align*}

$\bullet$ For $D_{51}^k,$  we get by integration by parts that
 \begin{align*}
 D_{51}^k=&-\e^2 \sum_{\substack{ k_1+k_2=k\\
k_1\leq k_2}}\Big\langle \pa_s^{k}\na\Big(\mathcal{L}(\mathfrak{a}_0\rho^R) ~z\pa_z u_p^0\Big), ~ (\pa_s^{k_1}\mathfrak{a})(\pa_s^{k_2}\na u^R)\Big\rangle_{X^{0}_{\mu,t}}\\
&-\e^2 \sum_{\substack{ k_1+k_2=k\\
k_1\leq k_2}}\Big\langle \pa_s^{k}\Big(\mathcal{L} (\mathfrak{a}_0\rho^R) ~z\pa_z u_p^0\Big), ~ (\pa_s^{k_1}\na\mathfrak{a})\cdot (\pa_s^{k_2}\na u^R)\Big\rangle_{X^{0}_{\mu,t}}.
 \end{align*}
 The boundary term vanishes due to $ \mathcal{L}(\mathfrak{a}_0\rho^R) ~z\pa_z u_p^0 \Big|_{y=0}=0.$
 
Along with facts 
 \begin{align*}
\e\pa_y \Big(\mathcal{L} (\mathfrak{a}_0\rho^R)~z\pa_z u_p^0\Big)=&-\e \Big(\f1{y^2}\int_\mathcal{C}  \mathfrak{a}_0\rho^R dy~z\pa_z u_p^0\Big)+\e \Big(\f{\mathfrak{a}_0\rho^R}{y}~z\pa_z u_p^0\Big)+\e \Big(\f1{y}\int_\mathcal{C}  \mathfrak{a}_0\rho^R dy~\pa_y(z\pa_z u_p^0)\Big)\\
\sim&\f{1}{y}\int_\mathcal{C}  \mathfrak{a}_0\rho^R dy+\mathfrak{a}_0\rho^R,\\
\e\pa_x \Big(\mathcal{L}(\mathfrak{a}_0\rho^R)~z\pa_z u_p^0\Big)\sim&\e\f1{y}\int_\mathcal{C} \pa_x(\mathfrak{a}_0 \rho^R) dy +\e\f1{y}\int_\mathcal{C}  \mathfrak{a}_0 \rho^R dy,
\end{align*}
 it follows from \eqref{assume: 1}, \eqref{eq:a-diff}, Lemma \ref{lem: product 2} and Lemma \ref{lem: app} that
\begin{align*}
\Big|\sum_{k=0}^{10}\d^{2k}D_{51}^k\Big|\leq&\|(\rho^R,\e\pa_x \rho^R)\|_{\widetilde{L}^2(0,t; X^{10}_\mu)}\big(C_0+C\|\mathfrak{a}-\mathfrak{a}_0\|_{X^7_\mu}+C\|\pa_y(\mathfrak{a}-\mathfrak{a}_0)\|_{X^7_\mu}\big)\|\e\na u^R\|_{\widetilde{L}^2(0,t; X^{10}_\mu)}\\
\leq&(C_0+C\e^\f12) \|(\rho^R,\e\pa_x \rho^R)\|_{\widetilde{L}^2(0,t; X^{10}_\mu)}\|\e\na u^R\|_{\widetilde{L}^2(0,t; X^{10}_\mu)}\\
\leq& \d \|\e\na u^R\|_{\widetilde{L}^2(0,t; X^{10}_\mu)}^2+ C\|(\rho^R,\e\pa_x \rho^R)\|_{\widetilde{L}^2(0,t; X^{10}_\mu)}^2.
\end{align*}
\smallskip

$\bullet$ For $D_{52}^k,$ according to $z=\f{y}{\e},$ we write it as
\begin{align*}
D_{52}^k=&\e \sum_{\substack{ k_1+k_2=k\\
k_1\geq k_2}}\Big\langle \pa_s^{k}\Big(\f1{\varphi}\int_\mathcal{C}  \rho^R dy~\pa_z u_p^0\Big), ~ (\pa_s^{k_1}(\mathfrak{a}-\mathfrak{a}_0))(\pa_s^{k_2}\varphi \tri u^R)\Big\rangle_{X^{0}_{\mu,t}}.
\end{align*}
A similar argument leading to Lemma \ref{lem: product 2}  yields  
\begin{align*}
\Big|\sum_{k=0}^{10}\d^{2k}D_{52}^k\Big|\leq&C\e\|\rho^R\|_{\widetilde{L}^2(0,t; X^{10}_\mu)}\|\mathfrak{a}-\mathfrak{a}_0\|_{\widetilde{L}^2(0,t; X^{10}_\mu)}\sup_{s\in[0,t]}\big(\|\varphi \tri u^R\|_{X^6_\mu}+\|\pa_y(\varphi \tri u^R)\|_{X^6_\mu}\big)\\
\leq&C\sup_{s\in[0,t]}\sup_{\mu<\mu_0-\la s}\big(\| \e u^R\|_{X^9_\mu}+\|\e\pa_y^2 u^R\|_{X^7_\mu}\big)\big(\|\rho^R\|_{\widetilde{L}^2(0,t; X^{10}_\mu)}^2+\|\mathfrak{a}-\mathfrak{a}_0\|_{\widetilde{L}^2(0,t; X^{10}_\mu)}^2\big)\\
\leq& C\e^\f12\big(\|\rho^R\|_{\widetilde{L}^2(0,t; X^{10}_\mu)}^2+\|\mathfrak{a}-\mathfrak{a}_0\|_{\widetilde{L}^2(0,t; X^{10}_\mu)}^2\big),
\end{align*}
by using \eqref{assume: 1} in the last step.

$\bullet$ For $D_{53}^k$, similar to the proof of $D_{51}^k$, we have
\begin{align*}
\Big|\sum_{k=0}^{10}\d^{2k}D_{53}^k\Big|\leq& \d \|\e\na u^R\|_{\widetilde{L}^2(0,t; X^{10}_\mu)}^2+ C\|(\rho^R,\e\pa_x \rho^R)\|_{\widetilde{L}^2(0,t; X^{10}_\mu)}^2.
\end{align*}

Putting the estimates of $D_{51}^k-D_{53}^k$ together, we arrive at

\begin{align*}
|D_5|\leq &C_0\d \|\e\na u^R\|_{\widetilde{L}^2(0,t; X^{10}_\mu)}^2+C\Big(\|(\rho^R,\e\pa_x \rho^R)\|_{\widetilde{L}^2(0,t; X^{10}_\mu)}^2+\|\mathfrak{a}-\mathfrak{a}_0\|_{\widetilde{L}^2(0,t; X^{10}_\mu)}^2\Big)\\
\leq&C_0\d h^{-\eta}(t,\mu) \e^2\mathcal{D} (t)+C\la^{-1}h^{-\eta}(t,\mu)\sup_{s\in[0,t]}\e^2\mathcal{E}(s).
\end{align*}
Here we used  \eqref{eq:a-diff}.

\underline{Estimate of $D_6.$} Following the same process of $D_5,$ we integrate by parts of $\pa_x$ and use $d^R\sim \na(u^R, v^R)$ to arrive at
\begin{align*}
|D_6|\leq &C_0\d h^{-\eta}(t,\mu) \e^2\mathcal{D}(t)+C\la^{-1}h^{-\eta}(t,\mu)\sup_{s\in[0,t]}\e^2\mathcal{ E}(s).
\end{align*}

Summing up the estimates of $D_i (i=0,\cdots,9)$ , we finally obtain
 \begin{align}\label{est: A_0}
|A_0|
\leq&\f{C_0}{A}\|\rho^R(t)\|_{X^{10}_\mu}^2+(\f12-\f1 A)\|u^R(t)\|_{X^{10}_\mu}^2+C_0\e^4+C_0\d^3\|h^{\f12}(\mathcal{N}_u,\mathcal{ N}_\rho)\|_{\widetilde{L}^2(0,t; X^{10}_\mu)}^2\\
\nonumber
&+ C_0\d h^{-\eta}(t,\mu) \e^2 \mathcal{D}(t)+CA\la^{-1}h^{-\eta}(t,\mu)\sup_{s\in[0,t]}\e^2 \mathcal{E}(s).
 \end{align}

Putting \eqref{est: A_1} and \eqref{est: A_0} together and multiplying $h^{\eta}(t,\mu)$ on both sides, we get the desired result.
\end{proof}

\subsection{Energy estimate of $\|(u^R, v^R, \rho^R)\|_{X^{10}_\mu}^2$.}

Before we present the main results of this subsection, we give some estimates  of $Z\mathfrak{a}$ with lower order derivative. Thanks to Lemma \ref{cor: est-a_0} and the definition of $Z$ (constant $\d$ in it) and $\widetilde{Z}$(constant $\kappa\gg\d^\f12$ in it), we have
\begin{align}
&\|Z\mathfrak{a}_0\|_{X^6_\mu}+\|\na Z\mathfrak{a}_0\|_{X^6_\mu}\leq C_0\d^\f12,\label{est:a-Low}\\
&\|\mathfrak{a}-\mathfrak{a}_0\|_{X^7_\mu}\leq C_0 \e^\f32,\quad \|\na(\mathfrak{a}-\mathfrak{a}_0)\|_{X^7_\mu}\leq C_0\e^\f12,\label{est:pa-Low}
\end{align}
by applying \eqref{assume: 1} and \eqref{eq:a-diff}.

\medskip

Now, we are in a position to give the main result of this section, which gives the energy estimate of $(u^R, v^R, \rho^R).$ For this, we need to use Lemma \ref{lem: est: A} to deal with the troubling term and use the cancellation between the equation of $v^R$ and the equation of $\rho^R$.
\begin{proposition}\label{pro: Zu^R-H}
Under the assumption \eqref{assume: 1},
there exist $\delta_0>0$ and $A_0>1$ such that for any $t \in[0,T]$, $\delta\in(0, \delta_0)$ and $A\geq A_0$, it holds that
\begin{align*}
&\sup_{\mu<\mu_0-\la t}h^\eta(t,\mu)\Big( \|(A^{-\f12} u^R, \f12v^R, \f{c_0}{2}\rho^R)(t)\|_{X^{10}_{\mu}}^2 +\big(c_0-(C_0\d^\f12+C\e^\f12)\big) \|\e(\na u^R, \na v^R)\|^2_{\tilde{L}^2(0,t; X^{10}_{\mu})}\Big)\\
&\leq  C_0\e^4+C_0 \d \e^2 \mathcal{D}(t)+\Big(\f{C_0}{A} +\f{CA}{\la}  \Big)\sup_{s\in[0,t]}\e^2 \mathcal{E}(s)\\
&\quad+C_0\d^3 \sup_{\mu<\mu_0-\la t}\Big(h^\eta(t,\mu)\|h^{\f12}(\mathcal{N}_u,\mathcal{N}_v,\mathcal{ N}_\rho)\|_{\widetilde{L}^2(0,t; X^{10}_\mu)}^2\Big) +C_0\d\la^{-1}\sup_{s\in[0, t]}\sup_{\mu< \mu_0-\la s}\Big(h^{\eta}(s,\mu)\|\mathcal{N}_\rho(s)\|_{X^{9}_\mu}^2\Big).
\end{align*}
%Here $d^R=\pa_x u^R+\pa_y v^R.$

\end{proposition}
\begin{proof}
Taking $Z^\al$ on the first three equations in \eqref{eq: Error-(u,v,rho)-1}, we take the inner product $X^{0}_{\mu,t}$  on  both sides with $(\mathfrak{a}_0Z^\al \rho^R, Z^\al u^R,Z^\al v^R)$, then we sum $\sum_{|\al|=0}^{10}$ and use Lemma \ref{cor: est-a_0}, Lemma \ref{est: product 1.2} to obtain
\begin{align}\label{est: ||u^R||_X^10}
&\f12\|(u^R, v^R)\|_{X^{10}_{\mu}}^2+\sum_{|\al|=0}^{10}\langle Z^\al\pa_s\rho^R, \mathfrak{a}_0Z^\al\rho^R\rangle_{X^{0}_{\mu,t}}+\langle v^R\pa_y u_p^0, u^R\rangle_{X^{10}_{\mu,t}}\\
\nonumber
&\quad-\e^2 \langle \mathfrak{a}(\tri u^R+\pa_x d^R), u^R\rangle_{X^{10}_{\mu,t}}-\e^2 \langle \mathfrak{a}(\tri v^R+\pa_y d^R), v^R\rangle_{X^{10}_{\mu,t}}\\
\nonumber
&\quad+ \langle \pa_y \rho^R, v^R\rangle_{X^{10}_{\mu,t}}+ \sum_{|\al|=0}^{10}\langle Z^\al(\rho^a\pa_y v^R), \mathfrak{a}_0Z^\al \rho^R\rangle_{X^{0}_{\mu,t}}\\
\nonumber
&\leq C_0\e^4+\d^3\|h^{\f12}(\mathcal{N}_u,\mathcal{N}_v, \mathcal{N}_\rho)\|_{\widetilde{L}^2(0,t;X^{10}_\mu)}^2+C\|h^{-\f12}(u^R, v^R, \rho^R)\|_{\widetilde{L}^2(0,t;X^{10}_\mu)}^2\\
\nonumber
&\quad+C_0\int_0^t\| (F_u, F_v, F_\rho)(s)\|_{X^{10}_{\mu}}\|(u^R, v^R, \rho^R)(s)\|_{X^{10}_{\mu}}ds\\
\nonumber
&\eqdef C_0\e^4+\d^3\|h^{\f12}(\mathcal{N}_u,\mathcal{N}_v, \mathcal{N}_\rho)\|_{\widetilde{L}^2(0,t;X^{10}_\mu)}^2+I_1+I_2.
\end{align}

\medskip

{\bf $\bullet$ The right hand of \eqref{est: ||u^R||_X^10}.}

\underline{  $I_1.$} By \eqref{est: u^R_L^2_t} , we have
\begin{align*}
I_1\leq&
C\la^{-1}h^{-\eta}(t,\mu)\sup_{s\in[0,t]}\sup_{\mu< \mu_0-\la s}\Big(h^{\eta}(s,\mu)\|(u^R, v^R, \rho^R)(s)\|_{X^{10}_\mu}^2\Big)\\
\leq&CA\la^{-1}h^{-\eta}(t,\mu)\sup_{s\in[0,t]}\e^2 \mathcal{E}(s).
\end{align*}

\underline{$I_2.$} By Lemma \ref{lem: (F_u, F_v, F_rho)}, we get
\begin{align*}
I_2\leq& C_0\int_0^t \big( \e^2+\|(\pa_x, \varphi\pa_y) (u^R, v^R, \rho^R)(s)\|_{X^{10}_{\mu}}+\|(u^R, v^R, \rho^R)(s)\|_{X^{10}_{\mu}} \big)\| (u^R, v^R, \rho^R)(s)\|_{X^{10}_{\mu}} ds.
\end{align*}
Following the argument in \eqref{est: integral-gain} and using Lemma \ref{lem: analyticity recovery}, we get 
\beno
I_2\leq C_0A\la^{-1}h^{-\eta}(t,\mu)\sup_{s\in[0,t]}\e^2 \mathcal{E}(s)+C_0\e^4.
\eeno

Summing up, we obtain
\begin{align*}
I_1+I_2\leq&CA\la^{-1}h^{-\eta}(t,\mu)\sup_{s\in[0,t]}\e^2 \mathcal{E}(s)+C_0\e^4.
\end{align*}

{\bf $\bullet$ The left hand side of \eqref{est: ||u^R||_X^10}. }

\underline{$\sum_{|\al|=0}^{10}\langle Z^\al\pa_s\rho^R, \mathfrak{a}_0Z^\al\rho^R\rangle_{X^{0}_{\mu,t}}$.} We write this term as
\begin{align*}
&\sum_{|\al|=0}^{10}\langle Z^\al\pa_s\rho^R, \mathfrak{a}_0Z^{\al}\rho^R\rangle_{X^{0}_{\mu,t}}\\
=&\f12 \sum_{|\al|=0}^{10}\|\sqrt{\mathfrak{a}_0} Z^\al \rho^R\|_{X^0_\mu}^2\Big|_{s=0}^{s=t}-\f12\sum_{|\al|=0}^{10}\langle Z^\al\rho^R, \pa_s\mathfrak{a}_0Z^{\al}\rho^R\rangle_{X^{0}_{\mu,t}}\\
\geq& \f{c_0}2\|\rho^R(t)\|_{X^{10}_\mu}^2-C_0\e^4-\f12\sum_{|\al|=0}^{10}\langle Z^\al\rho^R, \pa_s\mathfrak{a}_0Z^{\al}\rho^R\rangle_{X^{0}_{\mu,t}},
\end{align*}
where we used $\mathfrak{a}_0\geq c_0.$  
According to Lemma \ref{lem: app}, the last term above is bounded upper by
\begin{align*}
C_0\|\rho^R\|_{\widetilde{L}^2(0,t; X^{10}_\mu)}^2\leq C_0\la^{-1} h^{-\eta}(t, \mu) \sup_{s\in[0,t]}\e^2 \mathcal{E}(s),
\end{align*}
which implies
\begin{align}\label{est: rho^R-time}
\sum_{|\al|=0}^{10}\langle Z^\al\pa_s\rho^R, \mathfrak{a}_0Z^{\al}\rho^R\rangle_{X^{0}_{\mu,t}}\geq&\f{c_0}2\|\rho^R(t)\|_{X^{10}_\mu}^2-C_0\la^{-1} h^{-\eta}(t, \mu) \sup_{s\in[0,t]}\e^2 \mathcal{E}(s)-C_0\e^4.
\end{align}

\underline{$-\e^2 \langle \mathfrak{a}(\tri u^R+\pa_x d^R), u^R\rangle_{X^{10}_{\mu,t}}-\e^2 \langle \mathfrak{a}(\tri v^R+\pa_y d^R), v^R\rangle_{X^{10}_{\mu,t}}$.} For convenience, we denote this term by $D,$ and introduce $U^R=(u^R, v^R)$ which holds that $\dv U^R=d^R.$ 
Thus, $D$ can be written as 
\begin{align*}
D=&-\e^2 \langle \mathfrak{a}(\tri U^R+\na d^R), U^R\rangle_{X^{10}_{\mu,t}}=-\e^2 \langle \mathfrak{a}\dv (\na U^R+ d^RI_2), U^R\rangle_{X^{10}_{\mu,t}},
\end{align*}
where $I_2$ is the $2\times 2$ identity matrix. As a result,
we divide  $D$ into three parts:
\begin{align*}
D=&-\e^2\sum_{\substack{
|\beta|+|\gamma|=|\al|=0\\
|\beta|\leq |\gamma|}}^{10} \langle Z^{\beta}\mathfrak{a}~Z^{\gamma}\dv (\na U^R+ d^RI_2), Z^\al U^R\rangle_{X^{0}_{\mu,t}}\\
&-\e^2\sum_{\substack{
|\beta|+|\gamma|=|\al|=0\\
|\beta|\geq |\gamma|,|\beta|\geq1}}^{10} \langle Z^{\beta}(\mathfrak{a}-\mathfrak{a}_0)~Z^{\gamma}\dv (\na U^R+ d^RI_2), Z^\al U^R\rangle_{X^{0}_{\mu,t}}\\
&-\e^2\sum_{\substack{
|\beta|+|\gamma|=|\al|=0\\
|\beta|\geq |\gamma|,|\beta|\geq1}}^{10} \langle Z^{\beta}\mathfrak{a}_0~Z^{\gamma}\dv (\na U^R+ d^RI_2), Z^\al U^R\rangle_{X^{0}_{\mu,t}}\\
=&D_1+D_2+D_3.
\end{align*}

For $D_1,$ due to $Z^\al U^R|_{y=0}=0, $ we get by integration by parts that
\begin{align*}
D_1=&\e^2\sum_{\substack{
|\al|=0}}^{10} \langle \mathfrak{a}~Z^{\al} (\na U^R+ d^RI_2), Z^\al \na U^R\rangle_{X^{0}_{\mu,t}}\\
&+\e^2\sum_{\substack{
|\beta|+|\gamma|=|\al|=0\\
1\leq|\beta|\leq |\gamma|}}^{10} \langle Z^{\beta}\mathfrak{a}~Z^{\gamma} (\na U^R+d^RI_2), Z^\al \na U^R\rangle_{X^{0}_{\mu,t}}\\
&+\e^2\sum_{\substack{
|\beta|+|\gamma|=|\al|=0\\
|\beta|\leq |\gamma|}}^{10} \langle Z^{\beta}\na\mathfrak{a}\cdot Z^{\gamma}(\na U^R+ d^RI_2), Z^\al  U^R\rangle_{X^{0}_{\mu,t}}\\
&-\e^2\sum_{\substack{
|\beta|+|\gamma|=|\al|=1\\
|\beta|\leq |\gamma|}}^{10} \langle Z^{\beta}\mathfrak{a}~[ Z^{\gamma},\pa_y] (\pa_y U^R+ d^RI_2), Z^\al  U^R\rangle_{X^{0}_{\mu,t}}\\
&+\e^2\sum_{\substack{
|\beta|+|\gamma|=|\al|=1\\
|\beta|\leq |\gamma|}}^{10} \langle [\pa_y,Z^{\beta}]\mathfrak{a}~Z^{\gamma}(\pa_y U^R+ d^RI_2), Z^\al  U^R\rangle_{X^{0}_{\mu,t}}\\
&+\e^2\sum_{\substack{
|\beta|+|\gamma|=|\al|=1\\
|\beta|\leq |\gamma|}}^{10} \langle Z^{\beta}\mathfrak{a}~Z^\gamma (\pa_y U^R+ d^R I_2), [\pa_y,Z^\al]  U^R\rangle_{X^{0}_{\mu,t}}\\
=&D_1^0+\cdots+D_1^5.
\end{align*}
The term $D_1^0$ is the main term in $D.$  We use $\mathfrak{a}\geq c_0>0$ and the fact $\dv U^R=d^R$ to get 
\begin{align*}
D_1^0\geq c_0\|\e(\na U^R, d^R)\|_{\widetilde{L}^2(0, t; X^{10}_\mu)}^2\geq  c_0\|\e \na U^R\|_{\widetilde{L}^2(0, t; X^{10}_\mu)}^2.
\end{align*}
For $D_1^1, $ due to $|\gamma|\geq|\beta|\geq1,$ we get by Lemma \ref{lem: product 2} and \eqref{est:a-Low}-\eqref{est:pa-Low}  that 
\begin{align*}
|D_1^1|\leq&(C_0\|Z\mathfrak{a}_0\|_{X^6_\mu}+C_0\|\pa_yZ\mathfrak{a}_0\|_{X^6_\mu}+C\|Z(\mathfrak{a}-\mathfrak{a}_0)\|_{X^6_\mu}+C\|\pa_yZ(\mathfrak{a}-\mathfrak{a}_0)\|_{X^6_\mu})\|\e\na U^R \|_{\widetilde{L}^2(0, t; X^{10}_\mu)}\\
&\times\big(\|\e\na U^R \|_{\widetilde{L}^2(0, t; X^{10}_\mu)}+\|\e d^R \|_{\widetilde{L}^2(0, t; X^{10}_\mu)}\big)\\
\leq&(C_0\d^\f12+C\e^\f12)\|\e\na U^R \|_{\widetilde{L}^2(0, t; X^{10}_\mu)}^2,
\end{align*}
by using the fact 
\begin{align}\label{est: phi^R}
\|d^R \|_{\widetilde{L}^2(0, t; X^{10}_\mu)}\leq 2\|\na U^R \|_{\widetilde{L}^2(0, t; X^{10}_\mu)}.
\end{align}
For $D_1^2,$ we use Lemma \ref{lem: product 2}, Lemma \ref{cor: est-a_0} and \eqref{assume: 1}, \eqref{eq:a-diff} again to deduce
\begin{align*}
|D_1^2|\leq& (C_0\|\na\mathfrak{a}_0\|_{X^7_\mu}+C\|\na(\mathfrak{a}-\mathfrak{a}_0)\|_{X^7_\mu})(\|\e\na U^R \|_{\widetilde{L}^2(0, t; X^{10}_\mu)}+\|\e d^R \|_{\widetilde{L}^2(0, t; X^{10}_\mu)})\\
&\times \big(\d^{-1}\|\e U^R \|_{\widetilde{L}^2(0, t; X^{10}_\mu)}+\d\|\e\na U^R \|_{\widetilde{L}^2(0, t; X^{10}_\mu)}\big)\\
\leq&(C_0+C\e^\f12)\big(\d^{-3}\| \e U^R \|_{\widetilde{L}^2(0, t; X^{10}_\mu)}^2+\d\|\e\na U^R \|_{\widetilde{L}^2(0, t; X^{10}_\mu)}^2\big)\\
\leq& (C_0\d+C\e^\f12)\|\e\na U^R \|_{\widetilde{L}^2(0, t; X^{10}_\mu)}^2+C\| \e U^R \|_{\widetilde{L}^2(0, t; X^{10}_\mu)}^2.
\end{align*}
For $D_1^3,$ we compute the commutator first. Recalling the definition of $Z^\al=(\d\pa_x) ^{\al_1}(\d\varphi\pa_y)^{\al_2}(\d\pa_t)^{\al_0}$, we have $[Z^\al,\pa_y]=0$ when $\al_2=0$. Thus, we only need to focus on the commutator $[Z_2^{\al_2},\pa_y]$. A direct calculation gives
\begin{align}\label{est: commutator}
[Z^\al,\pa_y]f=-\d|\al|\varphi' Z^{\al-e_2}\pa_y f,\end{align}
where $e_2=(0,1,0)$ and $|\varphi'|\leq 2.$ Thus, by Lemma \ref{cor: est-a_0} and \eqref{assume: 1}, \eqref{eq:a-diff}, \eqref{est: phi^R}, we get
\begin{align*}
D_1^3=&-\e^2\sum_{\substack{
|\beta|+|\gamma|=|\al|=1\\
|\beta|\leq |\gamma|}}^{10} \langle Z^{\beta}\mathfrak{a}~\d |\gamma|\varphi' \varphi Z^{\gamma-e_2}\pa_y(\pa_y U^R+d^R I_2), \varphi^{-1}Z^{\al} U^R\rangle_{X^{0}_{\mu,t}}\\
\leq& \d\big(C_0+C\|\mathfrak{a}-\mathfrak{a}_0\|_{X^7_\mu}+C\|\pa_y(\mathfrak{a}-\mathfrak{a}_0)\|_{X^7_\mu}\big)\|\e\pa_y U^R \|_{\widetilde{L}^2(0, t; X^{10}_\mu)}\|\e\pa_y U^R \|_{\widetilde{L}^2(0, t; X^{9}_\mu)}\\
\leq&(C_0\d+C\e^\f12)\|\e\pa_y U^R \|_{\widetilde{L}^2(0, t; X^{10}_\mu)}^2,
\end{align*}
where we use the fact $\varphi^{-1}Z^{\al} U^R=\d Z^{\al-e_2} \pa_yU^R$ for $\al_2\geq1.$
Similarly, we have
\begin{align*}
|D_1^4|+|D_1^5|\leq  (C_0\d+C\e^\f12)\|\e\pa_y U^R \|_{\widetilde{L}^2(0, t; X^{10}_\mu)}^2.
\end{align*}
Summing up the estimates in $D_1^0-D_1^5$, we arrive at
\begin{align}\label{est: D_1}
|D_1|\geq&  \Big(c_0-(C_0\d^\f12+C\e^\f12)\Big)\|\e\na U^R \|_{\widetilde{L}^2(0, t; X^{10}_\mu)}^2-C\|\e U^R \|_{\widetilde{L}^2(0, t; X^{10}_\mu)}^2\\
\nonumber
\geq&\Big(c_0-(C_0\d^\f12+C\e^\f12)\Big)\|\e\na U^R \|_{\widetilde{L}^2(0, t; X^{10}_\mu)}^2-CA \la^{-1}h^{-\eta}(t,\mu)\sup_{s\in[0,t]}\e^2 \mathcal{E}(s).
\end{align}

For $D_2,$ we get by Lemma \ref{lem: product 2} and \eqref{eq:a-diff} that
\begin{align}\label{est: D_2}
\nonumber
|D_2|\leq&C\|\mathfrak{a}-\mathfrak{a}_0\|_{\widetilde{L}^2(0, t; X^{10}_\mu)}\sup_{s\in[0,t]}\|\e\dv(\na U^R+d^R)(s)\|_{X^7_\mu}\big(\|\e\na U^R\|_{\widetilde{L}^2(0,t; X^{10}_\mu)}+\| \e U^R\|_{\widetilde{L}^2(0,t; X^{10}_\mu)}\big)\\
\leq&\d\|\e\na U^R\|_{\widetilde{L}^2(0,t; X^{10}_\mu)}^2+C\big(\|\e U^R\|_{\widetilde{L}^2(0,t; X^{10}_\mu)}^2+\|\rho^R\|_{\widetilde{L}^2(0,t; X^{10}_\mu)}^2\big)\\
\nonumber
\leq&C_0\d \|\e\na U^R \|_{\widetilde{L}^2(0, t; X^{10}_\mu)}^2+CA\la^{-1}h^{-\eta}(t,\mu)\sup_{s\in[0,t]}\e^2 \mathcal{E}(s),
\end{align}
where we use $\|\e\dv(\na U^R+d^R)(s)\|_{X^7_\mu}\leq 4\e^\f12,$ according to \eqref{assume: 1}.

For $D_3,$ using a similar argument in $D_1$ and estimate \eqref{est:a-Low}, we have 
\begin{align}\label{est: D_3}
|D_3|\leq&C_0\d \|\e\na U^R \|_{\widetilde{L}^2(0, t; X^{10}_\mu)}^2+CA\la^{-1}h^{-\eta}(t,\mu)\sup_{s\in[0,t]}\e^2 \mathcal{E}(s).
\end{align}

As a result, we combine \eqref{est: D_1}, \eqref{est: D_2} and \eqref{est: D_3} together to obtain
\begin{align}\label{est: D}
|D|\geq&\Big(c_0-(C_0\d+C\e^\f12)\Big)\|\e\na U^R \|_{\widetilde{L}^2(0, t; X^{10}_\mu)}^2-CA\la^{-1}h^{-\eta}(t,\mu)\sup_{s\in[0,t]}\e^2 \mathcal{E}(s).
\end{align}

\underline{$\langle \pa_y \rho^R, v^R\rangle_{X^{10}_{\mu,t}} + \sum_{|\al|=0}^{10}\langle Z^\al(\rho^a\pa_y v^R), \mathfrak{a}_0Z^\al \rho^R\rangle_{X^{0}_{\mu,t}}.$}
Due to $v^R|_{y=0},$ by integration by parts, \eqref{est:a-Low}, \eqref{est: commutator} and \eqref{est: pa_y v^R}, we infer that
\begin{align}\label{est: pressure-cancellation}
&\langle \pa_y \rho^R, v^R\rangle_{X^{10}_{\mu,t}}+ \sum_{|\al|=0}^{10}\langle Z^\al(\rho^a\pa_y v^R), \mathfrak{a}_0Z^\al \rho^R\rangle_{X^{0}_{\mu,t}}\\
=
\nonumber
&\sum_{|\al|=1}^{10}\big\langle [Z^\al,\pa_y] v^R, Z^\al \rho^R\big\rangle_{X^{0}_{\mu,t}}+ \sum_{|\al|=1}^{10}\big\langle\varphi^{-1} Z^\al v^R, \varphi[Z^\al,\pa_y]\rho^R\big\rangle_{X^{0}_{\mu,t}}\\
\nonumber
&+\sum_{|\al|=0}^{10}\langle [Z^\al,\rho^a ]\pa_y v^R,  \mathfrak{a}_0Z^\al \rho^R\rangle_{X^{0}_{\mu,t}}\\
\nonumber
\leq&C_0\d^\f12\int_0^t\|\pa_y v^R(s)\|_{X^{9}_{\mu}}\|\rho^R(s)\|_{X^{10}_{\mu}}ds\\
\nonumber
\leq&\d^\f12\int_0^t\big(C_0\e^2+C\|(\rho^R, u^R, v^R)(s)\|_{X^{10}_\mu}+C_0\|\mathcal{N}_\rho(s)\|_{X^9_\mu}\big)\|\rho^R(s)\|_{X^{10}_{\mu}}ds\\
\nonumber
\leq&C_0\e^4+CA\la^{-1}h^{-\eta}(t,\mu)\sup_{s\in[0,t]}\e^2 \mathcal{E}(s)\\
\nonumber
&+C_0\d\la^{-1}h^{-\eta}(t,\mu)\sup_{s\in[0,t]}\sup_{\mu< \mu_0-\la s}\Big(h^{\eta}(s,\mu)\|\mathcal{N}_\rho(s)\|_{X^{9}_\mu}^2\Big).
\end{align}

In the end, multiplying $h^{\eta}(t,\mu)$ on both sides of \eqref{est: ||u^R||_X^10} and substituting above estimates into it, we get by 
 Lemma \ref{lem: est: A} that
\begin{align*}
& h^{\eta}(t,\mu)\Big(\|(A^{-\f12}u^R, \f12v^R, \f{c_0}{2}\rho^R)(t)\|_{X^{10}_{\mu}}^2+\Big(c_0-(C_0\d^\f12+C\e^\f12)\Big)\|\e(\na u^R, \na v^R) \|_{\widetilde{L}^2(0, t; X^{10}_\mu)}^2\Big)\\
&\leq  C_0\e^4+\Big(\f{C_0}{A} +\f{CA}{\la}  \Big)\sup_{s\in[0,t]}\e^2 \mathcal{E}(s)+C_0\d^3 h^{\eta}(t,\mu)\|h^{\f12}(\mathcal{N}_u,\mathcal{ N}_v,\mathcal{ N}_\rho)\|_{\widetilde{L}^2(0,t; X^{10}_\mu)}^2\\
&\qquad+C_0\d\la^{-1}\sup_{s\in[0, t]}\sup_{\mu< \mu_0-\la s}\Big(h^{\eta}(s,\mu)\|\mathcal{N}_\rho(s)\|_{X^{9}_\mu}^2\Big)+C_0 \d \e^2 \mathcal{D}(t),
\end{align*}
which implies our result.

\end{proof}

\medskip

Next, we estimate the second term in $\mathcal{E}(t)$, i.e. 
\beno 
\sup_{\mu< \mu_0-\la t}h^{\eta+1}(t,\mu)\|\e\pa_x\rho^R\|_{X^{10}_{\mu}}^2,
\eeno 
which appears in the estimates of $D_{51}^k$ in Lemma \ref{lem: est: A}.

 \begin{proposition}\label{pro: e pa_tpa_x rho^R}
There exists $\delta_0>0$ such that for any $t \in[0,T]$ and $\delta\in(0, \delta_0)$, it holds that
\begin{align*}
&\sup_{\mu< \mu_0-\la t}\Big(h^{\eta+1}(t,\mu)\|\e\pa_x\rho^R(t)\|_{X^{10}_\mu}^2\Big)\\
&\leq C_0\e^6+C_0\d\e^2 \mathcal{D}(t) +C\la^{-1} \sup_{s\in[0,t]}\e^2 \mathcal{E}(s) +C_0\d^3\sup_{\mu< \mu_0-\la t}\Big(  h^{1+\eta}(t,\mu)\|h^{\f12}\e \pa_x\mathcal{N}_\rho\|_{\widetilde{L}^2(0,t; X^{10}_\mu)}^2\Big).
\end{align*}

\end{proposition}

\begin{proof}
Recall that $\rho^R$ satisfies the equation
\begin{align*}
\pa_t \rho^R+\rho^a \pa_y v^R = -\mathcal{N}_\rho-F_{\rho}.
\end{align*}
Take $\e\pa_x$ to the above equation to yield
\begin{align*}
\pa_t (\e\pa_x\rho^R)= -\e\pa_x(\mathcal{N}_\rho+\rho^a \pa_y v^R)-\e\pa_xF_{\rho}.
\end{align*}
Taking the inner product $X^{10}_{\mu,t}$ with $\e\pa_x\rho^R$ and using Lemma \ref{cor: est-a_0}, we obtain
\begin{align}\label{eq: ||e pa_x rho^R||_X^10}
\f12\|\e\pa_x\rho^R\|_{X^{10}_\mu}^2\leq& C_0\e^6+
\d^3 \|h^{\f12}\e \pa_x\mathcal{N}_\rho\|_{\widetilde{L}^2(0,t; X^{10}_\mu)}^2+C\|h^{-\f12}\e\pa_x\rho^R\|_{\widetilde{L}^2(0,t; X^{10}_\mu)}^2\\
\nonumber
&+\d\|   h^{\f12} \e\pa_x(\rho^a\pa_y v^R)\|_{\widetilde{L}^2(0,t; X^{10}_\mu)}^2
+\int_0^t \|  \e \pa_x F_\rho^1\|_{X^{10}_{\mu}}  \|    \e\pa_x\rho^R\|_{X^{10}_{\mu}}ds\\
\nonumber
&+\|\e\pa_x F_\rho^2\|_{\widetilde{L}^2(0,t; X^{10}_\mu)}\|\e\pa_x \rho^R\|_{\widetilde{L}^2(0,t; X^{10}_\mu)}\\
\nonumber
\eqdef & C_0\e^6+\d^3  \| h^{\f12}\e\pa_x\mathcal{N}_\rho\|_{\widetilde{L}^2(0,t; X^{10}_\mu)}^2+\sum_{i=1}^4 II_i,
\end{align}
where 
\beno
F^1_\rho &=& u^a\pa_x \rho^R+v^a\pa_y \rho^R+\rho^R(\pa_x u^a+\pa_y v^a), \\
F^2_\rho &=& u^R\pa_x \rho^a+v^R\pa_y \rho^a. 
\eeno

Next we give the estimates of $II_i~(i=1,\cdots 4)$ one by one.\smallskip

\underline{Estimate of $II_1.$} By the fact \eqref{inequ: norm 2}, we have
 \begin{align*}
 II_1
 \leq&C
\int_0^t h^{-1}(s,\mu)\|\e\pa_x\rho^R\|_{X^{10}_{\mu}}^2ds\\
\leq&C\sup_{s\in[0,t]}\sup_{\mu< \mu_0-\la s}(h^{\eta+1}(s, \mu)\|\e\pa_x\rho^R\|_{X^{10}_{\mu}}^2)\int_0^t h^{-\eta-2}(s,\mu)ds\\
\leq&C\la^{-1} h^{-1-\eta}(t,\mu)\sup_{s\in[0,t]}\e^2 \mathcal{E}(s).
 \end{align*}

\underline{Estimate of $II_2.$} By Lemma \ref{lem: analyticity recovery}, Lemma \ref{lem: product 1} and Lemma \ref{cor: est-a_0}, we have
\begin{align*}
II_2\leq&C_0\d \|h^{-\f12}\e\pa_y v^R\|_{\widetilde{L}^2(0,t; X^{10}_{\mu'})}^2\leq C_0\d h^{-1-\eta}(t,\mu') \e^2 \mathcal{D}(t)
\end{align*}
due to $h(s,\mu')\geq h(t,\mu')$ for $s\leq t.$\smallskip

\underline{Estimate of $II_3.$} Thanks to the definition of $F_\rho^1$, we have
\begin{align*}
F_\rho^1
\sim& \pa_x\rho^R+\varphi \pa_y\rho^R+ \rho^R.
\end{align*}
Then by Lemma \ref{lem: product 1}, Lemma \ref{cor: est-a_0} and Lemma \ref{lem: analyticity recovery}, we get
\begin{align*}
II_3\leq& C_0\int_0^t\|\e\pa_x (\pa_x, \varphi\pa_y)\rho^R\|_{X^{10}_{\mu}}\cdot \|\e\pa_x \rho^R\|_{X^{10}_{\mu}} ds\\
\leq& C_0\int_0^th^{-1}(s,\mu)\|\e\pa_x\rho^R\|_{X^{10}_{\mu'}}\cdot \|\e\pa_x \rho^R\|_{X^{10}_{\mu}} ds\\
\leq&C_0\la^{-1} h^{-1-\eta}(t,\mu)\sup_{s\in[0,t]}\e^2 \mathcal{E}(s).
\end{align*}

\underline{Estimate of $II_4.$} Thanks to
\begin{align*}
F_\rho^2\sim u^R+ v^R,
\end{align*}
we get by Lemma \ref{lem: product 1} and Lemma \ref{cor: est-a_0} that 
\begin{align*}
II_4\leq&C_0\|\e\na(u^R, v ^R)\|_{\widetilde{L}^2(0,t; X^{10}_\mu)}\|\e\pa_x\rho^R\|_{\widetilde{L}^2(0,t; X^{10}_\mu)}\\
\leq&C_0\d h^{-1}(t,\mu)\|\e\na(u^R, v ^R)\|_{\widetilde{L}^2(0,t; X^{10}_\mu)}^2+Ch(t,\mu)\|\e\pa_x\rho^R\|_{\widetilde{L}^2(0,t; X^{10}_\mu)}^2\\
\leq& C_0\d h^{-1-\eta}(t,\mu)\e^2 \mathcal{D}(t)+C\la^{-1} h^{1-\eta}(t,\mu)\sup_{s\in[0,t]}\e^2 \mathcal{E}(s).
\end{align*}

Collecting all the above estimates together, we have
\begin{align*}
II_1+\cdots +II_4\leq&C_0\d h^{-1-\eta}(t,\mu)\e^2 \mathcal{D}(t)+C\la^{-1} h^{-1-\eta}(t,\mu)\sup_{s\in[0,t]}\e^2 \mathcal{E}(s).
\end{align*}
which along with \eqref{eq: ||e pa_x rho^R||_X^10} implies our result. 

\end{proof}

\section{Normal energy estimates}

To  obtain the $L_{x,y}^\infty$ norm of $(\rho^R,u^R, v^R)$, we also need the estimates of normal derivatives, i.e., $(\pa_y u^R,\pa_y v^R,\pa_y\rho^R)$.

Taking $\pa_y$ on system \eqref{eq: Error-(u,v,rho)-1}, it is easy to deduce that $(\pa_y u^R,\pa_y v^R,\pa_y\rho^R)$ satisfies the following system
\begin{align}\label{eq: (pa_y u^R, pa_y v^R, pa_y rho^R)}
\left\{
\begin{aligned}
\pa_t\pa_y u^R&+\pa_yG_1^R=   -F_{\pa_y u},\\
\pa_t\pa_y v^R& +\pa_y G_2^R=  -F_{\pa_y v},\\
\pa_t\pa_y \rho^R&+\rho^a\pa_y^2v^R=-\pa_y \mathcal{N}_\rho-\pa_y\rho^a\pa_y v^R-F_{\pa_y \rho},\\
\end{aligned}
\right.
\end{align}
where $(G_1^R, G_2^R)$ is given by
\begin{align}
G_1^R=&\pa_x \rho^R-\e^2\mathfrak{a}(\tri u^R+\pa_x d^R)-\e^2 (\f{1}{\rho^\e}-\f{1}{\rho^a})\big(\tri u^a+\pa_x(\pa_x u^a+\pa_y v^a)\big)+\mathcal{N}_u-R_u,\label{def: G_1^R}\\
G_2^R=&\pa_y \rho^R-\e^2\mathfrak{a}(\tri v^R+\pa_y d^R)-\e^2 (\f{1}{\rho^\e}-\f{1}{\rho^a})\big(\tri v^a+\pa_y(\pa_x u^a+\pa_y v^a)\big) +\mathcal{N}_v-R_v,\label{def: G_2^R}
\end{align}
and $F_{\pa_y u}, F_{\pa_y u}$ and $F_{\pa_y \rho}$ are defined by
\begin{align}
F_{\pa_y u}=&\pa_y(u^a\pa_x u^R+v^a\pa_y u^R+u^R\pa_x u^a+v^R\pa_y u^a)\\
\nonumber
=&u^a\pa_x \pa_yu^R+v^a\pa_y^2 u^R+u^R\pa_x \pa_yu^a+v^R\pa_y^2 u^a\\
\nonumber
&+\pa_yu^a\pa_x u^R+\pa_yv^a\pa_y u^R+\pa_yu^R\pa_x u^a+\pa_yv^R\pa_y u^a,\label{eq: F pa_y u}\\
F_{\pa_y v}=&\pa_y(u^a\pa_x v^R+v^a\pa_y v^R+u^R\pa_x v^a+v^R\pa_y v^a)\\
\nonumber
=&u^a\pa_x \pa_yv^R+v^a\pa_y^2 v^R+u^R\pa_x \pa_yv^a+v^R\pa_y^2 v^a \\
\nonumber
&+\pa_yu^a\pa_x v^R+\pa_yv^a\pa_y v^R+\pa_yu^R\pa_x v^a+\pa_yv^R\pa_y v^a,\label{eq: F pa_y v}\\
F_{\pa_y \rho}=&\pa_y F_\rho.
\end{align}

Moreover, using system \eqref{eq: Error-(u,v,rho)-1} and $(u^R, v^R)|_{y=0}=0,$
we get the boundary condition
\begin{align}\label{BC: pa_y u}
(G_1^R, G_2^R)|_{y=0}=0.
\end{align}

 \medskip
 
 \subsection{Estimate of $(F_{\pa_y u}, F_{\pa_y v}, F_{\pa_y \rho})$}

\begin{lemma}\label{lem: (F_pa_y u, F_pa_y v, F_pa_y rho)}
It holds that
\begin{align*}
&\|F_{\pa_y u}\|_{X^9_\mu}\leq C_0\Big(\e+\|(\pa_x, \varphi\pa_y)\pa_y u^R\|_{X^9_\mu}+\|(\pa_y u^R,\pa_y v^R)\|_{X^9_\mu}+\e^{-1}\|\mathcal{N}_\rho\|_{X^9_\mu}\Big)+C\e^{-1}\|(u^R,v^R, \rho^R)\|_{X^{10}_\mu}, \\
&\|F_{\pa_y v}\|_{X^9_\mu}\leq C_0\Big( \|(\pa_x, \varphi\pa_y)\pa_y v^R\|_{X^9_\mu}+\| (\pa_yu^R,\pa_y v^R)\|_{X^9_\mu}\Big)+C\e^{-1}\|(u^R, v^R, \rho^R)\|_{X^{10}_\mu},\\
&\|F_{\pa_y \rho}\|_{X^9_\mu}\leq C_0\Big(\e+\|(\pa_x, \varphi\pa_y)\pa_y (u^R,\rho^R)\|_{X^9_\mu}+\| \pa_y(u^R,v^R,  \rho^R)\|_{X^9_\mu}\Big) +C\e^{-1}\|(u^R, v^R, \rho^R)\|_{X^{10}_\mu},\\
&\|\pa_y^2 v^R+\mathfrak{a}_0\pa_y\mathcal{N}_\rho\|_{X^{8}_\mu}\leq C_0\e+C\Big(\|\pa_y(u^R, v^R, \rho^R)\|_{X^{9}_\mu}+\e^{-1}\|(u^R, v^R, \rho^R)\|_{X^{10}_\mu}\Big),
\end{align*}
where $\mathcal{N}_\rho$ is given in \eqref{def: N_rho}.
\end{lemma}
\begin{proof}
The proof of $F_{\pa_y u}$, $F_{\pa_y v}$ and $F_{\pa_y\rho}$ are almost the same. Here, we only give the proof of $F_{\pa_y u}.$ 

By Lemma \ref{lem: app} and Lemma \ref{lem: product 1},  we have
\begin{align}\label{est: F_pa_y u-1}
&\|F_{\pa_y u}-v^R\pa_y^2 u_p^0-\pa_y v^R\pa_y u_p^0\|_{X^9_\mu}\\
\nonumber
&\leq C_0\big(\|\pa_x\pa_y u^R\|_{X^9_\mu}+\|\varphi \pa_y \pa_y u^R\|_{X^9_\mu}+\e^{-1}\|u^R\|_{X^9_\mu}+\e^{-1}\|\pa_x u^R\|_{X^9_\mu}\\
\nonumber
&\qquad+\|\pa_y u^R\|_{X^9_\mu}+\e^{-1}\|v^R\|_{X^9_\mu}+\|\pa_y v^R\|_{X^9_\mu}\big)\\
\nonumber
&\leq C_0\big(\|(\pa_x, \varphi\pa_y)\pa_y u^R\|_{X^9_\mu}+\|(\pa_y u^R,\pa_y v^R)\|_{X^9_\mu}\big)+C\e^{-1}\|(u^R, v^R)\|_{X^{10}_\mu},
\end{align}
where we use $\e^{-1}\|\pa_x u^R\|_{X^9_\mu}\leq \e^{-1}\d^{-1}\| u^R\|_{X^{10}_\mu}$ in the last line.

The most difficult terms are  $v^R\pa_y^2 u_p^0$ and $\pa_y v^R\pa_y u_p^0$, which behave as
 \begin{align*}
v^R\pa_y^2 u_p^0=\e^{-2}v^R\pa_z^2 u_p^0\sim \e^{-1}\pa_y v^R z\pa_z^2u_p^0,\quad ,\pa_y v^R\pa_y u_p^0\sim \e^{-1}\pa_y v^Rz\pa_z u_p^0.
\end{align*}
Hence, by Lemma \ref{lem: product 1}, Lemma \ref{lem: app} and Lemma \ref{lem: (F_u, F_v, F_rho)}, we get
\begin{align}\label{est: F_pa_y u-2}
\|v^R\pa_y^2 u_p^0\|_{X^9_\mu}+\|\pa_y v^R\pa_y u_p^0 \|_{X^9_\mu}\leq& C_0\e^{-1}\|\pa_y v^R\|_{X^9_\mu}\\
\nonumber
\leq&C_0\big(\e+\e^{-1}\|\mathcal{N}_\rho\|_{X^9_\mu}\big)+C\e^{-1}\|(u^R,v^R, \rho^R)\|_{X^{10}_\mu}.
\end{align}
This along with \eqref{est: F_pa_y u-1} yields
\begin{align*}
\|F_{\pa_y u}\|_{X^9_\mu}\leq& C_0\Big(\e+\|(\pa_x, \varphi\pa_y)\pa_y u^R\|_{X^9_\mu}+\|(\pa_y u^R,\pa_y v^R)\|_{X^9_\mu}+\e^{-1}\|\mathcal{N}_\rho\|_{X^9_\mu}\Big)+C\e^{-1}\|(u^R,v^R, \rho^R)\|_{X^{10}_\mu}.
\end{align*}

Next, we give the proof of $\pa_y^2 v^R+\mathfrak{a}_0\pa_y\mathcal{N}_\rho$. Multiplying $\mathfrak{a}_0$ on both sides of $\eqref{eq: (pa_y u^R, pa_y v^R, pa_y rho^R)}_3$, we obtain
 \begin{align*}
 \pa_y^2 v^R+\mathfrak{a}_0\pa_y\mathcal{N}_\rho
 =&-\mathfrak{a}_0\big(\pa_t \pa_y\rho^R+F_{\pa_y\rho}+\pa_y\rho^a\pa_y v^R\big).
 \end{align*}
Then by Lemma \ref{lem: product 1}, Lemma \ref{cor: est-a_0} and Lemma \ref{lem: (F_u, F_v, F_rho)}, we  get
\begin{align*}
\|\pa_y^2 v^R+\mathfrak{a}_0\pa_y\mathcal{N}_\rho\|_{X^8_\mu}\leq&C_0(\|\pa_t \pa_y\rho^R+F_{\pa_y\rho}\|_{X^8_\mu}+\|\pa_y v^R\|_{X^8_\mu})\\
\leq& C_0\e+C\big(\|\pa_y(u^R, v^R, \rho^R)\|_{X^{9}_\mu}+\e^{-1}\|(u^R, v^R, \rho^R)\|_{X^{10}_\mu}\big).
\end{align*}
 
\end{proof}

\subsection{Energy estimate of $\|(\pa_y u^R,\pa_y v^R,\pa_y\rho^R)\|_{X^9_{\mu}}^2$}

\begin{proposition}\label{pro: pa_y u^R-2}
Under the assumption \eqref{assume: 1},
there exists $\delta_0>0$ such that for any $t \in[0,T]$ and $\delta\in(0, \delta_0)$, it holds that
\begin{align*}
&\sup_{\mu< \mu_0-\la t}h^\eta(t,\mu)\Big( \|\pa_y(u^R, v^R, c_0\rho^R)\|_{X^9_{\mu}}^2 +\big(c_0-(C\e^\f12+C_0\d^\f12)\big)\|  \e (\pa_y^2 u^R,\pa_y^2v^R)\|^2_{\tilde{L}^2(0, t ;X^9_{\mu})} \Big)\\
\nonumber
&\quad\leq C_0\e^2 +C\e^2 \mathcal{D}(t)  +CA\la^{-1}\sup_{s\in[0,t]} \mathcal{E}(s)+C_0\d^3\|h^{\f12}\pa_y\mathcal{N}_\rho \|_{\widetilde{L}^2(0,t; X^9_{\mu})}^2\\
&\quad\quad+C_0\la^{-1}\sup_{s\in[0,t]}\sup_{\mu<\mu_0-\la s}\Big( h^\eta(s,\mu)\|\e^{-1}(\mathcal{N}_u, \mathcal{N}_v, \mathcal{N}_\rho)\|_{X^9_\mu}^2\Big).
\end{align*}
\end{proposition}

\begin{proof}
Taking $Z^\al$ on  equations in \eqref{eq: (pa_y u^R, pa_y v^R, pa_y rho^R)}, we take the inner product $X^{0}_{\mu,t}$  on  both sides with $(Z^\al \pa_y u^R, Z^\al \pa_yv^R, \mathfrak{a}_0Z^\al \pa_y\rho^R)$, then we sum $\sum_{|\al|=0}^{9}$ and use Lemma \ref{cor: est-a_0} to obtain
\begin{align}\label{est: ||(pa_y u^R, pa_y v^R, pa_ rho^R)||_X^9}
&\f12\|(\pa_y u^R, \pa_y v^R)\|_{X^9_{\mu}}^2+\sum_{|\al|=0}^{9}\langle Z^\al\pa_s\pa_y\rho^R, \mathfrak{a}_0Z^\al\pa_y\rho^R\rangle_{X^{0}_{\mu,t}}+\sum_{|\al|=0}^9\Big\langle Z^\al(\rho^a\pa_y^2 v^R), \mathfrak{a}_0Z^\al\pa_y\rho^R\Big\rangle_{X^0_{\mu,t}}\\
\nonumber
&\qquad+\Big\langle \pa_yG_1^R,\pa_y u^R\Big\rangle_{X^9_{\mu,t}}+\Big\langle \pa_y G_2^R,\pa_y v^R\Big\rangle_{X^9_{\mu,t}}\\
\nonumber
&\leq  C_0\e^2+\d^3\|h^{\f12}\pa_y\mathcal{N}_\rho \|_{\widetilde{L}^2(0,t; X^9_{\mu})}^2+ \d\|\pa_y \rho^a\pa_y v^R\|_{\widetilde{L}^2(0,t; X^9_{\mu})}^2\\
\nonumber
&\qquad+C_0\int_0^t\|(F_{\pa_y u},F_{\pa_y v},F_{\pa_y \rho})\|_{X^{9}_\mu}\|(\pa_y u^R, \pa_y v^R,\pa_y \rho^R)\|_{X^9_{\mu}}ds+C\|h^{-\f12}\pa_y\rho^R\|_{\widetilde{L}^2(0,t; X^9_{\mu})}^2 \\
\nonumber
&= C_0\e^2+\d^3\|h^{\f12}\pa_y\mathcal{N}_\rho \|_{\widetilde{L}^2(0,t; X^9_{\mu})}^2+I_1+I_2+ I_3.
\end{align}

\underline{Estimate of $I_1.$} By  Lemma \ref{lem: product 2} and Lemma \ref{cor: est-a_0},  we get
\begin{align*}
I_1\leq& C_0\d(\|\pa_y\rho^a\|_{X^9_{\mu}}^2+\|\pa_y^2\rho^a\|_{X^9_{\mu}}^2)\|\pa_y v^R\|_{\widetilde{L}^2(0,t; X^9_{\mu})}^2\\
\leq&C_0\d\|\pa_y v^R\|_{\widetilde{L}^2(0,t; X^9_{\mu})}^2\\
\leq&C_0\la^{-1}h^{-\eta}(t,\mu)\sup_{s\in[0,t]}\sup_{\mu< \mu_0-\la s}(h^{\eta}(s,\mu)\|\pa_yv^R(s)\|_{X^{9}_\mu}^2)\\
\leq& C_0\la^{-1}h^{-\eta}(t,\mu)\sup_{s\in[0,t]} \mathcal{E}(s),
\end{align*}
where we used \eqref{est: u^R_L^2_t}  in the third step.\smallskip

\underline{Estimate of $I_2.$} By Lemma \ref{lem: (F_pa_y u, F_pa_y v, F_pa_y rho)}, we get
\begin{align*}
I_2\leq&\int_0^t \Big(C_0\e+C_0\|(\pa_x, \varphi\pa_y)\pa_y(u^R, v^R, \rho^R)\|_{X^9_\mu}+C_0\|\pa_y(u^R, v^R, \rho^R)\|_{X^9_\mu}\\
&\qquad\qquad\qquad+C\e^{-1}\|(u^R,v^R, \rho^R)\|_{X^{10}_\mu}+C_0\e^{-1}\|\mathcal{N}_\rho\|_{X^9_\mu}\Big)\|\pa_y(u^R, v^R, \rho^R)\|_{X^9_\mu}ds.
\end{align*}
As in \eqref{est: integral-gain}, we have
\begin{align*}
I_2\leq&C_0\e^2+\la^{-1}h^{-\eta}(t,\mu)\sup_{s\in[0,t]}\sup_{\mu<\mu_0-\la s}\\
&\times\Big(h^{\eta}(s,\mu)\Big(C_0\|\pa_y(u^R, v^R, \rho^R)(s)\|_{X^9_\mu}^2 +C\e^{-2}\|(u^R, v^R, \rho^R)(s)\|_{X^{10}_\mu}^2+C_0\e^{-2}\|\mathcal{N}_\rho(s)\|_{X^9_\mu}^2\Big)\Big)\\
\leq& C_0\e^2+CA\la^{-1}h^{-\eta}(t,\mu)\sup_{s\in[0,t]}\mathcal{E}(s)+C_0\la^{-1}h^{-\eta}(t,\mu)\sup_{s\in[0,t]}\sup_{\mu<\mu_0-\la s}\Big(h^{\eta}(s,\mu)\e^{-2}\|\mathcal{N}_\rho(s)\|_{X^9_\mu}^2\Big) .
\end{align*}

\underline{Estimate of $I_3$.}  As in \eqref{est: u^R_L^2_t}, we have
\begin{align*}
I_3\leq C\la^{-1}h^{-\eta}(t,\mu)\sup_{s\in[0,t]}\sup_{\mu< \mu_0-\la s}\Big(h^{\eta}(s,\mu)\|\pa_y \rho^R(s)\|_{X^{9}_\mu}^2\Big)\leq C\la^{-1}h^{-\eta}(t,\mu)\sup_{s\in[0,t]} \mathcal{E}(s).
\end{align*}

Summing up, we arrive at 
\begin{align}\label{est: pa_y u-1}
I_1+I_2+I_3\leq& C_0\e^2+CA\la^{-1}h^{-\eta}(t,\mu)\sup_{s\in[0,t]}\mathcal{E}(s)\\
\nonumber
&+C_0\la^{-1}h^{-\eta}(t,\mu)\sup_{s\in[0,t]}\sup_{\mu<\mu_0-\la s}\Big(h^{\eta}(s,\mu)\e^{-2}\|\mathcal{N}_\rho(s)\|_{X^9_\mu}^2\Big).
\end{align}
Next we focus on the left hand side of \eqref{est: ||(pa_y u^R, pa_y v^R, pa_ rho^R)||_X^9}. \smallskip

\underline{Estimate of $\sum_{|\al|=0}^{9}\langle Z^\al\pa_s\pa_y\rho^R, \mathfrak{a}_0Z^\al\pa_y\rho^R\rangle_{X^{0}_{\mu,t}}.$} Using the same argument in \eqref{est: rho^R-time} in Proposition \ref{pro: Zu^R-H}, we get
\begin{align}\label{est: pa_y rho^R-time}
\sum_{|\al|=0}^{9}\langle Z^\al\pa_s\pa_y\rho^R, \mathfrak{a}_0Z^\al\pa_y\rho^R\rangle_{X^{0}_{\mu,t}}\geq& \f{c_0}2\|
\pa_y\rho^R(t)\|_{X^{9}_\mu}^2-C_0\e^2-C_0\la^{-1} h^{-\eta}(t, \mu) \sup_{s\in[0,t]} \mathcal{E}(s).
\end{align}

\underline{Estimate of $\Big\langle \pa_yG_1^R,\pa_y u^R\Big\rangle_{X^9_{\mu,t}}.$}
Due to boundary conditions \eqref{BC: pa_y u}, we get by integration by parts that
\begin{align*}
\Big\langle \pa_yG_1^R,\pa_y u^R\Big\rangle_{X^9_{\mu,t}}=&-\Big\langle G_1^R,\pa_y^2 u^R\Big\rangle_{X^9_{\mu,t}}+\sum_{|\al|=1}^9\Big\langle [Z^{\al},\pa_y]G_1^R,Z^\al \pa_y u^R\Big\rangle_{X^0_{\mu,t}}\\
&-\sum_{|\al|=1}^9\Big\langle G_1^R, [\pa_y, Z^\al]\pa_y u^R\Big\rangle_{X^0_{\mu,t}}=H_1+H_2+H_3.
\end{align*}

$\bullet$ We write
\begin{align*}
H_1=\Big\langle \e^2\mathfrak{a}\pa_y^2 u^R,\pa_y^2 u^R\Big\rangle_{X^9_{\mu,t}}-\Big\langle G_1^R+\e^2\mathfrak{a}\pa_y^2 u^R,\pa_y^2 u^R\Big\rangle_{X^9_{\mu,t}}=H_1^1+H_1^2.
\end{align*}
 According to $\mathfrak{a}\geq c_0>0,$  using the same argument in $D_1^1$ in Proposition \ref{pro: Zu^R-H},  $H_1^1$ is bounded from below by
 \begin{align*}
 &c_0\|\e\pa_y^2 u^R\|_{\widetilde{L}^2(0,t;X^9_\mu)}^2-\e^2\sum_{|\al|=1}^9\Big|\Big\langle [Z^\al, \mathfrak{a}]\pa_y^2 u^R, Z^\al \pa_y^2 u^R\Big\rangle_{X^0_{\mu,t}}  \Big|. 
 \end{align*}
Thanks to
\beno
[Z^\al, \mathfrak{a}]=[Z^\al, \mathfrak{a}_0]+[Z^\al, \mathfrak{a}-\mathfrak{a}_0],
\eeno
 we get by Lemma \ref{lem: product 2}, Lemma \ref{cor: est-a_0} and \eqref{est:a-Low}-\eqref{est:pa-Low} that 
\begin{align*}
\e^2\sum_{|\al|=1}^9\Big|\Big\langle [Z^\al, \mathfrak{a}_0]\pa_y^2 u^R, Z^\al \pa_y^2 u^R\Big\rangle_{X^0_{\mu,t}}  \Big|\leq& C_0(\|Z\mathfrak{a}_0\|_{X^9_\mu}+\|\pa_yZ\mathfrak{a}_0\|_{X^9_\mu})\|\e\pa_y^2 u^R\|_{\widetilde{L}^2(0,t;X^9_\mu)}^2\\
\leq&C_0\d^\f12\|\e\pa_y^2 u^R\|_{\widetilde{L}^2(0,t;X^9_\mu)}^2,
\end{align*}
and 
\begin{align*}
\e^2\sum_{|\al|=1}^9&\Big|\Big\langle [Z^\al, \mathfrak{a}-\mathfrak{a}_0]\pa_y^2 u^R, Z^\al \pa_y^2 u^R\Big\rangle_{X^0_{\mu,t}}  \Big|\leq C\Big(\|\mathfrak{a}-\mathfrak{a}_0\|_{X^6_\mu}+\|\pa_y(\mathfrak{a}-\mathfrak{a}_0)\|_{X^6_\mu}\Big)\|\e\pa_y^2 u^R\|_{\widetilde{L}^2(0,t;X^9_\mu)}^2\\
&\qquad +C(\|\mathfrak{a}-\mathfrak{a}_0\|_{\widetilde{L}^2(0,t;X^9_\mu)}+\|\pa_y(\mathfrak{a}-\mathfrak{a}_0)\|_{\widetilde{L}^2(0,t;X^9_\mu)})\|\e\pa_y^2 u^R\|_{\widetilde{L}^2(0,t;X^9_\mu)}\sup_{s\in[0,t]}(\|\e\pa_y^2 u^R\|_{X^6_\mu})\\
\leq& \big(C\e^\f12+C_0\d\big)\|\e\pa_y^2 u^R\|_{\widetilde{L}^2(0,t;X^9_\mu)}^2+C\big(\|\rho^R\|_{\widetilde{L}^2(0,t;X^9_\mu)}^2+\|\pa_y\rho^R\|_{\widetilde{L}^2(0,t;X^9_\mu)}^2\big)\\
\leq&\big(C\e^\f12+C_0\d\big)\|\e\pa_y^2 u^R\|_{\widetilde{L}^2(0,t;X^9_\mu)}^2+C\la^{-1}h^{-\eta}(t,\mu)\sup_{s\in[0,t]}\mathcal{E}(s),
\end{align*}
where  we use $\|\e\pa_y^2 u^R\|_{X^6_\mu}\leq \e^\f12$. This shows that
\begin{align*}
H_1^1\geq& \big(c_0-(C\e^\f12+C_0\d^\f12)\big)\|\e\pa_y^2 u^R\|_{\widetilde{L}^2(0,t;X^9_\mu)}^2-C\la^{-1}h^{-\eta}(t,\mu)\sup_{s\in[0,t]}\mathcal{E} (s).
\end{align*}

For  $H_1^2,$ according to the definition of $G_1^R$ in \eqref{def: G_1^R},
we get by Lemma \ref{lem: product 2}, Lemma \ref{cor: est-a_0}, \eqref{eq:a-diff} and the fact $\pa_x d^R\sim \pa_x\na(u^R, v^R)$ that
\begin{align*}
\|G_1^R+\e^2\mathfrak{a}\pa_y^2 u^R\|_{\widetilde{L}^2(0,t;X^9_\mu)}\leq&C_0\|(\pa_x\rho^R,\rho^R,\mathcal{N}_u, R_u, \e^2\mathfrak{a}\pa_x^2 u^R, \e^2\mathfrak{a}\pa_x d^R )\|_{\widetilde{L}^2(0,t;X^9_\mu)}\\
\leq&C_0\e^2+C_0\|(\pa_x\rho^R,\rho^R,\mathcal{N}_u )\|_{\widetilde{L}^2(0,t;X^9_\mu)}\\
&+C\sup_{s\in [0,t]}(\|\mathfrak{a}_0\|_{X^7_\mu}+\|\mathfrak{a}-\mathfrak{a}_0\|_{X^7_\mu})(\|\e^2\na( u^R, v^R)\|_{\widetilde{L}^2(0,t;X^{10}_\mu)}\\
&+C\e\|\mathfrak{a}-\mathfrak{a}_0\|_{\widetilde{L}^2(0,t;X^{9}_\mu)}\sup_{s\in[0,t]}\|\e(\pa_xu^R,\pa_y \pa_xu^R, \pa_y v^R, \pa_y^2 v^R)\|_{X^7_\mu}\\
\leq& C_0\e^2+C_0\|\mathcal{N}_u \|_{\widetilde{L}^2(0,t;X^9_\mu)}+C\|\rho^R\|_{\widetilde{L}^2(0,t;X^{10}_\mu)}+C\|\e^2\na( u^R, v^R)\|_{\widetilde{L}^2(0,t;X^{10}_\mu)},
\end{align*}
which implies 
\begin{align*}
|H_1^2|\leq& \|\pa_y^2 u^R\|_{\widetilde{L}^2(0,t;X^9_\mu)}\|G_1^R+\e^2\mathfrak{a}\pa_y^2 u^R\|_{\widetilde{L}^2(0,t;X^9_\mu)}\\
\leq&\|\pa_y^2 u^R\|_{\widetilde{L}^2(0,t;X^9_\mu)}\Big(C_0\e^2+C_0\|\mathcal{N}_u \|_{\widetilde{L}^2(0,t;X^9_\mu)}+C\|\rho^R\|_{\widetilde{L}^2(0,t;X^{10}_\mu)}+C\|\e^2\na( u^R, v^R)\|_{\widetilde{L}^2(0,t;X^{10}_\mu)}\Big)\\
\leq&(\f{c_0}{10}+\d)\|\e\pa_y^2 u^R\|_{\widetilde{L}^2(0,t;X^9_\mu)}^2+C_0\e^2+C\e^2\|\na( u^R, v^R)\|_{\widetilde{L}^2(0,t;X^{10}_\mu)}^2\\
&\qquad+\la^{-1}h^{-\eta}(t,\mu)\times\sup_{s\in[0,t]}\sup_{\mu< \mu_0-\la s}\Big(h^{\eta}(s,\mu)\Big(C\e^{-2}\|\rho^R\|_{X^{10}_\mu}^2+C_0\e^{-2}\|\mathcal{N}_u\|_{X^9_\mu}^2\Big)\Big)\\
\leq&(\f{c_0}{10}+C_0\d)\|\e\pa_y^2 u^R\|_{\widetilde{L}^2(0,t;X^9_\mu)}^2+C_0\e^2+C\e^2 h^{-\eta}(t,\mu) \mathcal{D}(t)+C\la^{-1}h^{-\eta}(t,\mu)\sup_{s\in[0,t]} \mathcal{E}(s)\\
&+C_0\la^{-1}h^{-\eta}(t,\mu)\sup_{s\in[0,t]}\sup_{\mu< \mu_0-\la s} \Big(h^\eta(s,\mu)\e^{-2}\|\mathcal{N}_u\|_{X^9_\mu}^2\Big).
\end{align*}
Thus, we obtain
 \begin{align*}
 |H_1|\geq&\big(\f 9{10}c_0-(C\e^\f12+C_0\d)\big)\|\e\pa_y^2 u^R\|_{\widetilde{L}^2(0,t;X^9_\mu)}^2-C_0\e^2-C\e^2 h^{-\eta}(t,\mu) \mathcal{D}(t)\\
 &-C\la^{-1}h^{-\eta}(t,\mu)\sup_{s\in[0,t]} \mathcal{E}(s)-C_0\la^{-1}h^{-\eta}(t,\mu)\sup_{s\in[0,t]}\sup_{\mu< \mu_0-\la s} \Big(h^\eta(s,\mu)\e^{-2}\|\mathcal{N}_u\|_{X^9_\mu}^2\Big).
 \end{align*}

$\bullet$ For $H_2, H_3$, according to \eqref{est: commutator}, there exists $\d$ in the front of them. Thus, using the same process in $H_1$, we can prove that
\begin{align*}
|H_2|+|H_3|\leq&\big(\f 1{10}c_0+C\e^\f12+C_0\d\big)\|\e\pa_y^2 u^R\|_{\widetilde{L}^2(0,t;X^9_\mu)}^2+C_0\e^2+C\e^2 h^{-\eta}(t,\mu) \mathcal{D}(t)\\
&+C\la^{-1}h^{-\eta}(t,\mu)\sup_{s\in[0,t]} \mathcal{E}(s)\\
&+C_0\la^{-1}h^{-\eta}(t,\mu)\sup_{s\in[0,t]}\sup_{\mu< \mu_0-\la s} \Big(h^\eta(s,\mu)\e^{-2}\|\mathcal{N}_u\|_{X^9_\mu}^2\Big).
\end{align*}

This shows that
\begin{align}\label{est: H_1}
\Big|\Big\langle \pa_yG_1^R,\pa_y u^R\Big\rangle_{X^9_{\mu,t}}\Big|\geq &\big(\f 45c_0-(C\e^\f12+C_0\d^\f12)\big)\|\e\pa_y^2 u^R\|_{\widetilde{L}^2(0,t;X^9_\mu)}^2-C_0\e^2-C\e^2 h^{-\eta}(t,\mu) \mathcal{D}(t)\\
\nonumber
&-C\la^{-1}h^{-\eta}(t,\mu)\sup_{s\in[0,t]} \mathcal{E}(s)\\
&-C_0\la^{-1}h^{-\eta}(t,\mu)\sup_{s\in[0,t]}\sup_{\mu< \mu_0-\la s} \Big(h^\eta(s,\mu)\e^{-2}\|\mathcal{N}_u\|_{X^9_\mu}^2\Big).\nonumber
\end{align}

\underline{Estimate of $\Big\langle \pa_y G_2^R,\pa_y v^R\Big\rangle_{X^9_{\mu,t}}+\sum_{|\al|=0}^9\Big\langle Z^\al(\rho^a\pa_y^2 v^R), \mathfrak{a}_0Z^\al\pa_y\rho^R\Big\rangle_{X^0_{\mu,t}}.$} 
Applying integration by parts and the boundary condition \eqref{BC: pa_y u}, we write
\begin{align*}
&\Big\langle \pa_y G_2^R,\pa_y v^R\Big\rangle_{X^9_{\mu,t}}+\sum_{|\al|=0}^9\Big\langle Z^\al(\rho^a\pa_y^2 v^R), \mathfrak{a}_0Z^\al\pa_y\rho^R\Big\rangle_{X^0_{\mu,t}}\\
=&-\Big\langle G_2^R-\pa_y \rho^R,\pa_y^2 v^R\Big\rangle_{X^9_{\mu,t}}+\sum_{|\al|=1}^9\Big\langle [Z^{\al},\pa_y](G_2^R-\pa_y \rho^R),Z^\al \pa_y v^R\Big\rangle_{X^0_{\mu,t}}\\
&-\sum_{|\al|=1}^9\Big\langle G_2^R-\pa_y \rho^R, [\pa_y, Z^\al]\pa_y v^R\Big\rangle_{X^0_{\mu,t}}+\sum_{|\al|=1}^9\Big\langle [Z^{\al},\pa_y]\pa_y \rho^R,Z^\al \pa_y v^R\Big\rangle_{X^0_{\mu,t}}\\
&-\sum_{|\al|=1}^9\Big\langle \pa_y \rho^R,[\pa_y, Z^\al]\pa_y v^R\Big\rangle_{X^0_{\mu,t}}+\sum_{|\al|=0}^9\Big\langle [Z^\al,\rho^a]\pa_y^2 v^R, \mathfrak{a}_0Z^\al\pa_y\rho^R\Big\rangle_{X^0_{\mu,t}}
=& H_4+\cdots+ H_9.
\end{align*}

$\bullet$ Since $\tri v^R+\pa_y d^R=2\pa_y^2 v^R+\pa_x^2 v^R+\pa_x\pa_y u^R\sim 2\pa_y^2 v^R+\pa_x\na(u^R, v^R),$ using a similar argument in $H_1$, we have
\begin{align*}
&|H_4+H_5+H_6|\\
\geq& \big(\f 45c_0-(C\e^\f12+C_0\d^\f12)\big)\|\e\pa_y^2 v^R\|_{\widetilde{L}^2(0,t;X^9_\mu)}^2-C_0\e^2-C\e^2 h^{-\eta}(t,\mu) \mathcal{D}(t)\\
&-C\la^{-1}h^{-\eta}(t,\mu)\sup_{s\in[0,t]} \mathcal{E}(s)-C_0\la^{-1}h^{-\eta}(t,\mu)\sup_{s\in[0,t]}\sup_{\mu< \mu_0-\la s} \Big(h^\eta(s,\mu)\e^{-2}\|\mathcal{N}_v\|_{X^9_\mu}^2\Big).\nonumber
\end{align*}

$\bullet$  Due to \eqref{est: commutator}, it follows from Lemma \ref{lem: (F_pa_y u, F_pa_y v, F_pa_y rho)}  that
\begin{align*}
&|H_7|+|H_8|+|H_9|\\
&\leq C_0\d^\f12\int_0^t\|\pa_y^2 v^R+\mathfrak{a}_0\pa_y\mathcal{N}_\rho\|_{X^8_\mu}\|\pa_y \rho^R\|_{X^9_\mu}ds+C_0\d\|h^{\f12}\mathfrak{a}_0\pa_y\mathcal{N}_\rho\|_{\widetilde{L}^2(0,t;X^{8}_\mu)}\|h^{-\f12}\pa_y\rho^R\|_{\widetilde{L}^2(0,t;X^{9}_\mu)}\\
&\leq \d^\f12\int_0^t\Big(C_0\e+C\|\pa_y(u^R, v^R, \rho^R)\|_{X^{9}_\mu}+C\e^{-1}\|(u^R, v^R, \rho^R)\|_{X^{10}_\mu}\Big)\|\pa_y \rho^R\|_{X^9_\mu}ds\\
&\qquad+C_0\d\|h^{\f12}\pa_y\mathcal{N}_\rho\|_{\widetilde{L}^2(0,t;X^{8}_\mu)}\|h^{-\f12}\pa_y\rho^R\|_{\widetilde{L}^2(0,t;X^{9}_\mu)}\\
&\leq C_0\e^2+CA\la^{-1}h^{-\eta}(t,\mu)\sup_{s\in[0,t]} \mathcal{E} (s)+C_0\d^3\|h^{\f12}\pa_y\mathcal{N}_\rho\|_{\widetilde{L}^2(0,t;X^{8}_\mu)}^2.
\end{align*}

Summing up, we get
\begin{align}\label{est: H_2}
&\Big|\Big\langle \pa_y G_2^R,\pa_y v^R\Big\rangle_{X^9_{\mu,t}}+\sum_{|\al|=0}^9\Big\langle Z^\al(\rho^a\pa_y^2 v^R), \mathfrak{a}_0Z^\al\pa_y\rho^R\Big\rangle_{X^0_{\mu,t}}\Big|\\
\nonumber
&\geq \big(\f 45c_0-(C\e^\f12+C_0\d^\f12)\big)\|\e\pa_y^2 v^R\|_{\widetilde{L}^2(0,t;X^9_\mu)}^2-C_0\e^2-C\e^2 h^{-\eta}(t,\mu) \mathcal{D}(t)-C_0\d^3\|h^{\f12}\pa_y\mathcal{N}_\rho\|_{\widetilde{L}^2(0,t;X^{8}_\mu)}^2\\
\nonumber
&\qquad-CA\la^{-1}h^{-\eta}(t,\mu)\sup_{s\in[0,t]} \mathcal{E}(s)-C_0\la^{-1}h^{-\eta}(t,\mu)\sup_{s\in[0,t]}\sup_{\mu< \mu_0-\la s} \Big(h^\eta(s,\mu)\e^{-2}\|\mathcal{N}_v\|_{X^9_\mu}^2\Big).
\end{align}

Substituting \eqref{est: pa_y u-1}, \eqref{est: pa_y rho^R-time}, \eqref{est: H_1} and \eqref{est: H_2} into \eqref{est: ||(pa_y u^R, pa_y v^R, pa_ rho^R)||_X^9},  we derive the result. \end{proof}

\section{Proof of Theorem \ref{thm:main}}
In this section, we prove Theorem \ref{thm:main}. First of all, the local well-posedness of the compressible Navier-Stokes equations in the analytic space can be proved by the same argument as the above sections. More precisely, there exists $T_\e>0$ depending on $\e$ such that \eqref{eq: CNS} has a unique solution $(\rho^\e, u^\e, v^\e)$ in $[0, T_\e]$ satisfying 
\beno
\sup_{t\in[0, T_\e]}\sup_{\mu< \mu_0-\la t}(\|(\rho^\e, u^\e, v^\e) \|_{X^{10}_\mu} +\|\na(\rho^\e, u^\e, v^\e) \|_{X^{9}_\mu} )<+\infty.
\eeno
Moreover, let $T^*$ be the maximal existence time of the solution.  It is easy to prove that the solution can be extended after $t=T^*$ if 
\beno
\sup_{t\in[0, T]}\sup_{\mu< \mu_0-\la t}\|(\rho^\e, u^\e, v^\e) \|_{X^{10}_\mu}+\|\na(\rho^\e, u^\e, v^\e) \|_{X^{9}_\mu}  <\infty\quad \textrm{for any}\quad T\in[0, T^*).
\eeno

\medskip

In the sequel, we use the bootstrap argument to prove the main theorem.

{\bf
We assume that there exist positive constants $C_*$, $C_{**}$ and $C_{***}$ such that
\begin{align}
\sup_{t\in[0,T]}\mathcal{E}(t)\le C_*^2\e^2,\quad \sup_{t\in[0,T]}\sup_{\mu< \mu_0-\la t}\|\pa_y^2u^R\|_{X^7_\mu}\leq C_{**}, \quad \sup_{t\in[0,T]}\sup_{\mu< \mu_0-\la t}\|\pa_y^2v^R\|_{X^7_\mu}\leq C_{***}\e. \label{assume: 3}
\end{align}

 }
 \medskip
 
 Before we proceed with the bootstrap argument, we give the relationship between \eqref{assume: 1} and $\mathcal{E}(t)$.
 \begin{lemma}\label{lemma:G(t)}
 Under the assumption \eqref{assume: 3}, if $\la\delta^2\geq 2A,$ we have 
 \ben\label{eq:energy-low}
\sup_{t\in[0, T]}\sup_{\mu<\mu_0-\la t}\Big(\e^{-2}\|(u^R, v^R, \rho^R)(t)\|_{X^{9}_{\mu}}^2+\|\pa_y(u^R, v^R, \rho^R)(t)\|_{X^{8}_{\mu}}^2\Big)\leq  \frak{C}\e^2,
 \een
 and
 \beno
\sup_{t\in[0, T]}\sup_{\mu<\mu_0-\la t}\|\pa_y^2 (\e u^R,  v^R)\|_{X^7_\mu} \leq \frak{C}\e,
 \eeno
 where $\frak{C}$ depends on $C_*, C_{**}$ and $C_{***}$ but independent of $\e$ and $\delta$.
 \end{lemma}
 \begin{proof}
Notice that for any analytic function $f$ , it holds that
\begin{align}\label{est: f_X^k-0}
\nonumber
\|f\|_{X^k_\mu}^2=\sum_{|\al|=0}^k\|Z^\al f\|^2_{X^0_\mu}\leq& C_0\sum_{|\al|=0}^k\Big\|\int_0^tZ^\al \pa_sfd s\Big\|_{X^0_\mu}^2 +C_0\sum_{|\al|=0}^k\|(Z^\al f)(0,\cdot, \cdot)\|_{X^0_{\mu_0}} \\
\leq& C_0\int_0^t \|\pa_s f \|_{X^k_\mu}^2ds+C_0\sum_{|\al|=0}^k\|(Z^\al f)(0,\cdot, \cdot)\|_{X^0_{\mu_0}}\\
\nonumber
\leq& \f{C_0}{\delta^2\la} \sup_{s\in[0,t]}\sup_{\mu< \mu_0-\la s}(h^\eta(s,\mu)\|f\|_{X^{k+1}_\mu}^2)+C_0\sum_{|\al|=0}^k\|(Z^\al f)(0,\cdot, \cdot)\|_{X^0_{\mu_0}}.
\end{align}
Then, according to \eqref{initial: 6} and Lemma \ref{cor: est-a_0}, we have
\begin{align*}
&\sup_{t\in[0, T]}\sup_{\mu<\mu_0-\la t}\Big(\e^{-2}\|(u^R, v^R, \rho^R)(t)\|_{X^{9}_{\mu}}^2+\|\pa_y(u^R, v^R, \rho^R)(t)\|_{X^{8}_{\mu}}^2\Big)\nonumber\\
&\leq C_0\e^2+\f{C_0A}{\delta^2 \la}\mathcal{E}(t)\leq C_0\e^2+\f{C_0AC_*^2}{ \delta^2\la}\e^2\leq C_0(1+C_*^2)\e^2 ,
\end{align*}
by taking $\la\delta^2 \geq 2A$, which along with \eqref{assume: 3} deduces that 
\begin{align*}
\|\pa_y^2 (\e u^R,  v^R)\|_{X^7_\mu}\leq C_{**}\e+C_{***}\e\leq (C_{**}+C_{***})\e.
\end{align*}
\end{proof}
\begin{remark}\label{lemma:G(t)-R}
The above lemma ensures that the assumption \eqref{assume: 1} holds.
\end{remark}

\medskip

To proceed, let us estimate nonlinear terms.
\begin{lemma}\label{lem: Nonlinear}
Under the assumption \eqref{assume: 3}, it holds that
\begin{align}
\sup_{\mu<\mu_0-\la t}\Big(\e^{-2} h^\eta(t,\mu) \|h^{\f12}(\mathcal{N}_u,\mathcal{N}_v,\mathcal{ N}_\rho)\|_{\widetilde{L}^2(0,t; X^{10}_\mu)}^2\Big)\leq&C_0\frak{C}A\d^{-4}\la^{-1}\mathcal{E}(t)+  C_0\frak{C}\d^{-2}\mathcal{D}(t),\label{est: N_1}\\
\sup_{\mu< \mu_0-\la t}\Big(h(t,\mu)^{1+\eta}\|h^{\f12} \pa_x\mathcal{N}_\rho\|_{\widetilde{L}^2(0,t; X^{10}_\mu)}^2\Big)\leq&C_0\frak{C}A\d^{-4}\la^{-1}\mathcal{E}(t)+  C_0\frak{C}\d^{-2}\e^2 \mathcal{D}(t),\label{est: N_2}\\
\sup_{\mu< \mu_0-\la t}\Big(h^\eta(t,\mu)\|h^{\f12}\pa_y\mathcal{N}_\rho\|_{\widetilde{L}^2(0,t; X^{9}_\mu)}^2\Big)\leq&C_0\frak{C}A\d^{-4}\la^{-1}\mathcal{E}(t)+  C_0\frak{C}\d^{-2}\mathcal{D}(t),\label{est: N_3}\\
\sup_{\mu< \mu_0-\la t}\Big(h^{\eta}(t,\mu)\e^{-2}\|(\mathcal{N}_u, \mathcal{N}_v, \mathcal{N}_\rho)\|_{X^{9}_\mu}^2\Big)\leq&C_0\frak{C}A\d^{-4}\la^{-1}\mathcal{E}(t).\label{est: N_4}
\end{align}
Here the definition of $(\mathcal{N}_\rho, \mathcal{N}_u, \mathcal{N}_v)$ is given in $\eqref{def: N_rho}-\eqref{def: N_v}.$
\end{lemma}

\begin{proof}
Here we only give the proof of \eqref{est: N_1}. The others  can be proved by the same argument. Firstly, recalling the definition of $\mathcal{N}_\rho$ in \eqref{def: N_rho}, we divide it into two parts
\begin{align*}
\mathcal{N}_\rho=\mathcal{N}_\rho^1+\mathcal{N}_\rho^2,
\end{align*}
where 
\begin{align}
\mathcal{N}_\rho^1=u^R\pa_x \rho^R+v^R\pa_y \rho^R,\quad \mathcal{N}_\rho^2=\rho^R(\pa_x u^R+\pa_y v^R).
\end{align}
By Lemma \ref{lem: product 2} and \eqref{eq:energy-low}, under the assumption \eqref{assume: 3}, we have
\begin{align*}
&\|h^{\f12}\mathcal{N}_\rho^1\|_{\widetilde{L}^2(0,t; X^{10}_\mu)}\\
\leq&
C_0\d^{-1} \|(u^R, \pa_y u^R)\|_{X^7_\mu}\|h^{\f12}\pa_x \rho^R\|_{\widetilde{L}^2(0,t; X^{10}_\mu)} +C_0\d^{-1}\| (\pa_x\rho^R,\pa_y\pa_x\rho^R)\|_{X^7_\mu}\| h^{\f12} u^R\|_{\widetilde{L}^2(0,t; X^{10}_\mu)}\\
&+C_0\d^{-1}\|(\pa_y v^R,\pa_y^2 v^R)\|_{X^7_\mu}\|h^{\f12}\varphi \pa_y\rho^R\|_{\widetilde{L}^2(0,t; X^{10}_\mu)}+C_0\d^{-1}\|\pa_y \rho^R\|_{X^7_\mu}\|h^{\f12}(v^R,\pa_y v^R)\|_{\widetilde{L}^2(0,t; X^{10}_\mu)} \\
\leq&C_0 \frak{C}\e \d^{-2}\Big(\|h^{\f12}(\pa_x, \varphi \pa_y)\rho^R\|_{\widetilde{L}^2(0,t; X^{10}_\mu)}+\|h^{\f12}(u^R,  v^R)\|_{\widetilde{L}^2(0,t; X^{10}_\mu)}\Big)+C_0\frak{C}\d^{-1} \|\e\pa_y v^R\|_{\widetilde{L}^2(0,t; X^{10}_\mu)}\\
\leq&C_0 \frak{C}\e \d^{-2}\Big(\int_0^t h^{-1}(s,\mu)(\|\rho^R\|_{X^{10}_{\mu'}}^2+\|(u^R, v^R)\|_{X^{10}_\mu}^2)ds\Big)^\f12+ C_0\frak{C}\d^{-1}   h^{-\eta/2}(t,\mu)\e \mathcal{D}(t)^\f12\\
\leq&C_0 \frak{C}\e \d^{-2} A^\f12\la^{-\f12}h^{-\eta/2}(t,\mu)\e \mathcal{E}(t)^\f12+ C_0\frak{C}\d^{-1}  h^{-\eta/2}(t,\mu) \e \mathcal{D}(t)^\f12,
\end{align*}
which gives
\begin{align*}
\sup_{\mu<\mu_0-\la t}\Big(\e^{-2}h^\eta(t,\mu)\|h^{\f12}\mathcal{N}_\rho^1\|_{\widetilde{L}^2(0,t; X^{10}_\mu)}^2\Big)\leq& C_0\frak{C}\d^{-4}\la^{-1}\mathcal{E}(t)+  C_0\frak{C}\d^{-2}\mathcal{D}(t).
\end{align*}
Similarly, we have
\begin{align*}
\sup_{\mu<\mu_0-\la t}\Big(\e^{-2}h^\eta(t,\mu)\|h^{\f12}(\mathcal{N}_u,\mathcal{N}_v)\|_{\widetilde{L}^2(0,t; X^{10}_\mu)}^2\Big)\leq&C_0\frak{C}\d^{-4}\la^{-1}\mathcal{E}(t)+  C_0\frak{C}\d^{-2}\mathcal{D}(t).
\end{align*}

For $\mathcal{N}_\rho^2$, we get by Lemma \ref{lem: product 2}  that
\begin{align*}
\|h^{\f12}\mathcal{N}_\rho^2\|_{\widetilde{L}^2(0,t; X^{10}_\mu)}\leq &C_0\d^{-1}\|\rho^R\|_{\widetilde{L}^2(0,t; X^{10}_\mu)} \|(\pa_xu^R,\pa_x\pa_yu^R,\pa_y v^R,\pa_y^2v^R )\|_{X^7_\mu}\\
&+C_0\d^{-1}\|(\rho^R, \pa_y \rho^R)\|_{X^7_\mu}\|(\pa_x u^R, \pa_y v^R)\|_{\widetilde{L}^2(0,t; X^{10}_\mu)}\\
\leq&C_0\frak{C}\d^{-2}\|\rho^R\|_{\widetilde{L}^2(0,t; X^{10}_\mu)}+ C_0\frak{C}\d^{-1}\e\|\na(u^R, v^R)\|_{\widetilde{L}^2(0,t; X^{10}_\mu)}\Big),
\end{align*}
which gives
\begin{align*}
\sup_{\mu<\mu_0-\la t}\Big(\e^{-2}h^\eta(t,\mu)\|h^{\f12}\mathcal{N}_\rho^2\|_{\widetilde{L}^2(0,t; X^{10}_\mu)}^2\Big)\leq& C_0\frak{C}\d^{-4}\la^{-1}\mathcal{E}(t)+  C_0\frak{C}\d^{-2}\mathcal{D}(t).
\end{align*}
 \end{proof}

\smallskip

 \noindent{\bf Proof of Theorem \ref{thm:main}}.
First of all, by Remark \ref{lemma:G(t)-R}, under the assumption \eqref{assume: 3},  the assumption \eqref{assume: 1} holds. Thus,  Proposition \ref{pro: Zu^R-H}, Proposition \ref{pro: e pa_tpa_x rho^R} and Proposition \ref{pro: pa_y u^R-2} hold. Collecting them together, we have 
\begin{align*} 
& \sup_{t\in[0, T]}\mathcal{E}(t)+ c_0 \mathcal{D}(T)\\
\nonumber
&\le C_0\e^2+(CA\la^{-1}+C_0A^{-1})\sup_{t\in[0, T]}\mathcal{E}(t)+(C\e^\f12+C_0\delta^\f12)\mathcal{D}(T)\\
&+C_0\d^3 \sup_{t\in[0, T]}\sup_{\mu<\mu_0-\la t}\Big(h^\eta(t,\mu)\Big(\e^{-2}\|h^{\f12}(\mathcal{N}_u,\mathcal{N}_v,\mathcal{ N}_\rho)\|_{\widetilde{L}^2(0,t; X^{10}_\mu)}^2\nonumber\\
&\qquad\qquad\qquad+\|h^{\f12}\pa_y\mathcal{ N}_\rho\|_{\widetilde{L}^2(0,t; X^{9}_\mu)}^2+h(t,\mu)\|h^{\f12} \pa_x\mathcal{N}_\rho\|_{\widetilde{L}^2(0,t; X^{10}_\mu)}^2\Big)\Big)\nonumber\\
\nonumber
&\qquad+C_0\la^{-1}\sup_{t\in[0,T]}\sup_{\mu<\mu_0-\la t}\Big(h^{\eta}\e^{-2}\|(\mathcal{N}_\rho,\mathcal{N}_u,\mathcal{N}_v )\|_{X^{9}_\mu}^2\Big),
\end{align*}
which implies 
\begin{align}\label{est: energy-1}
&(1-CA\la^{-1}-C_0A^{-1})\sup_{t\in[0, T]}\mathcal{E}(t)+\big(c_0-(C\e^\f12+C_0\d^\f12)\big)\mathcal{D}(T)\\
\nonumber
&\le C_0\e^2+C_0\d^3 \sup_{t\in[0, T]}\sup_{\mu<\mu_0-\la t}\Bigg(h^\eta(t,\mu)\Big(\e^{-2}\|h^{\f12}(\mathcal{N}_u,\mathcal{N}_v,\mathcal{ N}_\rho)\|_{\widetilde{L}^2(0,t; X^{10}_\mu)}^2\nonumber\\
&\qquad\qquad\qquad+\|h^{\f12}\pa_y\mathcal{ N}_\rho\|_{\widetilde{L}^2(0,t; X^{9}_\mu)}^2+h(t,\mu)\|h^{\f12} \pa_x\mathcal{N}_\rho\|_{\widetilde{L}^2(0,t; X^{10}_\mu)}^2\Big)\Bigg)\nonumber\\
\nonumber
&\qquad+C_0\la^{-1}\sup_{t\in[0,T]}\sup_{\mu<\mu_0-\la t}\Big(h^{\eta}\e^{-2}\|(\mathcal{N}_\rho,\mathcal{N}_u,\mathcal{N}_u )\|_{X^{9}_\mu}^2\Big)\\
&\le C_0\e^2 +C\frak{C}A\la^{-1} \sup_{t\in[0, T]}\mathcal{E}(t) +  C_0\frak{C}\d \mathcal{D}(T),\nonumber
\end{align}
which we used Lemma \ref{lem: Nonlinear} in the last inequality. By now, we obtain
\begin{align*}
&(1-C_0A^{-1}-CA\la^{-1}-C\frak{C}A\la^{-1}) \sup_{t\in[0, T]} \mathcal{E}(t)+\Big(c_0-(C\e^\f12+C_0\d^\f12+C_0\frak{C}\d)\Big)\mathcal{D}(T)\leq C_0\e^2.
\end{align*}

\smallskip

Next we give the definitions of $C_*, C_{**}$ and $C_{***}$. 

\underline{Definition of $C_*$.}  Firstly, we take  $\d$ and $\e$ small enough such that $C\e+C_0\d^\f12+C_0\frak{C}\d \leq \f12 c_0.$ Then, we take $A$ large enough such that $C_0A^{-1}\leq \f14,$ then we take $\la$ large enough such that $CA\la^{-1}+C\frak{C}A\la^{-1}\leq \f14,$ which give $C_0A^{-1}+CA\la^{-1}+C\frak{C}A\la^{-1}\leq \f12.$ Thus, we infer that
\begin{align*}
\f12\sup_{t\in[0, T]}\mathcal{E}(t)+\f{c_0}{2}\mathcal{D}(T)\leq  C_0 \e^2\leq \f{C_*^2}{4}\e^2.
\end{align*}
Here we take $C_*^2= 4C_0$.
 
\underline{Definitions of $C_{**}$ and $C_{***}$.} 
In the end, we should estimate $\|\pa_y^2u^R\|_{X^7_\mu}$ and $\|\pa_y^2v^R\|_{X^7_\mu}.$
By the second  equation of \eqref{eq: Error-(u,v,rho)-1}, we infer from \eqref{assume: 3}, \eqref{eq:energy-low}, Lemma \ref{cor: est-a_0} and  Lemma \ref{lem: (F_u, F_v, F_rho)} that
\begin{align*}
\|\e^2\pa_y^2 u^R\|_{X^{7}_\mu}\leq& \|\rho^\e(\pa_t u^R+ v^R\pa_y u_p^0+\mathcal{N}_u+F_{u})+\e^2\pa_x^2 u^R+\e^2\pa_x d^R\|_{X^{7}_\mu}\\
\leq&C_0(\|\rho^a\|_{X^7_\mu}+\|\pa_y\rho^a\|_{X^7_\mu}+\d^{-1}\|\rho^R\|_{X^7_\mu}+\d^{-1}\|\pa_y\rho^R\|_{X^7_\mu})\\
&\qquad\times\Big(\|(\pa_t, \pa_x, \varphi\pa_y)(u^R, v^R, \rho^R)\|_{X^7_\mu}+\|(\pa_y v^R, \mathcal{N}_u)\|_{X^7_\mu}\Big)\\
&+2\e^2\d^{-2}\|u^R\|_{X^{9}_\mu}+\e^2\d^{-1}\|\pa_y v^R\|_{X^{8}_\mu}\\
\leq& (C_0+C_0\d^{-1}(1+C_*)\e)\Big(\|(\pa_t, \pa_x, \varphi\pa_y)(u^R, v^R, \rho^R)\|_{X^7_\mu}+\|( \mathcal{N}_u,  \mathcal{N}_\rho)\|_{X^7_\mu}\Big)\\
&+\e^2\d^{-2}C_0(1+C_*)\e^2+\e^2\d^{-1}C_0(1+C_*)\e\\
\leq& (C_0+C_0\d^{-1}\e+\f{A^\f12C_0C_*}{\d^2\la^\f12}\e)(\e^2+\f{A^\f12C_*}{\d^3\la}\e^{2})+\e^2\d^{-1}C_0(1+C_*)\e\\
\leq&  C_0(1+C_*)\e^2\leq \f12 C_{**}\e^2,
\end{align*}
where we use \eqref{est: N_4} and the argument in \eqref{eq:energy-low} to get
\begin{align*}
\|( \mathcal{N}_u,  \mathcal{N}_\rho)\|_{X^7_\mu}^2\leq& C_0\e^4+\f{C_0}{\d^2\la}\sup_{\mu< \mu_0-\la t}\Big(h^{\eta}(t,\mu)\|(\mathcal{N}_u,\mathcal{N}_\rho)\|_{X^{8}_\mu}^2\Big)\\
\leq&C_0\e^4+\f{C_0 A}{\la^2\d^6}C_*^2\e^4\leq C_0(1+\f{C_0 A}{\la^2\d^6}C_*^2)\e^4,
\end{align*}
 and take $\f{A^\f12}{\d^3\la}\leq \f12, ~\e^2\d^{-1}\leq \f12. $ Here we take $C_{**}=2C_0(1+C_*).$

Following the process of the fourth estimate in Lemma \ref{lem: (F_pa_y u, F_pa_y v, F_pa_y rho)}, we derive
\begin{align*}
\|\pa_y^2 v^R\|_{X^7_\mu}\leq&C_0(\e+\|\pa_y\mathcal{N}_\rho\|_{X^7_\mu})+C\Big(\|\pa_y(u^R, v^R, \rho^R)\|_{X^{7}_\mu}+\e^{-1}\|(u^R, v^R, \rho^R)\|_{X^{8}_\mu}\Big),
\end{align*}
 Thus, we need to estimate $\|\pa_y\mathcal{N}_\rho\|_{X^7_\mu}.$ Note that
\begin{align*}
\pa_y\mathcal{N}_\rho=&u^R\pa_x\pa_y \rho^R+\pa_y u^R \pa_x \rho^R+v^R\pa_y^2 \rho^R+\pa_y v^R\pa_y \rho^R\\
&+\pa_y\rho^R(\pa_x u^R+\pa_y v^R)+\rho^R(\pa_x\pa_y u^R+\pa_y^2 v^R).
\end{align*}
Then we deduce from Lemma \ref{lem: product 1} that
\begin{align*}
\|\pa_y\mathcal{N}_\rho\|_{X^7_\mu}\leq&C(\|(u^R,\pa_y u^R)\|_{X^8_\mu}\|\pa_y\rho^R\|_{X^8_\mu}+\|\pa_y u^R\|_{X^8_\mu}\|(\rho^R,\pa_y\rho^R)\|_{X^8_\mu}\\
&+\|(\pa_y v^R, \pa_y^2 v^R)\|_{X^7_\mu}\|\pa_y \rho^R\|_{X^8_\mu}+\|\pa_y^2 v^R\|_{X^7_\mu}\|(\rho^R, \pa_y\rho^R)\|_{X^7_\mu})\\
\leq&C\e\|(\rho^R,\pa_y\rho^R)\|_{X^7_\mu}+C\e\|\pa_y^2 v^R\|_{X^7_\mu},
\end{align*}
which gives 
\begin{align*}
\|\pa_y^2 v^R\|_{X^7_\mu}\leq&C_0\e+C\Big(\|\pa_y(u^R, v^R, \rho^R)\|_{X^{7}_\mu}+\e^{-1}\|(u^R, v^R, \rho^R)\|_{X^{8}_\mu}\Big)+C\e\|\pa_y^2 v^R\|_{X^7_\mu}.
\end{align*}
Taking $\e$ small enough such that $C\e\leq \f12$, we get
\begin{align*}
\|\pa_y^2 v^R\|_{X^7_\mu}\leq&C_0\e+C\Big(\|\pa_y(u^R, v^R, \rho^R)\|_{X^{7}_\mu}+\e^{-1}\|(u^R, v^R, \rho^R)\|_{X^{8}_\mu}\Big)\\
\leq&(C_0+\f{A^\f12CC_*}{\d\la^\f12})\e\leq C_0(1+C_*)\e\leq\f12 C_{***}\e.
\end{align*}
So, we may take $C_{***}=2C_0(1+C_*).$

Thus, if $\e$ small enough and $\la$ large enough, we have
\begin{align*}
\sup_{t\in[0,T]}\mathcal{E}(t)\le \f{C_*^2}{2}\e^2,\quad \sup_{t\in[0,T]}\sup_{\mu<\mu_0-\la t}\|\pa_y^2u^R\|_{X^7_\mu}\leq \f{C_{**}}{2}, \quad \sup_{t\in[0,T]}\sup_{\mu<\mu_0-\la t}\|\pa_y^2v^R\|_{X^7_\mu}\leq \f{C_{***}}{2}\e. 
\end{align*}
Then the bootstrap argument ensures that there exist  time $T>0$ and $C>0$ independent of $\e$ such that 
\beno
\sup_{t\in[0, T]}\mathcal{E}(t) \leq C \e^2.
\eeno
In particular,  we have
\beno
\|\pa_x^{9}(u^R, v^R, \rho^R)\|_{L^2_{x,y}}+\|\pa_y (u^R, v^R, \rho^R)\|_{L^2_{x,y}}\leq C \e,
\eeno
which  implies
\begin{align*}
\|(u^R, v^R, \rho^R)\|_{L^\infty_{x,y}}\leq C\e.
\end{align*}

This finishes the proof of Theorem \ref{thm:main}.

\appendix

\section{Well-posedness of compressible Euler equations and compressible Prandtl equations}

 In this appendix, we prove the well-posedness of the compressible Euler equations and the compressible Prandtl equation in the analytic space. The well-posedness of the linearized Euler equations and Prandtl equation is similar, thus they are omitted.
 
 \subsection{The compressible Euler equations}

We consider the compressible Euler equations
\begin{align}\label{equ:Euler-1}
\left\{
\begin{aligned}
&\pa_t u^e+u^e\pa_x u^e+v^e\pa_y u^e  +\pa_x \rho^e =0,\\
&\pa_t v^e+u^e\pa_x v^e+v^e\pa_y v^e  +\pa_y \rho^e =0,\\
&\pa_t \rho^e+u^e\pa_x \rho^e+v^e\pa_y \rho^e+\rho^e(\pa_x u^e+\pa_y v^e )=0,\\
&v^e|_{y=0}=0,\\
&(u^e, v^e, \rho^e)|_{t=0}=(u_0, v_0, \rho_0),
\end{aligned}
\right.
\end{align}
where $t\geq 0,~x\in D_\mu,~y\in\Om_\mu$ with $\mu< 8\mu_0-\la_E T$ and $t<T<T_E$ and satisfying $8\mu_0-\la T_E>3\mu_0.$ Here $\Om_\mu$ is a complex domain given by
\begin{align*}
\Om_\mu=\Big\{y\in\mathbb{C}:  \Re y>0,~|\Im y|< \min\{\mu \Re y, \mu \}\Big\}
\end{align*}

We introduce the inner product defined by
\begin{align*}
\langle f,g \rangle_{Y^k_{\mu,t}}=\sum_{|\beta|=0}^k  \sup_{0\le\th, \th'< \mu} \int_0^t\int_{D_{\th'}} \int_{\pa\Om_\th}\widetilde{ \pa}^\beta f\overline{\widetilde{\pa}^\beta g} dy dx ds,
\end{align*}
which implies 
\begin{align}\label{inequ: Y^k}
|\langle f,g \rangle_{Y^k_{\mu,t}}|\leq&\int_0^t\|f\|_{Y^k_\mu}\|g\|_{Y^k_\mu} ds.
\end{align}
Here $\widetilde{\pa}^\beta=\pa_x^{\beta_1}(\varphi\pa_y)^{\beta_2}\pa_y^{\beta_3}\pa_t^{\beta_0}$  with $\beta_1+\beta_2+\beta_3+\beta_0=|\beta|$. For $\beta_3=0$, we denote by $\dot{Y}^k_\mu$.
\medskip

In order to recover the derivative,  we need the following lemma.

\begin{lemma}\label{lem: Euler-analytic recover}
Let $\mu'>\mu$, it holds that
\begin{align*}
\|\pa_x f\|_{Y^k_\mu}\leq \f{C_0}{\mu'-\mu}\| f\|_{Y^k_{\mu'}}, \quad \|\varphi(y)\pa_y f\|_{Y^k_\mu}\leq \f{C_0}{\mu'-\mu}\| f\|_{Y^k_{\mu'}},
\end{align*}
where $\varphi(y)=\f{y}{1+y}$. The result is also valid by replacing $Y^k_\mu$ by $\dot{Y}^k_\mu$.

\end{lemma}

In order to estimate nonlinear terms, we need the following lemma.

\begin{lemma}\label{lem: product}
Let $k\geq 8$. It holds that 
\begin{align*}
\|fg\|_{Y^k_\mu}\leq& C_0\|f\|_{Y^k_\mu}(\|g\|_{Y^{k-3}_{\mu}}+\|\pa_yg\|_{Y^{k-3}_{\mu}})+C_0\|g\|_{Y^k_\mu}(\|f\|_{Y^{k-3}_{\mu}}+\|\pa_yf\|_{Y^{k-3}_{\mu}}).
\end{align*}
The result is also valid when replacing $Y^k_\mu$ by $\dot{Y}^k_\mu$.

\end{lemma}

The proof of the above two lemmas is almost the same as Lemma \ref{lem: analyticity recovery} and Lemma \ref{lem: product 1}.
 
\medskip

 Assume that the initial data has the bound
\begin{align}\label{initial(u_0,v_0,rho_0)-Euler}
\sum_{k=0}^{20}\|\pa^\beta(\rho^e, u^e, v^e)(0,\cdot,\cdot) \|_{L^2_{8\mu_0}}^2:=M_0^2<\infty,
\end{align}
where $\|\cdot\|_{L^2_\mu}$ is given in \eqref{def: f_L^2}. It is easy to see $\|(u_0, v_0, \rho_0)\|_{Y^{20}_\mu}^2\leq C_0 M_0^2$.
\begin{proposition}\label{pro: Euler} 
Let the initial data satisfy \eqref{initial(u_0,v_0,rho_0)-Euler} and $\rho^e\geq c_0>0$. Then
there exists $T_E>0$ such that the compressible Euler equations \eqref{equ:Euler-1} has a unique solution $(u^e, v^e, \rho^e)$ in $[0,T_E]$, which satisfies
\begin{align*}
&\sup_{t\in[0, T_E]}\sup_{\mu< 8\mu_0-\la_E t}\Big(h^{\eta}(t,\mu)\|(u^e, v^e, \rho^e)\|_{Y^{20}_\mu}^2+\|(u^e, v^e, \rho^e)\|_{Y^{19}_\mu}^2\Big)\leq C_0,
\end{align*}
where $h(t,\mu)=8\mu_0-\mu-\la_E t$ and $\eta\in(0,1).$
\end{proposition}

\begin{proof}
Here we only give {\it a priori} estimates. Thanks to the relation $$\|f\|_{Y^{20}_\mu}^2=\sum_{k=0}^{20}\|\pa_y^kf\|_{\dot{Y}^{20-k}_\mu},$$ we only need to prove the following estimates for $k=0,1,\cdots,20$:
\begin{align}\label{est: Euler-inductive 1}
\sup_{\mu< 8\mu_0-\la_E t}&\Big(h^{\eta}(t,\mu)\|\pa_y^k(u^e, v^e, \rho^e)\|_{\dot{Y}^{20-k}_\mu}^2+\|\pa_y^{k}(u^e, v^e, \rho^e)\|_{\dot{Y}^{19-k}_\mu}^2\Big)\leq C_0.
\end{align}
By the boundary condition $v^e|_{y=0}=0$ and the second equation in \eqref{equ:Euler-1}, we deduce $\pa_y\rho^e|_{y=0}=0.$ Then
by a similar process in sections 3-4, we can show that \eqref{est: Euler-inductive 1} holds for $k=0$ and $k=1$. Next we prove \eqref{est: Euler-inductive 1} for $k\geq 2$ .  Assume that \eqref{est: Euler-inductive 1} holds for $k\leq\ell$. For $k=\ell+1$, since $\pa_y^{\ell+1}v^e|_{y=0}$ and $\pa_y^{\ell+1}\rho^e|_{y=0}$  are not equal to zero,  we can not use direct energy estimate for $(\pa_y^{\ell+1}v^e, \pa_y^{\ell+1}\rho^e)$. However,  we can use the equation of $(v^e, \rho^e)$ to get
\begin{align*}
\pa_y^{\ell+1}v^e=&-\pa_y^\ell\Big(\f{1}{\rho^e}(\pa_t \rho^e+u^e\pa_x \rho^e+v^e\pa_y \rho^e)+\pa_x u^e\Big)\\
\sim&\pa_t\pa_y^\ell \rho^e+\pa_y^\ell \rho^e\pa_t \rho^e+u^e\pa_x\pa_y^\ell \rho^e+\pa_y^\ell u^e\pa_x \rho^e+\f{v^e}{\varphi}\varphi\pa_y^\ell\pa_y\rho^e+\pa_x\pa_y^\ell u^e\\
\sim&(\pa_t, \pa_x, \varphi\pa_y)\pa_y^\ell(u^e, v^e, \rho^e),\\
\pa_y^{\ell+1}\rho^e=&-\pa_y^\ell\Big(\pa_t v^e+u^e\pa_x v^e+v^e\pa_y v^e\Big)\\
\sim&\pa_t\pa_y^\ell v^e+u^e\pa_x\pa^\ell_y v^e+\pa_y^\ell u^e\pa_x v^e+\f{v^e}{\varphi}\varphi\pa_y^\ell\pa_y v^e+\pa_y^\ell v^e\pa_y v^e\\
\sim&(\pa_t, \pa_x, \varphi\pa_y)\pa_y^\ell(u^e, v^e, \rho^e).
\end{align*}
Thus, the derivative $\pa_y$ turns into $(\pa_t, \pa_x, \varphi\pa_y)$. Using the product estimates in Lemma \ref{lem: product}, we get
 \begin{align}\label{est: Euler-inductive 2}
 \sup_{\mu< 8\mu_0-\la_E t}&\Big(h^{\eta}(t,\mu)\|\pa_y^{\ell+1}( v^e, \rho^e)\|_{\dot{Y}^{20-(\ell+1)}_\mu}^2+\|\pa_y^{\ell+1}(v^e, \rho^e)\|_{\dot{Y}^{19-(\ell+1)}_\mu}^2\Big)\leq C_0.
 \end{align}
To obtain the estimate for $\pa_y^{\ell+1}u^e,$ we can use the energy estimate for $\pa_y^{\ell+1}u^e$, since we don't need to use the integration by parts for the pressure term $\pa_x\rho^e.$ As a result, along with \eqref{est: Euler-inductive 2} we get
\begin{align}\label{est: Euler-inductive 3}
\sup_{\mu< 8\mu_0-\la_E t}&\Big(h^{\eta}(t,\mu)\|\pa_y^{\ell+1}u^e\|_{\dot{Y}^{20-(\ell+1)}_\mu}^2+\|\pa_y^{\ell+1}u^e\|_{\dot{Y}^{19-(\ell+1)}_\mu}^2\Big)\leq C_0.
\end{align}
Summing \eqref{est: Euler-inductive 2}-\eqref{est: Euler-inductive 3} together, \eqref{est: Euler-inductive 1} holds for $k=\ell+1$. Then by the inductive argument, we know that \eqref{est: Euler-inductive 1} holds for $k\leq 20$. 
\end{proof}

 \subsection{Compressible Prandtl equations} 
Recall that $u^p$ satisfies
\begin{align}\label{eq: (u_p^0, v_p^1)-1}
\left\{
 \begin{aligned}
&\pa_t u^p+(\overline{u^e}+u^p)\pa_x u^p+\Big(-\f{\int_0^{+\infty}\pa_x(\overline{\rho^e} u^p)dy}{\overline{\rho^e}}+v^p+z\overline{\pa_y v^e}\Big)\pa_z u^p-\f{1}{\overline{\rho^e}}\pa_z^2 u^p=0,\\
&\pa_x(\overline{\rho^e}u^p )+\pa_z(\overline{\rho^e} v^p)=0,\\
&u^p|_{z=0}=-\overline{u^e},\quad v^p|_{z=\infty}=0,\\
&u^p|_{t=0}=0,
\end{aligned}
 \right.
  \end{align}
with $v^p=\f{\int_y^{+\infty}\pa_x(\overline{\rho^e} u_p^0)dz}{\overline{\rho^e}}$ and we denote by $\overline{f}=f(t,x,0).$  Here $t\geq 0, ~x\in D_\mu,~z\in \widetilde{\widetilde{\Om}}_\mu=\big\{z\in\mathbb{C}: \Re z>0, ~|\Im z|< \mu \Re z \big\}$ with $ \mu< 3\mu_0-\la_P T,~t< T< T_P.$

%
%To simplify the notations and the proof, we only give the proof for the case $\overline{\rho^e}=1$ in \eqref{eq: (u_p^0, v_p^1)-0}. Otherwise, we may introduce new variables
%\begin{align*}
%\widetilde{z}=\sqrt{\overline{\rho^e}}z,\quad \widetilde{t}=t, \quad\widetilde{x}=x.
%\end{align*}
%Under the new variables, we have the following relationships
%\beno
%\pa_t=\pa_{\widetilde{t}} +\pa_t(\sqrt{\overline{\rho^e}}) z\pa_{\widetilde{z}},\quad 
%\pa_x=\pa_{\widetilde{x}}  +\pa_x(\sqrt{\overline{\rho^e}})z\pa_{\widetilde{z}},\quad \pa_z=\sqrt{\overline{\rho^e}}\pa_{\widetilde{z}}.
%\eeno
%Here, one key observation is that $\pa_t(\sqrt{\overline{\rho^e}}) z$ and $\pa_x(\sqrt{\overline{\rho^e}})z $ vanish on $\widetilde{z}=0.$ Moreover, let $(\overline{u}^p, \overline{v}^p) (\widetilde{t},\widetilde{x},\widetilde{z})=(u^p,v^p)(t,x,z)$, a direct calculation gives
%\begin{align*}
%\f{1}{\overline{\rho^e}}\pa_z^2 u^p=\pa_{\widetilde{z}}^2\overline{u}^p ,
%\end{align*}
%and
%\begin{align*}
%\overline{v}^p(\widetilde{t},\widetilde{x},\widetilde{z})=-\f{1}{(\overline{\rho^e})^\f32}\int_{\f{\widetilde{z}}{\sqrt{\overline{\rho^e}}}}^\infty(\pa_{\widetilde{x}}  +\pa_x(\sqrt{\overline{\rho^e}})z\pa_{\widetilde{z}})\overline{u}^p d\widetilde{z}'.
%\end{align*}

We introduce the norm
\begin{align*}
 \|f\|_{\widetilde{L}^2(0,t; W^k_\mu)}^2=\sum_{|\al|=0}^k\sup_{0\leq \th<\mu}\Big(\int_0^t\|e^{\phi}\widetilde{Z}^\al f\|_{L^2_\mu}^2 ds\Big)
\end{align*}
and the inner product
\begin{align*}
\langle f,g \rangle_{W^k_{\mu,t}}=\sum_{|\al|=0}^k  \sup_{0\le\th, \th'< \mu} \int_0^t \int_{\pa D_{\th'}}\int_{\pa\widetilde{\widetilde{\Om}}_\th}e^{2\phi} \widetilde{Z}^\al f ~ \overline{\widetilde{Z}^\al  g} dz dx ds.
\end{align*}
Here $\widetilde{Z}^\al=(\kappa\pa_x)^{\al_1}(\kappa z\pa_z)^{\al_2}(\kappa\pa_t)^{\al_0}$ with $\al_1+\al_2+\al_0=|\al|$ and $\phi(t,z)=e^{(2-\la_P t)|z|^2}.$ Here $\kappa\in(0, \f1{10}]$ is a small constant determined later. In this section, we denote $C_0$ by a constant independent of $\kappa$.

%\begin{align}\label{inequ: W^k}
%|\langle f,g \rangle_{W^k_{\mu,t}}|\leq&\int_0^t\|f\|_{W^k_\mu}\|g\|_{W^k_\mu} ds.
%\end{align}

In order to recover the derivative, we need the following lemma.

\begin{lemma}\label{lem: Prandtl-analytic recover}
Let $\mu'>\mu$, it holds that
\begin{align*}
\|\pa_x f\|_{W^k_\mu}\leq& \f{C_0}{\mu'-\mu}\| f\|_{W^k_{\mu'}},\quad \|z\pa_z f\|_{W^k_\mu}\leq \f{C_0}{\mu'-\mu}\| f\|_{W^k_{\mu'}}.
\end{align*}
\end{lemma}

In order to estimate nonlinear terms, we need the following lemma.

\begin{lemma}\label{lem: product-P}
Let $s\geq 12,$ we have
\begin{align*}
\|fg\|_{W^k_\mu}\leq& C_0\kappa^{-1}\|f\|_{W^k_\mu}(\|g\|_{W^{k-3}_{\mu}}+\|\pa_zg\|_{W^{k-3}_{\mu}})+C_0\kappa^{-1}\|g\|_{W^k_\mu}(\|f\|_{W^{k-3}_{\mu}}+\|\pa_zf\|_{W^{k-3}_{\mu}}).
\end{align*}

\end{lemma}

The proof for the above two lemmas is almost the same as Lemma \ref{lem: analyticity recovery} and Lemma \ref{lem: product 1}.

\begin{proposition}\label{pro: Prandtl}
Let $(u_e^0, v_e^0, \rho_e^0)=(u^e, v^e, \rho^e)$ be given by Proposition \ref{pro: Euler}. There exist $T_P>0$ and $\kappa_0>0$ such that for $\kappa\in(0, \kappa_0)$, the compressible Prandtl equation \eqref{eq: (u_p^0, v_p^1)-1} has a unique solution $(u_p^0, v_p^1)$ in $[0, T_P],$ which satisfies
\begin{align*}
&\sup_{t\in[0, T_P]}\sup_{\mu< 3\mu_0-\la_P t}\Big(h^{\eta}(t,\mu)\|u^p\|_{W^{18}_\mu}^2+\|u^p\|_{W^{17}_\mu}^2+\|(\pa_z u^p,v^p)\|_{W^{16}_\mu}^2\Big)\leq C_0,\\
&\sup_{t\in[0, T_P]}\sup_{\mu< 3\mu_0-\la_P t}\Big(\|\pa_z^2u^p\|_{W^{15}_\mu}^2+\|\pa_z^2v^p\|_{W^{14}_\mu}^2\Big)\leq C_0,
\end{align*}
where $h(t,\mu)=3\mu_0-\mu-\la_P t$ and $\eta\in(0,1).$
\end{proposition}

\begin{proof}
Again,  we only give {\it a priori} estimates.
We introduce a new function 
\begin{align*}
\widetilde{u^p}=u^p+e^{-2\phi(t,z)}\overline{u^e}=u^p+g,
\end{align*}
which holds $\widetilde{u^p}|_{z=0}=0.$
It is easy to deduce that $\widetilde{u}^p$ satisfies
\begin{align}\label{eq: u^p}
\pa_t \widetilde{u^p}-\f{1}{\overline{\rho^e}}\pa_z^2 \widetilde{u^p}+F^p=0,
\end{align}
where
\begin{align*}
F^p=&(\overline{u^e}+\widetilde{u^p}-g)\pa_x(\widetilde{u^p}-g)+(-\f{\int_0^{+\infty}\pa_x(\overline{\rho^e} u^p)dy}{\overline{\rho^e}}+v^p+z\overline{\pa_y v^e})\pa_z(\widetilde{u}^p-g)\\
&-\pa_t g+\f{1}{\overline{\rho^e}}\pa_z^2 g.
\end{align*}

Taking the $W^{18}_{\mu,t}$ inner product on both sides of the above equation with $\widetilde{u^p}$, we have
\begin{align}\label{est: u_p-1}
\f12 \|\widetilde{u^p}\|_{W^{18}_\mu}^2+\la_P\||z|\widetilde{u^p}\|_{\widetilde{L}^2(0,t; W^{18}_\mu)}^2-\langle \f{1}{\overline{\rho^e}}\pa_z^2\widetilde{u^p}, \widetilde{u^p} \rangle_{W^{18}_{\mu,t}}+\langle F^p, \widetilde{u^p}  \rangle_{W^{18}_{\mu,t}}=0.
\end{align}
For the dissipation term, we get by integration by parts and using $\widetilde{u^p}|_{y=0}=0$ and $\f{1}{\overline{\rho^e}}\geq c_0$  that
\begin{align*}%\label{est: u_p-2}
\nonumber
-\langle \f{1}{\overline{\rho^e}}\pa_z^2\widetilde{u^p}, \widetilde{u^p} \rangle_{W^{18}_{\mu,t}}\geq &c_0\|\pa_z \widetilde{u^p}\|_{\widetilde{L}^2(0,t; W^{18}_\mu)}^2-\Big|\sup_{\mu<3\mu_0-\la_P t}\int_0^t\sum_{|\al|=1}^{18}\langle e^{2\phi}[\widetilde{Z}^\al,\f{1}{\overline{\rho^e}}]\pa_z \widetilde{u^p},\widetilde{Z}^\al\pa_z\widetilde{u^p}\rangle_{L^2_\mu}ds\Big|\\
&-\Big|\sup_{\mu<3\mu_0-\la_P t}\int_0^t\sum_{|\al|=1}^{18}\langle [e^{2\phi}\widetilde{Z}^\al,\pa_z](\f{1}{\overline{\rho^e}}\pa_z \widetilde{u^p}),\widetilde{Z}^\al\widetilde{u^p}\rangle_{L^2_\mu}ds\Big|\\
&-\Big|\sup_{\mu<3\mu_0-\la_P t}\int_0^t\sum_{|\al|=1}^{18}\langle e^{2\phi}\widetilde{Z}^\al(\f{1}{\overline{\rho^e}}\pa_z \widetilde{u^p}),[\pa_z, \widetilde{Z}^\al]\widetilde{u^p}\rangle_{L^2_\mu}ds\Big|\\
\nonumber
\geq& (c_0-\kappa)\|\pa_z \widetilde{u^p}\|_{\widetilde{L}^2(0,t; W^{18}_\mu)}^2-C\||z|\widetilde{u^p}\|_{\widetilde{L}^2(0,t; W^{18}_\mu)}^2,
\end{align*}
by using the fact
\begin{align*}
[\pa_z, e^{2\phi}\widetilde{Z}^\al]f=[\pa_z,e^{2\phi}]\widetilde{Z}^\al f+e^{2\phi}[\pa_z, \widetilde{Z}^\al]f\sim ze^{2\phi}\widetilde{Z}^\al f+\kappa e^{2\phi}\widetilde{Z}^{\al-e_2} \pa_zf,
\end{align*}
 and the argument used in $D_1^1$ in Proposition \ref{pro: Zu^R-H}.
 
For another term, we get by Lemma \ref{lem: product-P} that
\begin{align*}%\label{est: u_p-3}
&\big|\langle F^p, \widetilde{u^p}  \rangle_{W^{18}_{\mu,t}}\big|\\
\nonumber
&\leq C_0\sup_{s\in[0,t]}\Big(1+\kappa^{-1}\|\widetilde{u^p}\|_{W^{16}_\mu}+\kappa^{-1}\|\pa_z\widetilde{u^p} \|_{W^{16}_\mu}\Big)\int_0^t (1+\|\widetilde{u^p}\|_{W^{18}_\mu}+\|(\pa_x,z\pa_z)\widetilde{u^p}\|_{W^{18}_\mu})\|\widetilde{u^p}\|_{W^{18}_\mu}ds\\
\nonumber
&\leq C_0\la_P^{-1}h^{-\eta}(t,\mu)\sup_{s\in[0,t]}\Big(1+\kappa^{-1}\|\widetilde{u^p}\|_{W^{16}_\mu}+\kappa^{-1}\|\pa_z\widetilde{u^p} \|_{W^{16}_\mu}\Big)\Big(1+\sup_{s\in[0,t]}\sup_{\mu< 3\mu_0-\la_p s}\Big(h^{\eta}(s,\mu)\|\widetilde{u^p}\|_{W^{18}_\mu}^2\Big)\Big).
\end{align*}
Putting the above two estimates into \eqref{est: u_p-1} and multiplying $h^{\eta}(t,\mu)$ on both sides, we take $\sup_{t\in[0, T_p]}\sup_{\mu< 3\mu_0-\la_P t}$ and take $\la_P$ large enough such that $\la_P\geq 2C$ to obtain
\begin{align}\label{est: u_p-4}
&\sup_{t\in[0, T_p]}\sup_{\mu< 3\mu_0-\la_p t}(h^{\eta}(t,\mu)\|\widetilde{u^p}\|_{W^{18}_\mu}^2)+(2c_0-2\d)\sup_{t\in[0, T_p]}\sup_{\mu< 3\mu_0-\la_P t}\Big(h^{\eta}(t,\mu)\|\pa_z \widetilde{u^p}\|_{\widetilde{L}^2(0,t; W^{18}_\mu)}^2\Big)\\
\nonumber
&\leq C_0\la_P^{-1}\sup_{s\in[0,t]}\Big(1+\kappa^{-1}\|\widetilde{u^p}\|_{W^{16}_\mu}+\kappa^{-1}\|\pa_z\widetilde{u^p} \|_{W^{16}_\mu}\Big)\Big(1+\sup_{s\in[0,t]}\sup_{\mu< 3\mu_0-\la_P s}\Big(h^{\eta}(s,\mu)\|\widetilde{u^p}\|_{W^{18}_\mu}^2\Big)\Big).
\end{align}

Acting $\pa_z$ on \eqref{eq: u^p}, we get
\begin{align}\label{eq: pa_zu^p}
\pa_t \pa_z\widetilde{u^p}+\pa_z(-\f{1}{\overline{\rho^e}}\pa_z^2 \widetilde{u^p}+F^p)=0,
\end{align}
with boundary condition $(-\f{1}{\overline{\rho^e}}\pa_z^2 \widetilde{u^p}+F^p)|_{z=0}=0.$ Taking the $W^{17}_{\mu,t}$ inner product on both sides of the above equation with $\pa_z\widetilde{u^p}$, we have
\begin{align}\label{est: pa_zu_p-1}
\f12 \|\pa_z\widetilde{u^p}\|_{W^{17}_\mu}^2+\la_P\||z|\pa_z\widetilde{u^p}\|_{\widetilde{L}^2(0,t; W^{17}_\mu)}^2+\Big\langle \pa_z(-\f{1}{\overline{\rho^e}}\pa_z^2\widetilde{u^p}+F^p), \pa_z\widetilde{u^p}  \Big\rangle_{W^{17}_{\mu,t}}=0.
\end{align}
By integrating by parts and using a similar process as above, we have
\begin{align*}
&\Big\langle \pa_z(-\f{1}{\overline{\rho^e}}\pa_z^2\widetilde{u^p}+F^p), \pa_z\widetilde{u^p}  \Big\rangle_{W^{17}_{\mu,t}}\\
\geq&(\f 45c_0-\kappa)\|\pa_z^2 \widetilde{u^p}\|_{\widetilde{L}^2(0,t; W^{17}_\mu)}^2-C\||z|\pa_z\widetilde{u^p}\|_{\widetilde{L}^2(0,t; W^{17}_\mu)}^2-C_0\|F^p\|_{\widetilde{L}^2(0,t; W^{17}_\mu)}^2.
\end{align*}
 By Lemma \ref{lem: product-P}, it is easy to prove that 
 \begin{align*}
 \|F^p\|_{\widetilde{L}^2(0,t; W^{17}_\mu)}^2\leq& \int_0^t \|F_p\|_{W^{17}_\mu}^2ds\\
 \leq&C_0\sup_{s\in[0,t]}\Big(1+\kappa^{-2}\|\widetilde{u^p}\|_{W^{16}_\mu}^2+\kappa^{-2}\|\pa_z\widetilde{u^p} \|_{W^{16}_\mu}^2\Big)\int_0^t (1+\|\widetilde{u^p}\|_{W^{18}_\mu}^2)ds\\
 \leq&C_0\la_P^{-1}\sup_{s\in[0,t]}\Big(1+\kappa^{-2}\|\widetilde{u^p}\|_{W^{16}_\mu}^2+\kappa^{-2}\|\pa_z\widetilde{u^p} \|_{W^{16}_\mu}^2\Big)\Big(1+\sup_{s\in[0,t]}\sup_{\mu< 3\mu_0-\la_P s}\Big(h^{\eta}(s,\mu)\|\widetilde{u^p}\|_{W^{18}_\mu}^2\Big)\Big).
 \end{align*}
Putting the above estimates into \eqref{est: pa_zu_p-1},  we arrive at
\begin{align}\label{est: pa_z u_p-1}
&\sup_{t\in[0, T_P]}\sup_{\mu< 3\mu_0-\la_P t}(h^{\eta}(t,\mu)\|\pa_z\widetilde{u^p}\|_{W^{17}_\mu}^2)+(2c_0-2\d)\sup_{t\in[0, T_P]}\sup_{\mu< 3\mu_0-\la_P t}\Big(h^{\eta}(t,\mu)\|\pa_z^2 \widetilde{u^p}\|_{\widetilde{L}^2(0,t; W^{17}_\mu)}^2\Big)\\
&\leq C_0\la_P^{-1}\sup_{s\in[0,t]}\Big(1+\kappa^{-2}\|\widetilde{u^p}\|_{W^{16}_\mu}^2+\kappa^{-2}\|\pa_z\widetilde{u^p} \|_{W^{16}_\mu}^2\Big)\Big(1+\sup_{s\in[0,t]}\sup_{\mu< 3\mu_0-\la_P s}\Big(h^{\eta}(s,\mu)\|\widetilde{u^p}\|_{W^{18}_\mu}^2\Big)\Big),\nonumber
\end{align}
by taking $\la_P$ large enough.

As a result,  combining \eqref{est: u_p-4} and \eqref{est: pa_z u_p-1} and using the argument in \eqref{est: f_X^k-0}, we conclude that
\begin{align*}
\sup_{t\in[0, T_P]}\sup_{\mu< 3\mu_0-\la_P t}\Big(h^{\eta}(t,\mu)\|\widetilde{u^p}\|_{W^{18}_\mu}^2+\|\widetilde{u^p}\|_{W^{17}_\mu}^2+\|\pa_z \widetilde{u^p}\|_{X^{17}_\mu}\Big)\leq C_0,
\end{align*}
which along with $v^p=\f{\int_y^{+\infty}\pa_x(\overline{\rho^e} u_p^0)dz}{\overline{\rho^e}}$ and $u^p=\widetilde{u^p}-e^{-2\phi(t,z)}\overline{u^e}$ implies the first estimate. The second one can be deduced by using the equation of $u^p.$
\end{proof}

\subsection{Proof of Lemma \ref{lem: app}}
First of all, Proposition \ref{pro: Euler} gives the existence of the solution $(u_e^0, v_e^0)$ of \eqref{equ:Euler}  with the bound
\begin{align*}
\|(u_e^0, v_e^0)\|_{Y^{19}_\mu}\leq C_0.
\end{align*}
With $(u_e^0, v_e^0)$, Proposition \ref{pro: Prandtl} gives the existence of the solution$(u_p^0, v_p^1)$ of \eqref{eq: (u_p^0, v_p^1)-0} with the bound
\begin{align*}
\|u_p^0\|_{Y^{17}_\mu}+\|v_p^1\|_{Y^{16}_\mu}\leq C_0.
\end{align*}
Next we can solve the linearized Euler equation \eqref{eq: (rho_e^1, u_e^1, v_e^1)}  in $Y^{16}_\mu$ . Finally,
we solve the linearized Prandtl equation \eqref{eq: (u_p^1, v_p^2)} in $W^{14}_\mu$. Then Lemma \ref{lem: app} follows easily.

\section* {Acknowledgments.}
Chao Wang is partially supported by the NSF of China under Grant 12071008. Yuxi Wang is partially supported by the NSF of China under Grant 12101431. Zhifei Zhang is partially supported by the NSF of China under Grant 12171010 and 12288101.


\begin{thebibliography}{99}

 \bibitem{CWZ} Q. Chen, D. Wu and Z. Zhang, {\it On the $L^\infty$ stability of Prandtl expansions in Gevrey class.}  Sci. China Math. 65 (2022), no. 12, 2521–2562.
 
\bibitem{FTZ} M. Fei, T. Tao and Z. Zhang, {\it On the zero-viscosity limit of the Navier-Stokes equations in $R^3_+$ without analyticity.} J. Math. Pures Appl.,  {112}(2018), 170-229.

\bibitem{FN} E. Feireisl and A. Novotny, {\it Inviscid incompressible limits of the full Navier-Stokes-Fourier system.} Comm. Math. Phys., 321 (2013), 605-628.


\bibitem{GVMM1} D. Gerard-Varet, Y. Maekawa and N. Masmoudi, {\it Gevrey stability of Prandtl expansions for 2-dimensional Navier-Stokes flows. } Duke Math. J.,  167 (2018),  2531--2631.


\bibitem{GVMM2} D. Gerard-Varet, Y. Maekawa and N. Masmoudi, {\it Optimal Prandtl expansion around a concave boundary layer}, arXiv: 2005.05022. 

\bibitem{GGT} E. Grenier, Y. Guo and T. Nguyen, {\it Spectral instability of characteristic boundary layer flows}. Duke Math. J., 165 (2016), 3085-3146.

\bibitem{GMWZ}O. Gues, G. Metivier, M. Williams and K. Zumbrun, {\it Existence and stability of noncharacteristic boundary layers for the compressible Navier-Stokes and viscous MHD equations.}  Arch. Ration. Mech. Anal.,  197(2010), 1-87.

\bibitem{IP} D. Iftimie and G. Planas, {\it Inviscid limits for the Navier-Stokes equations with Navier friction boundary conditions.} Nonlinearity, 19(2006),  899-918.

\bibitem{IS} D. Iftimie and F. Sueur, {\it Viscous boundary layers for the Navier-Stokes equations with the Navier slip conditions.} Arch. Ration. Mech. Anal., 199 (2011), 145--175.

\bibitem{Kato} T. Kato, {\it Remarks on zero viscosity limit for nonstationary Navier-Stokes flows with boundary}, Seminar on nonlinear partial differential equations (Berkeley, Calif., 1983), 85-98. Mathematical Sciences Research Institute Publications, 2. Springer, New York, 1984.

\bibitem{KVW} I. Kukavica, V. Vicol and F. Wang, {\it The inviscid limit for the Navier-Stokes equations with data analytic only near the boundary.} Arch. Ration. Mech. Anal.,  237 (2020), 779--827.

\bibitem{LW} C. Liu and Y.-G. Wang, {\it Stability of boundary layers for the nonisentropic compressible circularly symmetric 2D flow}. SIAM J. Math. Anal.,  {46}(2014), 256-309.

\bibitem{Mae} Y. Maekawa, {\it On the inviscid limit problem of the vorticity equations for viscous incompressible flows in the half-plane. } Comm. Pure Appl. Math., 67(2014), 1045--1128.

\bibitem{MWWZ} N. Masmoudi, Y. Wang, D. Wu and Z. Zhang, {\it Tollmien-Schlichting waves in the subsonic regime. } arXiv:2305.03229. 
%\bibitem{MR1} N. Masmoudi and F. Rousset, {\it Uniform regularity for the Navier-Stokes equation with Navier boundary condition.} Arch. Ration. Mech. Anal.,  203 (2012), 529--575. 
%
%\bibitem{MR2} N. Masmoudi and F. Rousset, {\it Uniform regularity and vanishing viscosity limit for the free surface Navier-Stokes equations}.  Arch. Ration. Mech. Anal.,  223 (2017),  301--417. 

\bibitem{NN} T. Nguyen and  T. Nguyen, {\it The inviscid limit of Navier-Stokes equations for analytic data on the half-space.} Arch. Ration. Mech. Anal.,  230 (2018), 1103--1129. 

 
\bibitem{Paddick} M. Paddick, {\it The strong inviscid limit of the isentropic compressible Navier-Stokes equations with Navier boundary conditions.} Discrete Contin. Dyn. Syst., 36 (2016), 2673--2709.


\bibitem{Prandtl} L. Prandtl, {\it Uber fl\"{u}ssigkeits-bewegung bei sehr kleiner reibung.} Actes du 3me Congr$\acute{e}$s
international dse Math$\acute{e}$maticiens, Heidelberg. Teubner,leipzig, 1904,  484--491.

\bibitem{Rousset}  F. Rousset, {\it Characteristic boundary layers in real vanishing viscosity limits.} J. Differential Equations, 210(2005), 25--64.


\bibitem{SC2} M. Sammartino and R. E. Caflisch, {\it Zero viscosity limit for analytic solutions of the Navier-Stokes equation on a half-space. II.  Construction of the Navier-Stokes solution.}  Comm. Math. Phys., 192(1998), 463-491.

\bibitem{Sueur}  F. Sueur, {\it On the inviscid limit for the compressible Navier-Stokes system in an impermeable bounded domain.} J. Math. Fluid Mech., 16 (2014),  163--178. 


\bibitem{WW}  C. Wang and Y. Wang, {\it Zero-viscosity limit of the Navier-Stokes equations in a simply-connected bounded domain under the analytic setting.} J. Math. Fluid Mech.,  22(2020), Paper No. 8, 58 pp.

\bibitem{WWZ}  C. Wang, Y. Wang and Z. Zhang, {\it Zero-viscosity limit of the Navier-Stokes equations in the analytic setting.} Arch. Ration. Mech. Anal.,  224(2017), 555--595.

\bibitem{WangY} Y. Wang,  {\it Uniform regularity and vanishing dissipation limit for the full compressible Navier-Stokes system in three dimensional bounded domain.} Arch. Ration. Mech. Anal.,  221 (2016), 1345-1415.

\bibitem{WWi}Y.-G. Wang and M. Williams, {\it The inviscid limit and stability of characteristic boundary layers for the compressible Navier-Stokes equations with Navier-friction boundary conditions.} Ann. Inst. Fourier, 62 (2012), 2257--2314.

\bibitem{WXY}Y. Wang, Z. Xin and  Y. Yong, {\it Uniform regularity and vanishing viscosity limit for the compressible Navier-Stokes with general Navier-Slip boundary conditions in three-dimensional domains}.  SIAM J. Math. Anal.,  {47} (2015), 4123--4191.

\bibitem{WZ1} Y.-G. Wang and S. Zhu, {\it On the vanishing dissipation limit for the full Navier-Stokes-Fourier system with non-slip condition.} J. Math. Fluid Mech.,  20 (2018), 393--419. 

\bibitem{WZ2}  Y.-G. Wang and S. Zhu, {\it On the inviscid limit for the compressible Navier-Stokes system with no-slip boundary condition.} Quart. Appl. Math.,  76 (2018), 499--514.

\bibitem{XY} Z. Xin and T. Yanagisawa, {\it Zero-viscosity limit of the linearized Navier-Stokes equations for a compressible viscous fluid in the half-plane.}  Comm. Pure Appl. Math.,  52(1999), 479--541.

\bibitem{YZ} T. Yang and Z. Zhang, {\it Linear instability analysis on compressible Navier-Stokes equations with strong boundary layer.}  arXiv: 2203.17195.

\end{thebibliography}
\end{document}